\documentclass[dvips,11pt,preprint]{imsart}

\RequirePackage[OT1]{fontenc}
\RequirePackage{amsthm,amsmath}
\RequirePackage[colorlinks,citecolor=blue,urlcolor=blue]{hyperref}

\usepackage[english]{babel}
\usepackage{graphicx}
\usepackage{amsfonts}
\usepackage{amssymb}    
\usepackage{url}
\usepackage{epsfig}
\usepackage{color}
\usepackage{vmargin}

\setmarginsrb{3.7cm}{1.5cm}{3.7cm}{3.5cm}{1cm}{1cm}{1cm}{1cm}

\def\cov{\mbox{cov}}
\def\var{\mbox{Var}}

\def\DKhi{\mathrm{DKhi}}
\def\EDKhi{\mathrm{EDKhi}}

\def\X{\mathbf{X}}

\def\1{{\mathbf 1}}

\def\R{{\mathbb{R}}}

\def\E{{\mathbb{E}}}
\def\P{{\mathbb{P}}}

\newcommand{\eref}[1]{(\ref{#1})}
\newcommand{\pa}[1]{\left({#1}\right)}

\newcommand{\cro}[1]{\left[{#1}\right]}

\newcommand{\ac}[1]{\left\{{#1}\right\}}

\newtheorem{thrm}{Theorem}[section]
\newtheorem{prte}[thrm]{Proposition}
\newtheorem{lemma}[thrm]{Lemma}

\newcommand{\degr}{\mathrm{deg}} 
\newcommand{\pen}{\mathrm{pen}}
\newcommand{\LA}{\mathrm{LA}}
\newcommand{\EW}{\mathrm{EW}}
\newcommand{\QE}{\mathrm{QE}}
\newcommand{\CO}{\mathrm{C0}}
\newcommand{\nei}{\mathrm{ne}}
\newcommand{\crit}{\mathrm{Crit}}
\newcommand{\supp}{\mathrm{supp}}
\newcommand{\MSEP}{\mathrm{MSEP}}

\def\B{{\mathbb{B}}}
\def\G{\mathcal{G}}
\def\thetat{\tilde \theta}
\def\eps{\boldsymbol{\epsilon}}

\begin{document}
\begin{frontmatter}
\title{Graph selection with GGMselect}
\runtitle{GGMselect}

\begin{aug}
\author{\fnms{Christophe} \snm{Giraud}\ead[label=e1]{Christophe.Giraud@polytechnique.edu}},
\author{\fnms{Sylvie} \snm{Huet}\ead[label=e2]{Sylvie.Huet@jouy.inra.fr}}
\and
\author{\fnms{Nicolas} \snm{Verzelen}\ead[label=e3]{nicolas.verzelen@supagro.inra.fr}}

\runauthor{Giraud,  Huet and  Verzelen}

\end{aug}

\begin{abstract}
Applications on inference of biological networks have raised a strong interest in the problem 
of graph estimation in high-dimensional Gaussian graphical models. To handle this problem, we 
propose a two-stage procedure which first builds a family of candidate graphs from the data, 
and then selects one graph among this family according to a dedicated criterion. This estimation 
procedure is shown to be consistent in a high-dimensional setting, and its risk is controlled by a  
non-asymptotic oracle-like inequality. The procedure is tested on a real data set concerning gene 
expression data, and its performances are assessed on the basis of a large numerical study. 
The procedure is implemented in the R-package GGMselect available on the CRAN. 
\end{abstract}

	

\end{frontmatter}

\section{Introduction}
Biotechnological
developments in proteomics or transcriptomics enable to produce a
huge amount of  data. One of the challenges
for the statistician  is to infer from these data the regulation network of a
family of  genes (or proteins). The task is difficult due to the very
high-dimensional nature of the data and
the small sample size. For example, microarrays measure the expression
levels of a few thousand genes and the sample size $n$
is no more than a few tens. 
When no additional information is available,
the Gaussian graphical modeling, denoted GGM, has been proposed as a tool to
handle this
issue, see e.g. \cite{KW,Detal,WY}.  
 Graphical modeling is based on the notion of
   conditional dependency.
The principle underlying the GGM
    approach is the following : the existence of a regulation
    dependence between two genes corresponds to the
    existence of  a conditional dependence between their gene
    expression levels.
The conditional dependences between the gene expression levels
are represented by a graph $G$, where each node represents a gene and where an edge is set 
between two nodes $a$ and $b$ if there exists a conditional dependence between  their gene
    expression levels. According to the
  GGM principle, this graph $G$ coincides with the gene regulation network.

Let us describe more precisely the GGM setting. The gene expression levels  $(X_1 , . . . , X_p)$ of $p$ genes are
modeled
by a centered Gaussian law with covariance matrix $\Sigma$, denoted $\P_\Sigma$.
This
law $\P_\Sigma$ is a so-called graphical model according to a graph $G$,  if for
any genes $a$ and $b$ that are not
neighbours in $G$, the variables $X_{a}$ and $X_{b}$ are independent
conditionally on the remaining variables. Roughly
  speaking,  if genes $a$ and $b$  are not
neighbours in $G$, the variables $X_{a}$ and $X_{b}$ are
uncorrelated when the values of the remaining variables are fixed. There exists a unique graph
$G_{\Sigma}$ which is  minimal for the inclusion and  such that $\P_{\Sigma}$ is
a graphical model according to $G_{\Sigma}$. An edge between  $a$ and $b$ in $G_{\Sigma}$ therefore
represents the existence of a conditional dependence between the variables $X_{a}$ and $X_{b}$. As mentioned before, 
$G_{\Sigma}$ is a graph of interest since it shall coincide with the gene regulation network.
Our aim in this paper is to estimate this graph
from microarrays data
which are assumed to be a $n$-sample of  the law $\P_{\Sigma}$. We will pay a
special attention to the case
where $n<p$ and we  assume in the following that $\Sigma$ is
non-singular.


The problem of graph estimation in Gaussian graphical model when the sample size
$n$ is smaller (or much smaller) than the number $p$ of variables  is a current 
active field of research in
statistics. Many estimation procedures have been proposed  recently to perform
graph
estimation in Gaussian graphical model when $n<p$. A first class of procedures
is based on multiple testing on empirical partial covariance. If
$G_{\Sigma}$ denotes the (minimal) graph of the law $\P_{\Sigma}$, there is an
edge in $G_{\Sigma}$ between $a$ and $b$, if and only if the conditional
covariance of $X_{a}$ and $X_{b}$ given all the other variables is non-zero.
When $n<p$, the empirical version of the latter conditional covariance cannot be
computed, so several papers suggest to use instead the empirical conditional
covariance of $X_{a}$ and $X_{b}$ given $\ac{X_{s},\ s\in S}$ for some subsets
$S$ of $\ac{1,\ldots,p}\setminus\ac{a,b}$ with cardinality less than $n-2$. A
multiple testing procedure is then applied to detect if the conditional 
covariance $\cov(X_{a},X_{b} | X_{s},\ s\in S)$ is non-zero. Wille and B\"uhlmann~\cite{WB06}
restrict to the sets $S$ of cardinality less or equal to
one, Castelo and Roverato~\cite{CR06} consider the sets $S$ with cardinality at most $q$ (for
some fixed $q$) and Spirtes et al.~\cite{SGS00} (see also \cite{KB08}) propose a
procedure which avoid an exhaustive search
over all $S$.
A second class of procedures relies on the fact that the entries $\Omega_{a,b}$
of the inverse covariance matrix $\Omega=\Sigma^{-1}$ are non-zero if and only
if there is an edge between $a$ and $b$ in $G_{\Sigma}$. Several papers then
suggest to perform a sparse estimation of $\Omega$ in order to estimate the
graph $G_{\Sigma}$, see \cite{HLPL,YL,BGA,FHT,Fan08}. They propose to maximize the
log-likelihood of
$\Omega$ under $l^1$ constraints to enforce sparsity and they design
optimization algorithms to perform this maximization.
A third class of procedures uses the fact that the coefficients
$\theta_{a,b}$ of the regression of $X_{a}$ on $\ac{X_{b},\ b\neq a}$ are
non-zeros if and only if there is an edge between $a$ and $b$ in $G_{\Sigma}$.
Meinshausen and B\"uhlmann~\cite{MB06} and Rocha et al.~\cite{RZY09}
perform regressions with $l^1$ constraints, whereas Giraud~\cite{giraud08} (see
also \cite{Verzelen08}) proposes an exhaustive search over the set of
sparse graphs to obtain a
sparse estimate of the matrix $\theta$ and then detect the graph $G_{\Sigma}$.
Finally, a series of papers (e.g. \cite{wongcarter02,dellaportas2003,scott}) investigate 
a Bayesian approach to estimate the graph.

In this paper, we propose a new estimation scheme which combines the good
properties of different procedures. On the one hand, the procedures based
on the
empirical covariance or on $l^1$ regularisation share some nice computational
properties and they can handle several hundred of variables $X_{1},\ldots,
X_{p}$. Nevertheless, the theoretical results assessing their statistical
accuracy are either of asymptotic nature or rely on strong assumptions on the
covariance \cite{fan_covariance,rothman08}. Moreover, their performance heavily
depends on one (or several) tuning parameter, which is usually not dimensionless
and whose optimal value is unknown. To cope with this
issue, many authors propose to apply cross-validation or the BIC
criterion. However, the BIC criterion often overfits in a high
dimensional setting (see \cite{BGH09} and the simulations Section
\ref{section_simulations}) and cross-validation offers little theoretical
warranty.
On the other hand, the method proposed by Giraud~\cite{giraud08} has a good
statistical accuracy and strong theoretical
results have been established, but its computational complexity 
 is huge and it cannot be performed when the number $p$ of
variables is larger than a few tens. \\

Our strategy here is to build a data-driven
family of candidate graphs using several fast above-mentioned procedures and
then to apply the selection criterion presented in~\cite{giraud08} to select
one
graph among them.  We show that this criterion  can be used both
\begin{itemize}
\item[(i)] to choose the tuning parameter(s) of any estimation procedure with no
need of any additional knowledge. As such, our criterion is an
alternative to the BIC criterion and cross-validation.
\item[(ii)] and to compare several graphs produced by various estimation procedures,
including graphs built from a priori knowledge.
\end{itemize}

Our estimation procedure can handle several hundred
of variables $X_{1},\ldots,X_{p}$ and presents good statistical properties.  
It is proved to be consistent in a high-dimensional setting.  Furthermore, its risk is controlled by a non-asymptotic oracle-like
inequality. This means that the  risk of our estimator is almost as small as if we knew in advance the best graph in the data-driven family. A more formalized definition of the oracle property is given in Section \ref{section_oracle}.
 In contrast to other results in the
literature~\cite{giraud08,MB06}, this oracle inequality allows to deal with a
data-driven collection of graphs.
In addition, we propose families of candidate graphs which work well in practice
as shown  on simulated examples. 
Finally, the  procedure is implemented in the $R$-package {\it GGMselect}  
available on 
the {\it Comprehensive R Archive Network}. \url{http://cran.r-project.org/}

\medskip

The remaining of the paper is organized as follows. We describe the estimation
procedure in the next section and state some theoretical results on its
statistical accuracy in Section~\ref{section_theorique}. In Section~\ref{section_simulations}, we carry out some numerical
experiments in order to assess the
  performances of our procedure. In Section~\ref{section_ex}, we test
  our method on a real data set concerning gene expression data
  provided in~\cite{Hess} and already analyzed in~\cite{Ambroise}.  Section~\ref{section_preuves} is devoted to the proofs. Details on the collections of graphs are postponed to Section~\ref{section_family_details}.\\

\medskip

\noindent 
\textit{Notations.} To estimate the graph $G_{\Sigma}$, we will start
from a $n$-sample $X^{(1)},\ldots,X^{(n)}$ of the law $\P_{\Sigma}$. We denote
by $\X$ the $n\times p$ matrix whose rows are given by the vectors $X^{(i)}$,
namely $\X_{i,a}=X^{(i)}_{a}$ for $i=1,\dots,n$ and $a=1,\ldots,p$. We write
$\X_{a}$ for the $a^{\textrm{th}}$ column of $\X$. 
We also set $\Gamma=\ac{1,\ldots,p}$ and for any graph $G$ with nodes indexed by
$\Gamma$, we write $d_{a}(G)$ for the degree of the node $a$ in the graph $G$
(which is the number of edges incident to $a$) and $\degr(G)=\max_{a\in
\Gamma}d_{a}(G)$ for the degree of $G$. Moreover, the notation
$a\stackrel{G}{\sim}b$ means that the nodes $a$ and $b$ are neighbours in the
graph $G$.
Finally, we write $\Theta$ for the set  of $p\times p$ matrices with 0 on the
diagonal, $\|\cdot\|_{q\times p}$ for the Frobenius norm on $q\times p$ matrices
$$\|A\|^2_{q\times p}=\textrm{Tr}(A^TA)=\sum_{i=1}^q\sum_{j=1}^pA_{i,j}^2\ ,$$
$\|\cdot\|_{n}$ for the Euclidean norm on $\R^n$ divided by $\sqrt n$, and for
any $\beta\in\R^p$ we define supp$(\beta)$ as the set of the labels
$a\in\Gamma$ such that $\beta_{a}\neq 0$.

\section{Estimation procedure}\label{procedure}
GGMselect is a two-stage estimation procedure which first builds a data-driven
family $\widehat \G$ of candidate graphs and then applies a selection procedure
to pick one graph among these. We present the selection procedure in the next
paragraph and then describe different possible choices for the family of
candidate graphs $\widehat \G$.

\subsection{\label{SelProc.st}Selection procedure} 
We assume here that we have at hand a family $\widehat \G$ of candidate graphs,
which all have a degree smaller than $n-2$. 
To select a graph $\widehat G$ among the family $\widehat\G$, we use the
selection criterion introduced in~\cite{giraud08}. 
We write $\theta$ for the $p\times p$ matrix such that 
$$\E_{\Sigma}\cro{X_{a}\big| X_{b},\ b\neq a}=\sum_{b\neq
a}\theta_{a,b}X_{b}\quad\textrm{and}\quad\theta_{a,a}=0\quad \textrm{for all } a
\in \ac{1,\ldots,p}.$$
The matrix $\theta$ minimizes $\|\Sigma^{1/2}(I-\theta')\|_{p\times p}$ over the
set $\Theta$ of $p\times p$ matrices $\theta'$ with 0 on the diagonal. Since
$\X^T\X/n$ is an empirical version of $\Sigma$, an empirical version of
$\|\Sigma^{1/2}(I-\theta)\|_{p\times p}$ is $\|\X(I-\theta)\|_{n\times p}$
divided by $\sqrt{n}$. Therefore, for any graph $G$ in $\widehat\G$, we
associate an estimator $\widehat \theta_{G}$ of $\theta$ by setting
 \begin{equation} \label{thetahat}
 \widehat\theta_{G}=\textrm{argmin}\ac{\|\X(I-\theta')\|_{n\times
p}:\theta'\in\Theta_{G}},
 \end{equation}
 where $\Theta_{G}$ is the set of $p\times p$ matrices $\theta'$ such that
$\theta'_{a,b}$ is non-zero if and only if there is an edge between $a$ and $b$
in $G$.
 
 Finally, we select a graph $\widehat G$ in $\widehat \G$ by taking any
minimizer over $\widehat \G$ of the criterion 
 \begin{equation}\label{definition_critere}
\crit(G)= \sum_{a=1}^p\left[\|{\bf X}_a-{\bf X}[\widehat{\theta}_{G}]_{a}
\|_n^2\left(1+\frac{\pen[d_{a}(G)]}{n-d_{a}(G)}\right)\right]\ ,
\end{equation}
where $d_{a}(G)$ is the degree of the node $a$ in the graph $G$ and  the penalty
function $\pen: \mathbb{N}\rightarrow \mathbb{R}^+$ is of the form of the
penalties introduced in Baraud et al.~\cite{BGH09} for the fixed design
regression model.
 To compute this penalty, we define for any integers $d$ and $N$ the
$\text{DKhi}$ function by 
\begin{eqnarray*}
 \text{DKhi}(d,N,x)=\mathbb{P}\left(F_{d+2,N}\geq \frac{x}{d+2}\right)-
\frac{x}{d}\,\mathbb{P}\left(F_{d,N+2}\geq \frac{N+2}{Nd}x\right) , \, x>0\,,
\end{eqnarray*}
where $F_{d,N}$ denotes a Fisher random variable with $d$ and $N$ degrees of
freedom. The function $x\mapsto \text{DKhi}(d,N,x)$ is decreasing and we write
$\text{EDKhi}[d,N,x]$ for its inverse, see \cite{BGH09} Sect. 6.1 for more
details. Then, we fix some constant $K>1$ and set       
\begin{equation}\label{definition_penalite}
 \pen (d) =
K\,\frac{n-d}{n-d-1}\,\text{EDKhi}\left[d+1,n-d-1,\left(\binom{p-1}{d}
(d+1)^2\right)^{-1}\right].
\end{equation}
When $d$ remains small compared to $n$, the penalty function increases
approximately linearly with $d$. Actually, when  $ d\leq \gamma \, n/
\pa{2\pa{1.1+\sqrt{\log p}}^2}$ for some $\gamma<1,$
we approximately have  for large values of $p$ and $n$
$$\pen(d)\lesssim K \pa{1+e^{\gamma}\sqrt{2\log p}}^2(d+1),$$
see Proposition~4 in \cite{BGH09} for an exact bound.

The selection procedure depends on a dimensionless tuning parameter $K$. 
A larger value for $K$ yields a procedure more conservative. In theory (and in
practice) $K$ has to be larger than one. In our simulations, we set $K=2.5$.

\subsection{\label{FamGr.st}Family $\widehat\G$ of candidate graphs}
The computational complexity of the minimization of the
criterion~(\ref{definition_critere}) over the family $\widehat \G$ is linear
with respect to its size. In particular, minimizing~(\ref{definition_critere})
over all the graphs with degree smaller than some integer $D$, as proposed
in~\cite{giraud08}, is intractable when $p$ is larger than a few tens. To
overcome this issue, we propose to build a much smaller (data-driven) family
$\widehat\G$ of candidate graphs, with the help of various fast algorithms
dedicated to graph estimation.

 Since the procedure applies for any family $\widehat\G$, GGMselect 
allows to select the tuning parameter(s) of any graph estimation procedure and
also to compare any collection of
estimation procedures. Nevertheless, we advise in practice to choose one of the
four
families  of candidate graphs $\widehat{\mathcal{G}}_{\EW}$,
$\widehat{\mathcal{G}}_{\CO 1}$,
$\widehat{\mathcal{G}}_{\LA}$, and $\widehat{\mathcal{G}}_{\QE}$ presented
below, or the union
of them.
These families have been chosen on the
basis of theoretical results and simulation studies.

 In the following, we explain how to tune and compare graph
estimation procedures with GGMselect.
Afterwards, we describe the four above-mentioned
families, provide algorithms to
compute them efficiently, and discuss their computational complexity and their
size. Each family depends on an integer $D$, smaller than $n-2$, which
corresponds to the maximal degree of the graphs in this family.

\subsubsection{Tuning a procedure and comparing several ones}
Suppose we are given an estimation procedure $P$
depending on a tuning parameter $\lambda>0$ whose optimal value is
unknown or depends on unknown quantities. Let us denote
  by $\widehat{\mathcal{G}}_P$ the collection of graphs estimated
  using this procedure $P$ :
\begin{eqnarray}\label{collection_tuning}
\widehat{\mathcal{G}}_P=
\left\{\widehat{G}_{P}(\lambda)\ ,\,
\lambda>0\text{ and }\degr(G_{P}(\lambda))\leq D\right\}.
\end{eqnarray}
We propose to choose $\lambda$ by minimizing the
criterion~(\ref{definition_critere}) over the collection
$\widehat{\mathcal{G}}_P$. Thus, we  get an estimated graph 
$\widehat{G}_{P} = \widehat{G}_{P}(\widehat{\lambda}_{P})$.
Theorems \ref{proposition_risque} and
\ref{proposition_consistance} in Section~\ref{section_theorique} state that
GGMselect almost selects the best graph among
this collection. \\

Assume now that we have at hand a collection $\mathcal{P}$ of estimation
procedures which possibly depends on
tuning parameters. For any procedure $P\in\mathcal{P}$, we compute the
collection $\widehat{\mathcal{G}}_P$ either defined by (\ref{collection_tuning})
if $P$ depends on tuning
  parameters  or by
$\{\widehat{G}_P\}$ if not. Then, we propose to select a
procedure $\hat P$ and a graph 
  $\widehat{G}_{\widehat{P}}$   by minimizing the
criterion~(\ref{definition_critere})  over the collection 
\begin{eqnarray}\label{collection_mix}
\widehat{\mathcal{G}}_{\mathcal{P}}= \left\{\widehat{\G}_{P}\ ,\,
P\in \mathcal{P}\right\}\ .
\end{eqnarray}
Again, Theorems \ref{proposition_risque} and
\ref{proposition_consistance}  in Section~\ref{section_theorique} ensure  that
GGMselect almost selects the best graph among
the collection $\widehat{\mathcal{G}}_{\mathcal{P}}$.\\

Next, we briefly describe the four families of candidate graphs
$\widehat\G_{\CO 1}$, $\widehat\G_{\LA}$, $\widehat\G_{\EW}$ and $\widehat\G_{\QE}$ that we advise to use, the details being postponed to Section \ref{section_family_details}. Except $\widehat\G_{\QE}$, all the other families are built from estimators in the literature that depend on an unknown tuning parameter. On the one hand, GGMselect allows to tune these procedures. On the other hand, GGMselect allows to select an estimator by combining these different procedures.

\subsubsection{C01 family $\widehat{\mathcal{G}}_{\CO 1}$}

The family $\widehat{\mathcal{G}}_{\CO 1}$
derives from the estimation procedure proposed by Wille and B\"uhlmann~\cite{WB06} and is based on the 0-1 \emph{conditional independence graph}
$G_{01}$. This graph  is defined as follows.
For each pair of nodes $(a,b)$, we write $R_{a,b|\emptyset}$ for the correlation
between the variables $X_a$ and $X_b$ and  $R_{a,b|c}$ for the correlation of
$X_a$ and $X_b$ conditionally on $X_c$. Then, there is an edge between $a$ and
$b$ in $G_{01}$, if and only if 
 $R_{a,b|\emptyset}\neq 0$ and $R_{a,b|c}\neq 0$ for all $c\in
\Gamma\setminus\{a,b\}$, viz
\begin{eqnarray}
a\stackrel{{G}_{01}}{\sim}b &\Longleftrightarrow&   \min\left\{|R_{a,b|c}|,\
c\in \{\emptyset\}\cup\Gamma\setminus\{a,b\} \right\}>0\,.
\end{eqnarray}
 Although the 0-1 conditional independence graph $G_{01}$ does not  usually
coincide  with the graph $G_{\Sigma}$, there is a close connection between both
graphs in some cases (see Wille and B\"uhlmann). Given a number $0<\alpha<1$, 
Wille and B\"uhlmann propose to estimate $G_{01}$ by a graph $\widehat{G}_{01,\alpha}$ built from a collection of likelihood ratio test level of $\alpha$. The graph  $\widehat{G}_{01,\alpha}$  becomes more connected when $\alpha$ increases.  
We define the family
$\widehat{\mathcal{G}}_{\CO 1}$ as the set of graphs
$\widehat{G}_{01,\alpha}$ with all levels $\alpha$ small enough to ensure that
$\degr(\widehat{G}_{01,\alpha})\leq D$.\\

\noindent 
\textit{Complexity}. The computation of $\widehat{\mathcal{G}}_{\CO 1}$ goes very fast since its complexity
is of order $np^3$ (see Section \ref{section_family_details}). The size of the family $\widehat{\mathcal{G}}_{\CO 1}$ is
smaller than $pD$. Computational times for some examples are
given in Section~\ref{section_comparaison_temporel}.

 \subsubsection{Lasso-And  family $\widehat{\mathcal{G}}_{\LA}$}
 
The Lasso-And family
$\widehat{\mathcal{G}}_{\LA}$ derives from the estimation procedure proposed by
Meinshausen and B\"uhlmann~\cite{MB06} and is based on the 
LARS-lasso algorithm~\cite{lars}. 
For any $\lambda>0$, we define the $p\times p$ matrix
$\widehat{\theta}^{\lambda}$  by
\begin{eqnarray}\label{lasso_general}
 \widehat{\theta}^{\lambda}=\arg\!\min\left\{\|{\X}-{\X}\theta'\|_{n\times
p}^2+\lambda\|\theta'\|_1: \ \theta'\in\Theta \right\}, 
\end{eqnarray}
where $\Theta$ is the set of $p\times p$ matrices with 0 on the diagonal and
$\|\theta'\|_1=\sum_{a\neq b}|\theta'_{a,b}|$. Then, we define the graph
$\widehat{G}^{\lambda}_{\text{and}}$ by setting  an edge between $a$ and $b$  if
both $\widehat{\theta}_{a,b}^{\lambda}$ \underline{and}
$\widehat{\theta}_{b,a}^{\lambda}$ are non-zero.
 This graph $\widehat{G}^{\lambda}_{\text{and}}$ is exactly the estimator~(7)
introduced in \cite{MB06}. The size of
$\widehat{G}^{\lambda}_{\text{and}}$ has a tendency to increase when the tuning
parameter $\lambda$ decreases. Hence, we define the family
$\widehat{\mathcal{G}}_{\LA}$ as the set of graphs
$\widehat{G}^{\lambda}_{\text{and}}$ with all $\lambda$ large enough to ensure that
$\degr(\widehat{G}^{\lambda}_{\text{and}})\leq D$.\\

\noindent 
\textit{Complexity}. The complexity of the LARS-lasso algorithm is
unknown in general. Nevertheless, according to 
Efron et al.~\cite{lars} the algorithm requires $O(np(n\wedge p))$ operations in most cases.
Hence, the whole complexity of the LA algorithm is generally of the order 
$p^2n(n\wedge p)$ (see Section \ref{section_family_details}). Finally, the size of the family $\widehat\G_{\LA}$ cannot be
bounded uniformly, but  it remains  smaller than  $pD$ in practice.

\subsubsection{Adaptive lasso family $\widehat\G_{\EW}$}
The family $\widehat\G_{\EW}$ is a modified version of
$\widehat\G_{\LA}$ inspired by the adaptive lasso~\cite{zou_adaptive}. The major difference between $\widehat\G_{\EW}$ and 
$\widehat\G_{\LA}$ lies in  the replacement of the $l^1$ norm $\|\theta'\|_{1}$
in~(\ref{lasso_general})  by $\|\theta'/\widehat\theta^{\mathrm{init}}\|_{1}$,
where $\widehat\theta^{\mathrm{init}}$ is a preliminary estimator of $\theta$
and $\theta'/\widehat\theta^{\mathrm{init}}$ stands for the matrix with entries
$(\theta'/\widehat\theta^{\mathrm{init}})_{a,b}=\theta'_{a,b}/\widehat\theta^{
\mathrm{init}}_{a,b}$. Zou suggests to take for $\widehat\theta^{\mathrm{init}}$
a ridge estimator. Here, we propose to use instead  the Exponential Weights 
estimator $\widehat\theta^{EW}$ of Dalalyan and Tsybakov~\cite{DT08,DT09}. The
choice of this estimator appears more natural to us since it is designed for the
sparse setting and enjoys  nice theoretical properties. Moreover, we have
observed on some simulated examples, that the adaptive lasso with the
Exponential Weights  initial estimator performs much better than the adaptive
lasso with the ridge initial estimator.
 
Given $\lambda>0$,  $\widehat{\theta}^{\EW,\lambda}$ is the adaptive lasso estimator of $\theta$ with initial estimator $\widehat{\theta}^{\EW}$.
We define the graph $\widehat{G}^{\EW,\lambda}_{\text{or}}$
by setting an edge between $a$ and $b$ if either
 $\widehat{\theta}_{b,a}^{EW,\lambda}$ \underline{or}
$\widehat{\theta}_{a,b}^{EW,\lambda}$ is non-zero.
 Finally, the family  $\widehat{\mathcal{G}}_{\EW}$ is the set of graphs
$\widehat{G}^{\EW,\lambda}_{\text{or}}$ with $\lambda$ large enough to ensure that
$\degr(\widehat{G}^{\EW,\lambda}_{\text{or}})\leq D$.\\

\noindent 
\textit{Complexity}. 
The complexity of the estimation $\widehat{\theta}^{\EW}$  depends on the choices of the tuning
parameters for the Exponential Weights estimator (see Section \ref{section_family_details}). Some examples are given in Section
\ref{section_comparaison_temporel}.
The complexity of the other computations is the same as for  the  $\LA$-algorithm and is  of
the order $p^2n(n\wedge p)$ in practice.
Finally, as for $\widehat{\mathcal{G}}_{\LA}$, we do not know a general bound
for the size of $\widehat{\mathcal{G}}_{\EW}$, but it remains smaller than $pD$
in practice.

\subsubsection{Quasi-exhaustive family $\widehat\G_{\QE}$}
Roughly, the idea is to break down  the minimization of
the criterion~(\ref{definition_critere}) over all the graphs of degree at most
$D$ into $p$ independent problems. For each node $a\in\Gamma$, we estimate the 
neighborhood of $a$ by
$$\widehat{\nei}(a)=\textrm{argmin}\ac{\|\X_{a}-\textrm{Proj}_{V_{S}}(\X_{a}
)\|_n^2\left(1+\frac{\pen(|S|)}{n-|S|}\right): \ S\subset{\Gamma\setminus\{a\}}
\textrm{ and } |S|\leq D},$$
where $\pen$ is the penalty function (\ref{definition_penalite}) and
$\textrm{Proj}_{V_{S}}$ denotes the orthogonal projection from $\R^n$ onto
$V_{S}=\ac{\X\beta:  \beta\in\R^p\textrm{ and supp}(\beta)=S}$.
We know from \cite{Verzelen08} that $\widehat{\nei}(a)$ is a good
estimator of the true neighborhood of $a$, from a non-asymptotic point of view.
We then build two nested graphs $\widehat{G}_{K,\text{and}}$ and
$\widehat{G}_{K,\text{or}}$ in a similar way as in \cite{MB06}. Namely, there is an edge between $a$ and $b$ in $\widehat{G}_{K,\text{and}}$ if $a\in
\widehat{\nei}(b)\textrm{ \underline{and} } b\in \widehat{\nei}(a)$ and there is an edge between $a$ and $b$ in  $\widehat{G}_{K,\text{or}}$ if $a\in
\widehat{\nei}(b)\textrm{ \underline{or} } b\in \widehat{\nei}(a)$.
The family $\widehat{\mathcal{G}}_{\QE}$ is defined as the collection of all the
graphs that lie between $\widehat{G}_{K,\text{and}}$ and
$\widehat{G}_{K,\text{or}}$
\begin{eqnarray*}
\widehat{\mathcal{G}}_{\QE} = \left\{G,\ \widehat{G}_{K,\text{and}} \subset G
\subset \widehat{G}_{K,\text{or}}\text{ and }\degr(G)\leq D\right\}.
\end{eqnarray*}
It is likely that the graph $\widehat G_{\textrm{exhaustive}}$ which
minimizes~(\ref{definition_critere}) over all the graphs of degree at most $D$
belongs to the family $\widehat{\mathcal{G}}_{\QE}$.
In such a case, the minimizer $\widehat{G}_{\QE}$ of the
criterion~(\ref{definition_critere}) over  $\widehat\G_{\QE}$ coincides with the
estimator $\widehat G_{\textrm{exhaustive}}$.\\

\noindent 
\textit{Complexity}. The complexity of the computation of the
collections $\widehat{\nei}(a)$ is much smaller than the complexity of the
computation of $\widehat{G}_{\textrm{exhaustive}}$. Nevertheless, it still
remains of order $np^{D+1}D^3$ and  the size of the family
$\widehat{\mathcal{G}}_{\QE}$ can be of order $2^{pD/2}$ in the worst cases.
However, for sparse graphs $G_{\Sigma}$, the graphs $\widehat{G}_{K,\text{and}}$
and $\widehat{G}_{K,\text{or}}$ are  quite similar in practice, which makes the
size of $\widehat{\mathcal{G}}_{\QE}$ much smaller. The procedure then remains
tractable
for $p$ and $D$ reasonably small.

\section{Theoretical results} \label{section_theorique}
In order to assess the performance
of our selection procedure, we state in this section two kinds of theoretical results: a non-asymptotic
oracle-like inequality concerning the estimation of $\theta$ and a consistency
result for the estimation of $G_{\Sigma}$.

\subsection{A non-asymptotic oracle-like inequality}\label{section_oracle}
 We associate to the graph $\widehat G$ selected by the
procedure of Section~\ref{procedure}, the estimator
$\thetat=\widehat\theta_{\widehat G}$ of the matrix $\theta$, where $\widehat
\theta_{G}$ is given by~(\ref{thetahat}) for any graph $G\in\widehat\G$. The
quality of the estimation of $\theta$ is quantified by the $\MSEP$ of
$\thetat$
defined by
$$\MSEP(\thetat)=\E\cro{\|\Sigma^{1/2}(\thetat-\theta)\|^2_{p\times p}}.$$
We refer to the introduction of~\cite{giraud08} for a discussion on the
relevance of the use of the $\MSEP$ of $\thetat$ to assess the quality of the
estimator $\widehat G$. In the sequel, $I$ stands for the identity matrix of
size $p$.

First, we can compare the $\MSEP$ of $\thetat$ to the $\MSEP$ of
$\widehat\theta_{G_{\Sigma}}$ when the minimal graph $G_{\Sigma}$ belongs
to $\widehat \G$ with
large probability. Roughly speaking, the MSEP of $\thetat$ is in this case
smaller (up to a $\log p$ factor) than the MSEP of
$\widehat{\theta}_{G_{\Sigma}}$. This means that $\thetat$ performs almost as
well as if we knew the true graph $G_{\Sigma}$ in advance.

\begin{prte}\label{corollaire_risque}
Assume that $n\geq 9$.
Let $\widehat{\mathcal{G}}$ be any (data-driven) family of graphs with
maximal
degree $D_{\widehat{\mathcal{G}}}= \max \{\degr(G),\ G\in
\widehat{\mathcal{G}}\}$ fulfilling
\begin{eqnarray}\label{condition_degree}
 1\leq D_{\widehat{\mathcal{G}}}\leq \gamma\ \frac{n }{2(1.1+ \sqrt{\log p}
)^2}\ ,\hspace{0.5cm}\text{for some }\gamma<1\,.
\end{eqnarray}
If the minimal graph $G_{\Sigma}$
belongs to the family  $\widehat{\mathcal{G}}$ with large probability
\begin{equation}\label{condition_collection}
 \mathbb{P}\left(G_{\Sigma}\in\widehat{\mathcal{G}}\right)\geq 1 -
\alpha\exp(-\beta n^\delta), \hspace{1cm}\text{for some
}\alpha,\beta,\delta>0\
\end{equation}
then, the $\MSEP$ of the estimator $\widetilde{\theta}$  is upper bounded by
\begin{equation}\label{borne_oracle2}
\MSEP(\widetilde{\theta}) \leq  L_{K,\gamma}\log(p)\pa{
\MSEP(\widehat{\theta}_{G_{\Sigma}})\vee {\MSEP(I)\over n}}
+R_n\,.
\end{equation}
where $L_{K,\gamma}$ is a positive constant depending on $K$ and $\gamma$
only
and the residual term $R_{n}=R_n(\Sigma,\gamma,\alpha,\beta,\delta)$  is of
order $n^3\text{tr}(\Sigma) [e^{-n(\sqrt{\gamma}-\gamma)^2/4}+
\sqrt{\alpha}e^{-\frac{\beta}{2}n^\delta}]$.
\end{prte}

Observe that the residual term $R_n$ goes to 0 exponentially fast with
respect to $n$.
If we forget the term $n^{-1}\MSEP(I)$, then the risk bound
(\ref{borne_oracle2}) essentially states that
the estimator $\thetat$ performs almost as well as if we knew the graph
$G_{\Sigma}$ in advance.

Let us now compare the additional term $n^{-1}\MSEP(I)$ appearing in
(\ref{borne_oracle2}) with the risk
$\MSEP(\widehat{\theta}_{G_{\Sigma}})$. The additional
term $n^{-1}\MSEP(I)$  is equal to $n^{-1}\sum_a{\sigma_a^2}$, where
$\sigma_{a}^2$ stands for
the conditional variance of $X_{a}$ given the remaining variables.
Hence, this quantity is usually smaller than
the risk $\MSEP(\widehat{\theta}_{G_{\Sigma}})$ which is  a
variance term of order $n^{-1}\sum_a d_{a}(G_{\Sigma}){\sigma_a^2}$.
Nevertheless, when
the true graph $G_{\Sigma}$ is empty and  the collection $\widehat{\G}$
contains the empty
graph,  the additional term $n^{-1}\MSEP(I)$ is dominant and the estimator
$\thetat$ is not optimal. Such a drawback is actually unavoidable in model
selection when the target is too close to zero (see Sect.2.3.3 of \cite{BM01} for a discussion). Assumption~(\ref{condition_degree}) is discussed after Theorem
\ref{proposition_risque}.

In Proposition \ref{corollaire_risque}, we state that $\thetat$ performs
almost as well as $\widehat{\theta}_{G_{\Sigma}}$. Nevertheless, the risk
of the estimator $\widehat{\theta}_{G_{\Sigma}}$ can be quite large,
especially when the graph $G_{\Sigma}$ contains a lot of edges. For an
arbitrary graph $G$, the risk $\MSEP(\widehat{\theta}_{G})$ is the sum of
the bias and the variance terms.
If we consider a sparser graph $G$, the estimator $\widehat{\theta}_{G}$
is biased but its variance is smaller, so its risk
$\MSEP(\widehat{\theta}_{G})$ can be smaller.  The estimator
$\widehat{\theta}_{G^*}$ which minimizes the
$\MSEP$ over the collection of estimators
$(\widehat{\theta}_G)_{G\in\widehat{\mathcal{G}}}$ is called
the {\it oracle}.
Observe that the graph $G^*$ is unknown since it is related to the unknown
matrix $\theta$. One goal of model selection is to select an estimator
$\thetat$ which performs almost as well as the oracle estimator. Such a
result is stronger than Proposition \ref{corollaire_risque}. We state it
in the next theorem by providing a so-called {\it oracle inequality} (Eq.
(\ref{borne_oracle})).

\begin{thrm}\label{proposition_risque}
Assume that $n\geq 9$.
Let $\widehat{\mathcal{G}}$ be any (data-driven) family of graphs with
maximal
degree $D_{\widehat{\mathcal{G}}}= \max \{\degr(G),\ G\in
\widehat{\mathcal{G}}\}$ fulfilling~(\ref{condition_degree}).
Then, the $\MSEP$ of the estimator $\widetilde{\theta}$ is upper bounded by
\begin{equation}\label{borne_oracle}
\MSEP(\widetilde{\theta}) \leq  L_{K,\gamma}\log(p)\pa{
\E\cro{\inf_{G\in\widehat \G}\pa{\MSEP(\widehat{\theta}_G)}}\vee
{\MSEP(I)\over
n}}
+R_n\,.
\end{equation}
where $L_{K,\gamma}$ is a positive constant depending on $K$ and $\gamma$
only
and  the residual term $R_{n}=R_n(\Sigma,\gamma)$ (made explicit in the
proof)
is of order $n^3\text{tr}(\Sigma) e^{-n(\sqrt{\gamma}-\gamma)^2/4}$.
\end{thrm} 

If we forget the term $n^{-1}\MSEP(I)$ in (\ref{borne_oracle}),
Theorem~\ref{proposition_risque} states that under
Condition~(\ref{condition_degree}) the MSEP of $\tilde \theta$ nearly achieves,
up to a $\log (p)$ factor, the average minimal MSEP of the family of estimators
$\{\widehat\theta_{G},\ G\in\widehat\G\}$. Hence, $\widetilde{\theta}$ performs
almost as well as the oracle up to a $\log p$ factor. This logarithmic factor is
proved to be unavoidable from  a minimax point of view 
(see \cite{Verzelen08} Sect. 4.2).

Let us compare the risk bound (\ref{borne_oracle}) with Theorem~1 of
Giraud~\cite{giraud08}. This theorem claims that the procedure
nearly selects the best graph among a {\em fixed} collection of graphs.
 In contrast, our
collection of graphs $\widehat{\G}$ is not fixed a priori and depends on the
data ${\bf X}$. Here, we prove that the graph $\widehat{G}$ is nearly
the best (in terms of MSEP) among the random collection $\widehat{G}$. As a
simple example, let us consider the procedure GGMselect with the Lasso-And
family $\widehat{\G}_{LA}$. Theorem \ref{proposition_risque} tells us that the
selected graph $\widehat{G}$ nearly achieves the smallest MSEP among the
collection of Lasso-And graph estimators
$\{\widehat{G}_{\text{and}}^{\lambda}\}_{\lambda>0}$. In other words, GGMselect
nearly selects the best tuning parameter of the Lasso-And procedure.

The condition (\ref{condition_degree}) roughly states that we restrict
ourselves to graphs whose maximal degree is smaller than
$n/(2\log(p))$. For the related problem of random design regression, it is
proved in~\cite{Vminimax}
that theoretical limitations are occurring when the size of the support of the
parameter is larger than $n/(2\log(p))$. In this so-called ultra-high
dimensional setting, it is not possible to obtain an oracle bound of the form
(\ref{borne_oracle}) and it is shown that recovering the
support of the
parameter is almost impossible. In short, estimating a graph whose maximal
degree is larger than
$n/(2\log(p))$ is nearly impossible.

\subsection{Consistency of the selection procedure}
The next theorem states,  under mild assumptions, a consistency result 
for our selection procedure in  a high-dimensional setting.  In the spirit of
the results of Meinshausen and B\"uhlmann~\cite{MB06}, we consider the case
where  the number of variables $p$ increase with the sample size $n$.\\

\noindent
We make  the following assumptions:
\begin{eqnarray*}
\text{{\bf (H.1)}}&  & p_n\geq n\ . \\
\text{{\bf (H.2)}}& &\degr(G_{\Sigma_n})\leq \frac{n^s}{\log p_n}\wedge
\frac{n}{\log^2 p_n} \text{ for some }s<1\ .\\
\text{{\bf (H.3)}}& &\min_{a\neq b,\
b\in\nei_{G_{\Sigma_n}}(a)}\theta_{a,b}^2\min_{a\neq
b}\frac{\var(X_a|X_{-a})}{\var(X_b|X_{-b})}\geq n^{s'-1}\text{ for some }s'>s\ .
\end{eqnarray*}

\begin{thrm}\label{proposition_consistance}
Assume that the family $\widehat{\mathcal{G}}$ of candidate graphs  contains the
true graph with probability going to 1 and {\bf (H.1)}, {\bf (H.2)}, {\bf (H.3)}
are fulfilled. Then, the  estimation procedure GGMselect with $K  > \left[3\vee
{2.5\over (1-s)}\right]$ and 
\begin{eqnarray*}
D_{\widehat{\mathcal{G}}}= \max \{\degr(G),\ G\in \widehat{\mathcal{G}}\}&\leq&
\frac{n}{\log^2 p_n}
\end{eqnarray*}
is consistent. More precisely, there exist some universal constant $L$ and some
integer $n_0=n_0\left[K, s, s'\right]$ not depending on the true graph
$G_{\Sigma_n}$ nor on the covariance $\Sigma_n$ such that 
\begin{equation*}\mathbb{P} \left[ \widehat{G} = G_{\Sigma_n}\right]\geq 1-
Lp_n^{-1/2}-
\mathbb{P}\left[G_{\Sigma_n}\notin\widehat{\mathcal{G}}\right],\quad \textrm{
for any }n\geq n_{0}\,.
\end{equation*} 
\end{thrm}

Let us discuss the assumptions of the theorem and their  similarity with some of
the hypotheses made in \cite{MB06}. The
Assumption~{\bf (H.2)} is met if $p_{n}$ grows polynomially with respect to $n$ 
and the degree of the true graph does not grow faster than $n^{\kappa}$ with
$\kappa<s$ (which corresponds to Assumptions 1 and 2 in~\cite{MB06}). We
mention
that {\bf (H.2)} is not satisfied when $p_{n}$ grows exponentially with $n$
unless $G_{\Sigma_n}$ is empty. It is actually impossible to consistently
estimate a non-empty graph if $p_n$ is of order $\exp(n)$, see 
\cite{Vminimax}.

The Assumption {\bf (H.3)} ensures that the conditional variances as well as the
non-zero terms $\theta_{a,b}$ are large enough so that the edges can be
detected. 
To compare with~\cite{MB06}, Assumption {\bf (H.3)} is met as soon as
Assumption
2 and 5 in~\cite{MB06} are satisfied. In addition, we underline that we make  
no assumption on the $l^1$-norm of the prediction coefficients or on the signs
of $\theta_{a,b}$ (Assumptions  4 and 6 in \cite{MB06}).

Finally, we do not claim that the condition $K>\left[2.5/(1-s)\vee 3\right]$ is
minimal to obtain  consistency. It seems from simulation experiments that
smaller choices of $K$ also provide good estimations.

\section{Numerical study}\label{section_simulations}
It is essential to investigate the performance of statistical
procedures on data. Since we do not  know  the actual underlying
graph of conditional dependences on real data sets, we mainly opt for a
numerical study with simulated data. Our aims in this study are
 to evaluate the feasibility of the GGMselect procedure and to 
compare its performances  with those of
recent graph-selection procedures. \\


{\bf Simulating the data}.
The matrix $\X$ is composed of $n$ i.i.d. rows with Gaussian
$\mathcal{N}_{p}(0,\Omega^{-1})$ distribution where the inverse covariance
matrix $\Omega$ is constructed according to the following procedure.
We set $\Omega=BB^T+D$, where $B$ is a 
random sparse lower triangular matrix and $D$ is a diagonal matrix
with random entries of order $10^{-3}$. The latter matrix $D$ prevents
$\Omega$ from having too small eigenvalues.  To generate $B$ we split
$\ac{1,\ldots,p}$ into three consecutive sets $I_{1}$, $I_{2}$, $I_{3}$
of approximately equal size, and choose two real numbers
$\eta_\mathrm{int}$ and 
$\eta_\mathrm{ext}$ between 0 and 1. For any $a, b$ such that $1\leq
a<b \leq p$,  we set $B_{a,b}=0$ with probability
$1-\eta_\mathrm{int}$ if $a$ and $b$ are in the same set, and  we set
$B_{a,b}=0$  with probability $1-\eta_\mathrm{ext}$ if $a$ and $b$
belong to two different sets. Then, the lower diagonal values that
have not been set to 0 are drawn according to a uniform law on
$\cro{-1,1}$ and the diagonal values are drawn according to a uniform
law on $\cro{0,\varepsilon}$. Finally, we rescale $\Omega$ in order to
have 1 on the diagonal of $\Sigma=\Omega^{-1}$.  This matrix $\Sigma$ defines
a graph $G = G_{\Sigma}$ and a matrix
$\theta$ defined as in Section~\ref{SelProc.st}. The sparsity of the
graph is measured via a sparsity index noted $I_{s}$, defined as the
average number of edges per nodes in the graph.

In our simulation study we set
$\eta=\eta_\mathrm{int}=5\eta_\mathrm{ext}$, and $\varepsilon=0.1$. 
We evaluate the value of $\eta$ corresponding to a desired value of the sparsity
index $I_{s}$ by simulation.
$I_{s}$ equals the desired value.
Choosing $I_{s}$ small, we get sparse graphs whose edges distribution
is not uniform, see Figure~\ref{dessGr.fg}.\\

\begin{figure}[htbp]
\begin{center}
{\includegraphics[height=12.5cm,width=6cm,angle=270]{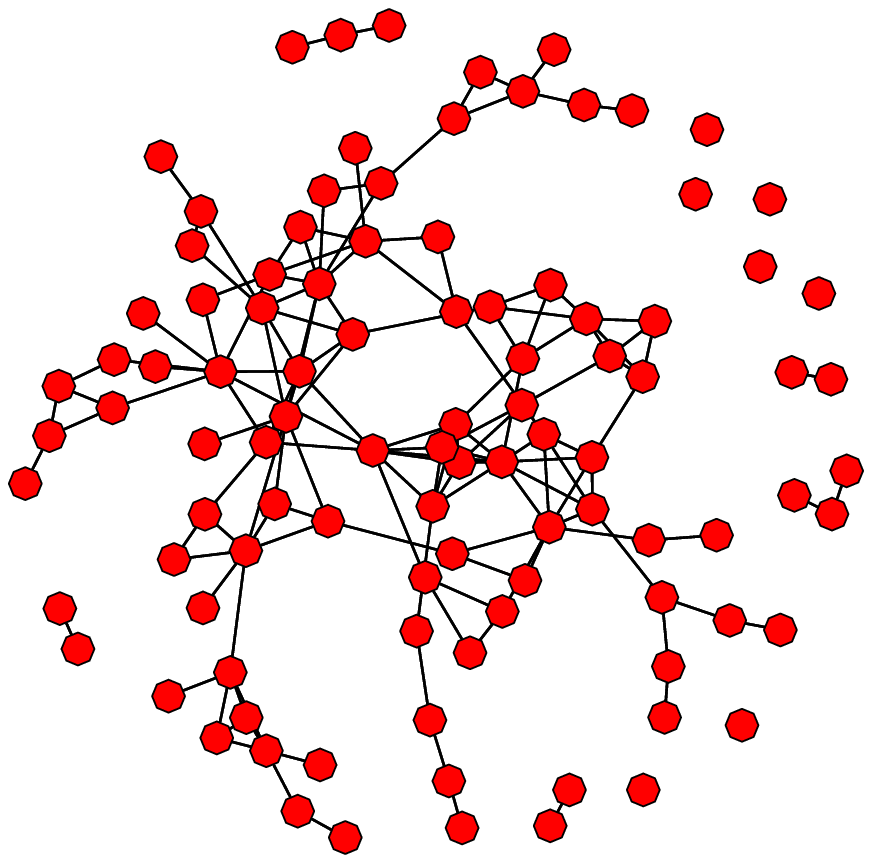}}
\end{center}
\caption{\label{dessGr.fg} One simulated graph $G$ with $p=30$ and
  $I_{s}=3$. The degree $\deg(G)$ of the graph equals 8.}
\end{figure}

{\bf GGMselect: choice of graphs families}.
Our procedure is applied for the families of graphs presented in
  Section~\ref{FamGr.st}. The methods  are respectively denoted
  {\tt C01}, {\tt LA}, {\tt EW} and {\tt QE}. 


The family $\widehat{\mathcal{G}}_{\EW}$ is based on the calculation
of exponential weight  estimators $\widehat{\theta}^{\EW}$. This
calculation depends on parameters, denoted
$\alpha, \beta, \sigma, \tau$~in \cite{DT09}, that defined the
aggregation procedure, and on parameters, denoted $h$ and
$T$~in \cite{DT09}, used in the
Langevin Monte-Carlo algorithm (see Section \ref{section_family_details} for details). We chose these parameters as
follows. The matrix 
$\X$ being scaled such that the norm of each column equals 1, we took
$\sigma=1/\sqrt{n}$,  and we 
set $\alpha=0$, $\beta=2/n$, $\tau=1/\sqrt{n(p-1)}$ and $h=10^{-3}$,
$T=200$. Using these parameters values we did not encountered
convergence problems in our simulation study.

As it was already mentioned in Section~\ref{FamGr.st}, the size of
the family $\widehat{\mathcal{G}}_{\QE}$ may be very large leading to  memory
size problems in the computational process. In that case, as soon as a
memory size problem is encountered, the
research between $\widehat{G}_{K,\text{and}}$ and
$\widehat{G}_{K,\text{or}}$ is stopped and  prolonged by  a stepwise
procedure.  

Our procedure depends on two parameters: $K$ occurring in the penalty
function (see Equation~\ref{definition_penalite}) and $D$ the maximum
degree of the graph. We choose $K=2.5$ in all simulation
experiments. In practice, we want to choose $D$  as large as
possible. From theoretical results in Section~\ref{section_theorique}
and in~\cite{Vminimax}, 
we know that we can take $D$ as large as $\lfloor
n/(2\log(p))\rfloor$, and that it is nearly 
impossible to perform
estimation of a graph when the maximal degree is larger than
$n/(2\log(p))$. We then set $D=\lfloor
n/(2\log(p))\rfloor$ except for ${\QE}$ whose algorithmic complexity
increases exponentially with $D$. 

All these methods  are implemented in R-2.7.2 in  the package {\tt GGMselect}. 

\subsection{CPU times}\label{section_comparaison_temporel}

We assess the practical feasibility of the methods we propose
from the point of view of the memory size and computer time. 
To this aim,
we simulate graphs with $p=30, 100, 200, 300, 500$ nodes, sparsity $I_{s}=3$
and $n=50$. 
The simulation were run on a Bi-Pro  Xeon quad core 2.66 GHz with  24 Go RAM. 
The computer time  being strongly dependent on the simulated graph we
calculate the mean of computer times over $N_{G}=100$ simulated
graphs. For each of these graphs, one matrix $\X$ is simulated. The
results are given in Table~\ref{CompT.tb}. The maximum degree $D$ of the
estimated graph was set to $\lfloor n/2\log(p)\rfloor$, except for the {\tt QE}
method  where
$D=3$ and 5. The maximum allowed memory size is
exceeded for the {\tt QE} method when $D=5$ and $p \geq 100$, and when
$D=3$ for $p\geq 300$. The {\tt LA} and {\tt C01} methods are running very
fast. The computing time for the {\tt EW} method increases quickly
with $p$: in
this simulation study, it is roughly proportional to
$\exp\left(\sqrt{p}/2\right)$,  see Figure~\ref{dessCT.fg}. This order of
magnitude is obviously
dependent on the choice of the parameters occurring in the Langevin
Monte-Carlo algorithm for calculating  $\widehat\theta^{EW}$. 

\begin{table}[htbp]
\begin{center}

\begin{tabular}{l|c|c|c|rr|rr|}
 & 
   \multicolumn{3}{c|}{$D=\lfloor n/(2\log(p))\rfloor$}
& \multicolumn{2}{c|}{$D=3$} & 
   \multicolumn{2}{c|}{$D=5$} \\ \hline
$p$  & \hspace{0.4 cm}
  {\tt EW} \hspace{0.4 cm} & 
 \hspace{0.4 cm}   {\tt LA} \hspace{0.4 cm}  & \hspace{0.4 cm} {\tt C01} \hspace{0.4 cm} 
& \multicolumn{2}{c|}{{\tt QE}} & 
   \multicolumn{2}{c|}{{\tt QE}}\\ \hline
30  & 7.1  & 0.46 & 0.04 & 16 &  {\small $[1.9, 1366]$} & 146 &  {\small $[125, 975]$}\\
100  & 111 & 3.11 & 0.13 
& $1956$ &  {\small $[240, 5628]$} & \multicolumn{2}{c|}{$>$ams}\\
200  & 853 
 & 8.0 \,\   & 0.68 & $4240 $ & {\small $[4008, 5178]$}& \multicolumn{2}{c|}{$>$ams}\\
300 & 
$4277$  & 15.5\,\  \  & 2.27  & \multicolumn{2}{c|}{$>$ams} & \multicolumn{2}{c|}{$>$ams}\\
500  & 158550 
 &   43\,\  \  & 9.7 & \multicolumn{2}{c|}{$>$ams} & \multicolumn{2}{c|}{$>$ams}\\
\end{tabular} 

\end{center}
\caption{\label{CompT.tb} Means and ranges (in square brackets) of
  computing times in seconds 
  calculated over $N_{G}=100$   simulated
  graphs. For {\tt EW}, {\tt LA} and {\tt C01} there is nearly no  variability
in the computing
  times. {\rm $>$ams} means that the maximum allowed memory size was exceeded.} 
\end{table}

\begin{figure}[htbp]
\begin{center}
\includegraphics[height=8cm,width=6cm,angle=270]{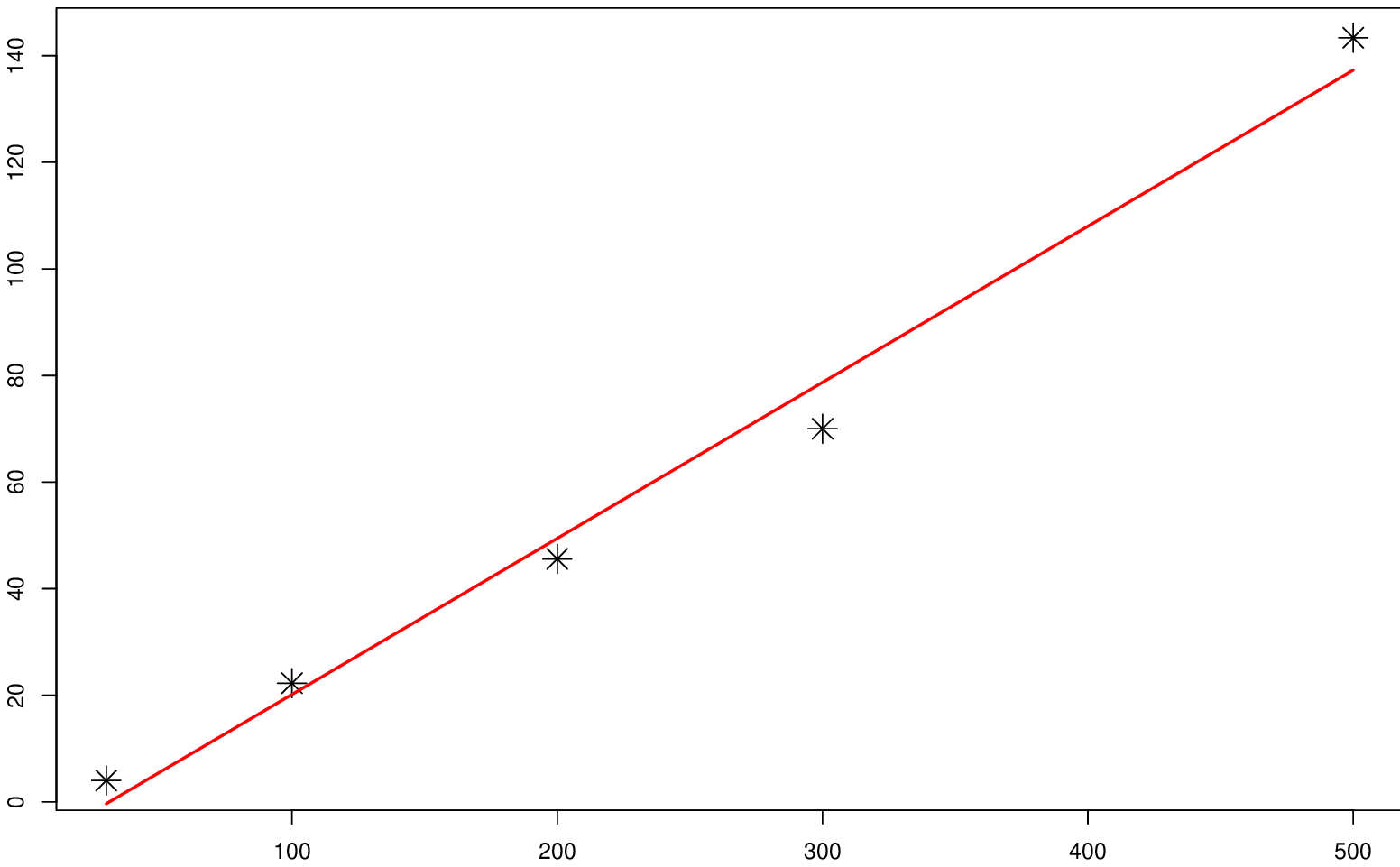}
\end{center}
\caption{\label{dessCT.fg}Graphic of $\log^{2}(\mbox{CPU time})$
  versus $p$ for the {\tt EW} method.}
\end{figure}

\subsection{Methods comparison}

We compare our methods with the following ones: 
\begin{itemize}
\item the 0-1 conditional independence approach proposed in \cite{WB06},
with the decision rule based on the adjusted
  p-values following the Benjamini-Hochberg procedure taking
  $\alpha=5\%$. 
\item the lasso approach, with the two variants {\tt and} and {\tt or}
  proposed in \cite{MB06}, taking   $\alpha=5\%$. 
\item the adaptive glasso method proposed  in \cite{Fan08}. It works
in two steps. First, the matrix
    $\Omega$ is estimated using 
    the glasso method.  Then the glasso procedure is run again using
    weights in the penalty that depend on the previous estimate of
    $\Omega$, see Equation~(2.5) in \cite{Fan08}. 
At each step the
    regularization  parameter is calculated by  K-fold
    cross-validation.
\end{itemize}
These methods will be denoted as {\tt WB}, {\tt MB.and}, {\tt MB.or} and
{\tt Aglasso}. They  were implemented in  R-2.7.2 using the packages
{\tt lars} for the {\tt MB} methods and the package {\tt glasso} for the last
one. \\

{\bf Assessing the performances of the methods}.
We  assess the performances of the investigated methods on the 
basis of $N_{G} \times N_{X}$ runs where $N_{G}$ is the number of
simulated graphs and  $N_{X}$ the number of matrices $\X$ simulated
for each of these graphs. We compare
each simulated graph  $G$ with the estimated graphs $\widehat{G}$ by
counting edges that are correctly identified as present or absent, and
those that are wrongly identified. We thus estimate  the false
discovery rate (or FDR)  defined as the expected 
proportion of wrongly detected  edges  among edges detected as
present, and  the
power defined as the expected proportion of rightly detected  edges
among edges  present in the graph.

The statistical procedures
designed to select graphs  have one or several parameters that must be
tuned. The quality of the final estimation is then affected as well by
the intrinsic ability of the procedure to select an accurate graph, as
by the parameter tuning.  
First, we  investigate the first issue by  varying the
values of the tuning parameters and plotting {\it power versus FDR
  curves}. We choose  $p=100$, $n=50$ and $I_{s}=3$. Then, taking the point
of view of a typical user, we compare the different procedures with
the tuning parameter recommended in the literature. 
We investigate the effect of $n$ by choosing $n=30, 50, 100, 150$,
keeping $p=100$. We also evaluate the  effect of graph sparsity taking
$I_{s} = 1, 2, 3, 4, 5$, $p=30$ to keep the computer time under
reasonable values, and $n=30$.  Finally, we compare our
 criterion defined by Equations~\eref{definition_critere}
 and~\eref{definition_penalite}  to a BIC-type criterion which selects
 a graph by minimizing with respect to $G \in \widehat \G$,
\begin{equation*}
 \crit_{\mbox{BIC}}(G)= \sum_{a=1}^p
\exp\left( \log \left\{\|{\bf X}_a-{\bf X}[\widehat{\theta}_{G}]_{a}
\|_n^2 \right\}+ d_{a}\frac{\log(p)}{n}\right)\ .
\end{equation*}
We base this last simulation study on empty
graphs with  $p=1000$ and $n=100$, in order to evaluate  in practice,
the tendancy of BIC to overfit in a high
dimensional setting.

\subsubsection{{\it Power versus FDR
  curves} when $p=100$}

The number of nodes $p$ and the number of observations $n$ being fixed
to $p=100$, $n=50$, for each of the $N_{G}=20$ simulated graphs, we
estimated the FDR, the power and the MSEP on the basis of
$N_{X}=20$ simulations. These calculations are done for different  values of
the tuning parameter. 
The means over the $N_{G}$ graphs are shown
at Figure~\ref{dess7.fg}. The standard errors of the
means over the $N_{G}$ graphs are smaller than 0.0057 for the FDR,
and 0.018 for the power. \\

\begin{figure}[htbp]
\begin{center}
{\includegraphics[height=12.5cm,width=8cm,angle=270]{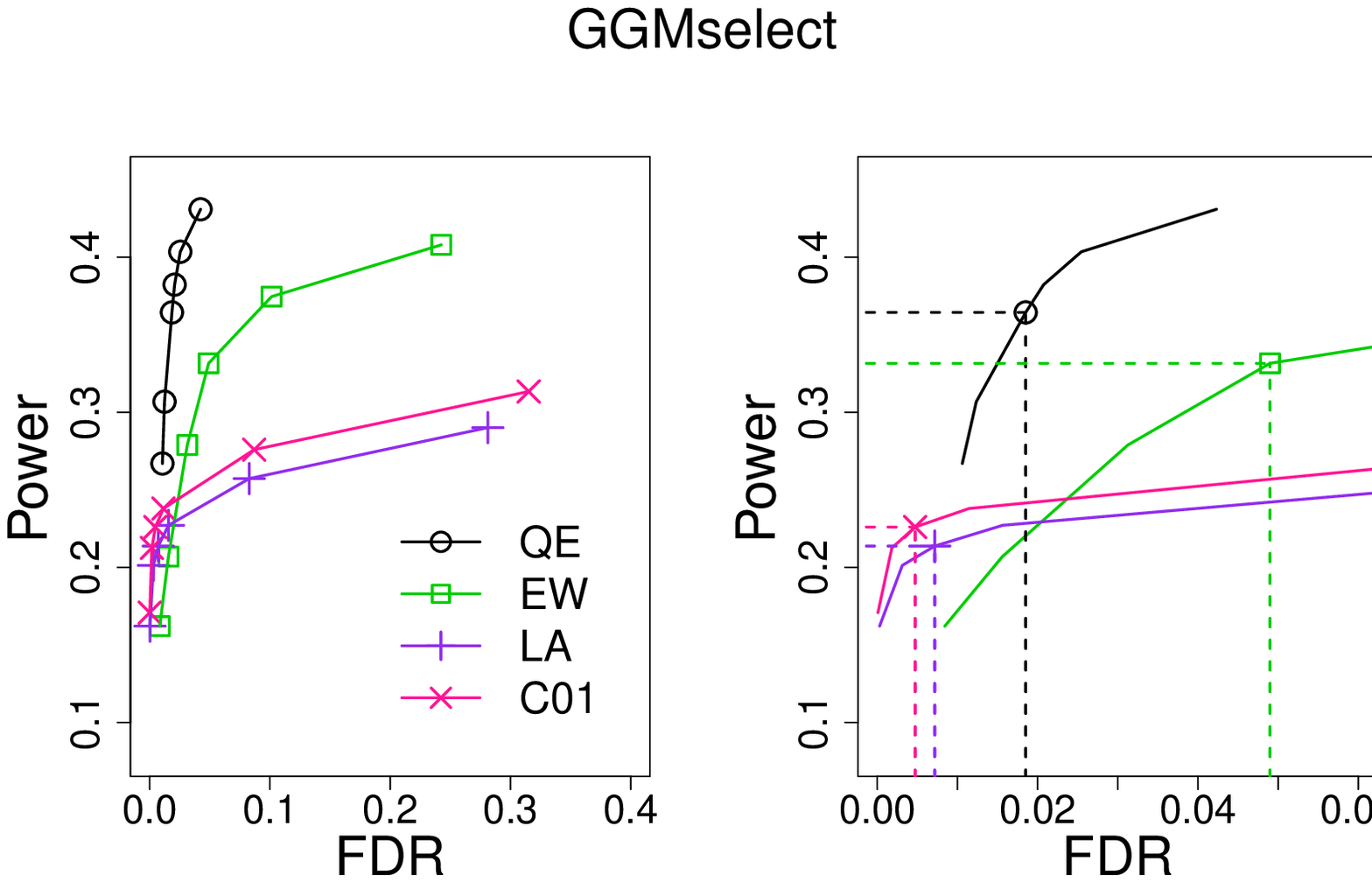}

\includegraphics[height=12.5cm,width=8cm,angle=270]{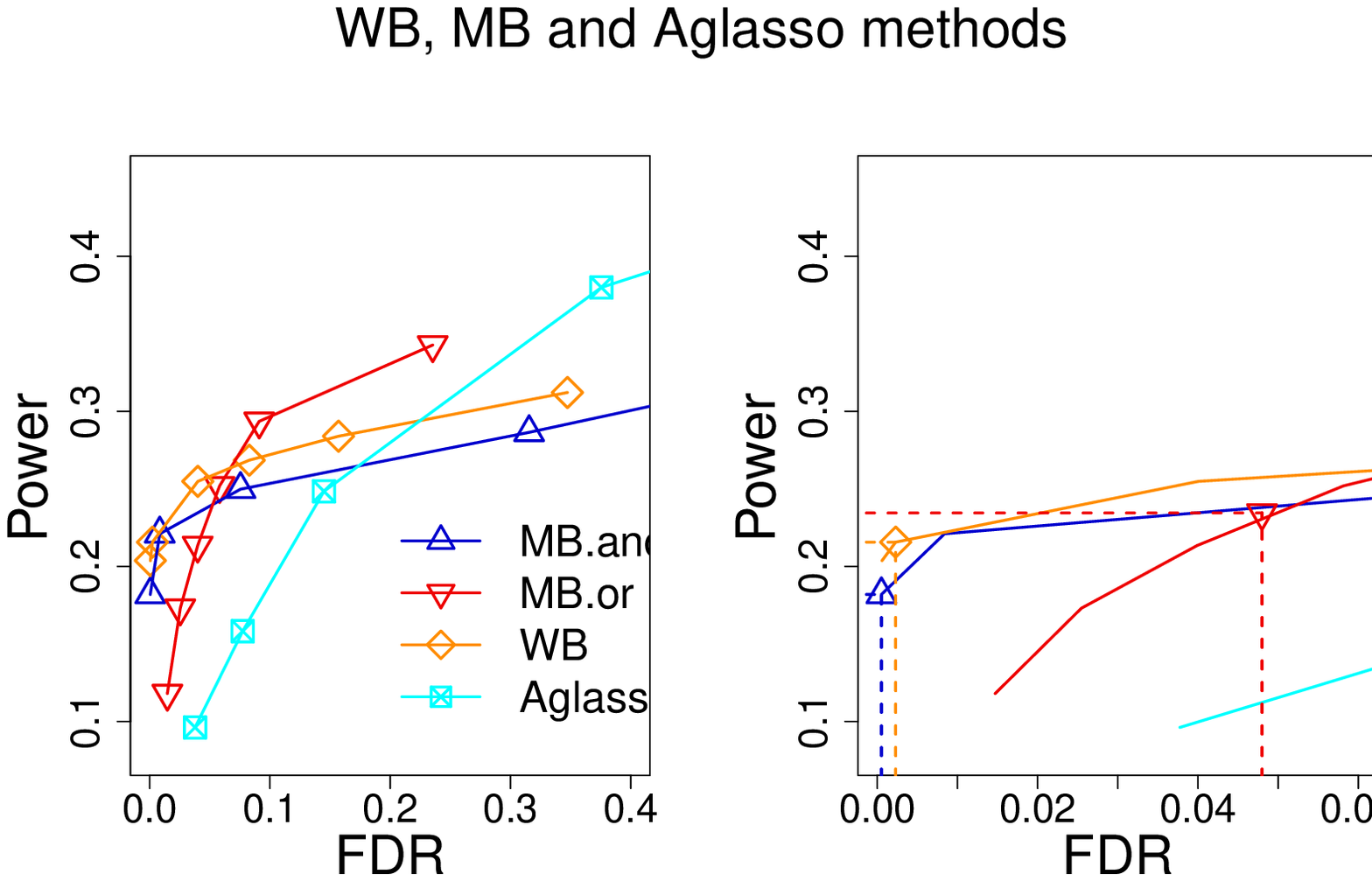}
}
\end{center}
\caption{\label{dess7.fg}Graphics of power versus FDR for the case
  $p=100$, $n=50$ and $I_{s}=3$. The marks on the graphics on the left
  correspond to different values of the tuning parameter. The curves
  for small FDR values are magnified on the graphics on the right. The FDR
  and power values corresponding to the tuning parameter recommended
  in the literature are superimposed on the curves (dashed lines) : $K=2.5$ for
{\tt  GGMselect}, $\alpha=5\%$ for {\tt WB} and {\tt MB} methods. For {\tt
  Aglasso},  with  $\lambda$ chosen by 5-fold cross-validation,  the FDR equals
  0.90 and the power equals 0.59 (not shown).}
\end{figure}


{\bf Choice of the family of candidate graphs in our procedure}.
The {\tt QE} method presents good  performances: the
FDR stays small and the power is high. Though it  was performed with $D=3$,
while  {\tt EW},  {\tt LA} and {\tt C01} were performed with $D=5$,
it works the best. The {\tt EW}
method is more powerful than {\tt LA} and {\tt C01} if one accepts a
FDR greater than 2.5\%. \\

{\bf Comparison with the other methods}.
The procedures  {\tt LA} and  {\tt C01} behave similarly to {\tt WB}
method. The {\tt MB.or} method presents  higher values of the power
when the FDR is larger than 5\%. The {\tt MB.and} keeps down  the
FDR but lacks power. The {\tt Aglasso} method behaves completely in a
different way: the curve stays under the others as long as the FDR is
smaller than 20\%. When the regularization parameter is chosen by
5-fold cross-validation, the power equals  $59\%$ at the price of a very
large FDR equal to 90\% (not shown). In the following we do not consider anymore
the adaptive glasso method, and focus on methods that have a good
control of the FDR. \\

{\bf Results when $p$ is very large face to $n$}.
Keeping $n=50$,
and taking $p=500$, we 
estimated the FDR and the power for all methods except the {\tt EW}
method for which the computing time is too large for carrying
out a simulation study. The results are given at
Figure~\ref{dess83.fg}. The method {\tt QE} was performed with $D=2$,
while the {\tt LA} and {\tt C01} were performed with $D=5$. As
expected, for all methods,  the power is lower for $p=500$
than for $p=100$. The  between procedures comparison stay the same.

\begin{figure}[htbp]
\begin{center}
{\includegraphics[height=12.5cm,width=8cm,angle=270]{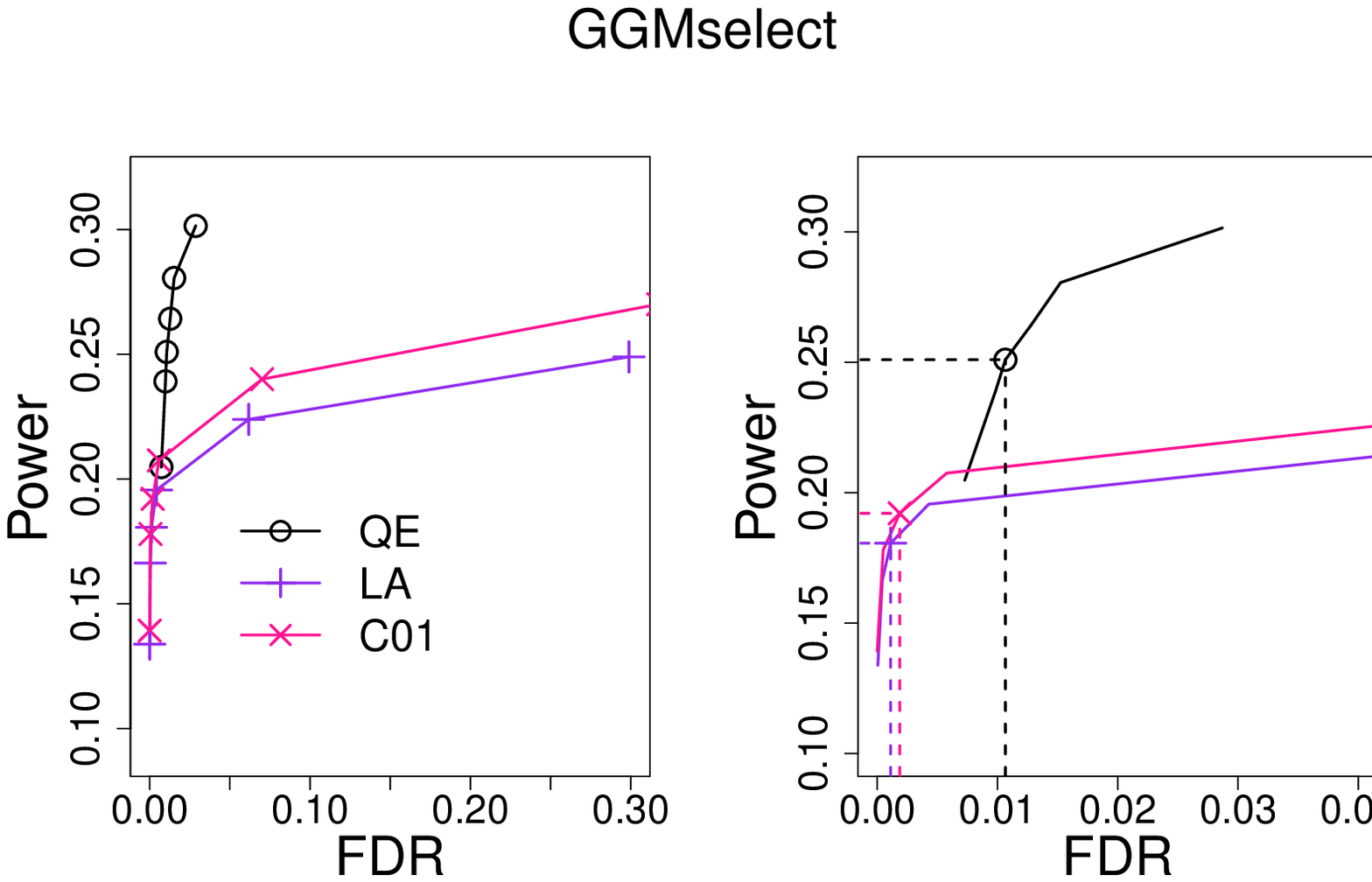}

\includegraphics[height=12.5cm,width=8cm,angle=270]{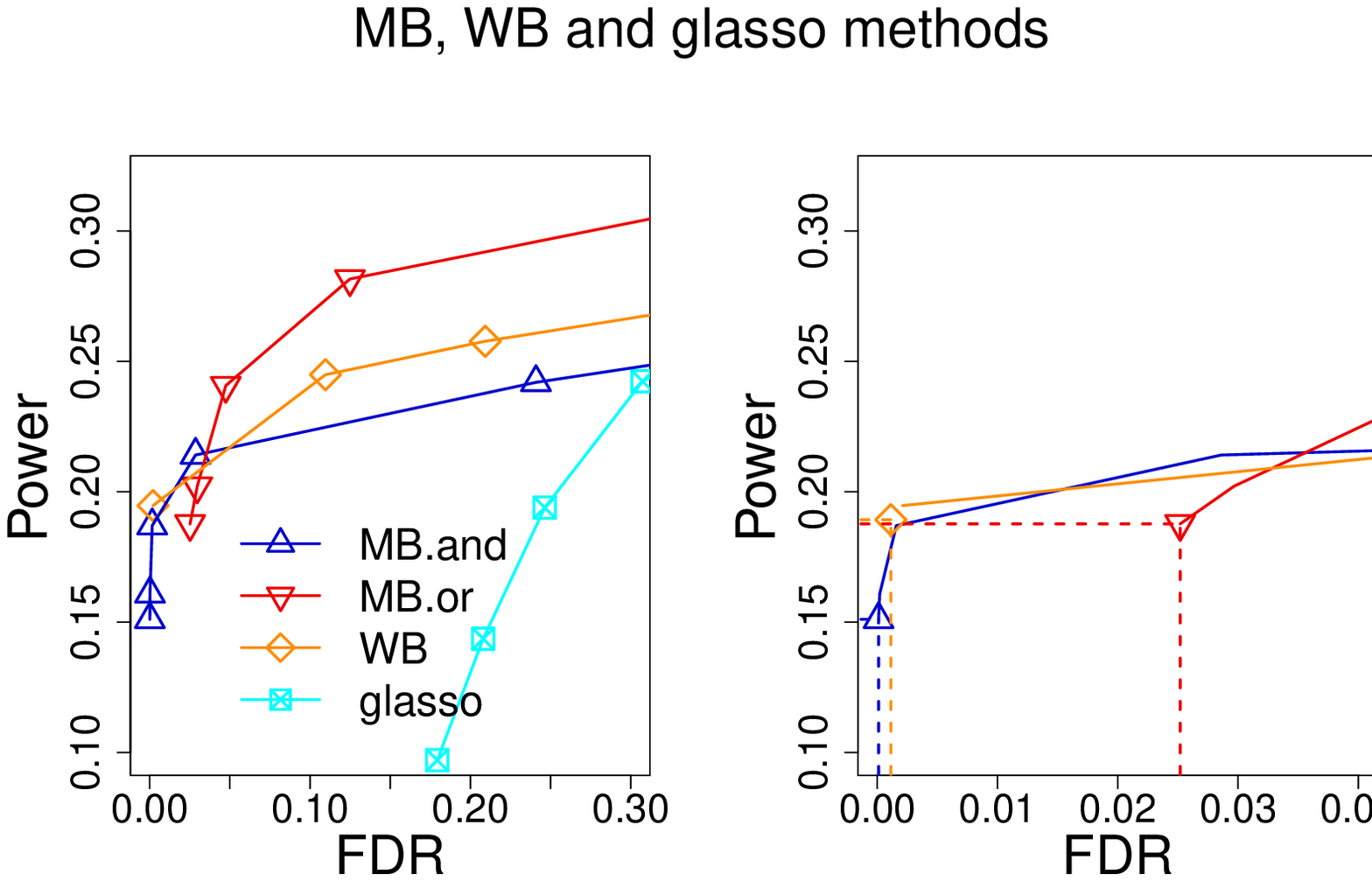}
}
\end{center}
\caption{\label{dess83.fg}Graphics of power versus FDR for the case
  $p=500$, $n=50$ and $I_{s}=3$. The marks on the graphics on the left
  correspond to different values of the tuning parameter. The curves
  for small FDR values are magnified on the graphics on the right. The FDR
  and power values corresponding to the tuning parameter recommended
  in the literature are superimposed on the curves (dashed lines) : $K=2.5$ for
{\tt  GGMselect}, $\alpha=5\%$ for {\tt WB} and {\tt MB} methods.}
\end{figure}

\subsubsection{Effect of the number of observations $n$}

Keeping $p=100$ and $I_{s}=3$, the variations of the FDR and power
values versus the number of observations, are shown in
Figure~\ref{dess4.fg}. The {\tt QE} 
method is applied with $D=3$ while {\tt EW}, {\tt LA} and {\tt C01} are
applied with $D=5$. For all methods the power
increases with $n$ while the FDR decreases  for {\tt EW} and increases
for {\tt MB.or}, {\tt LA} and {\tt C01}. {\tt QE} and {\tt EW} are the
most powerful. When $n$ is small, the {\tt QE}
method stays more powerful than {\tt EW} in spite of a smaller  $D$.


\begin{figure}[htbp]
\begin{center}
{\includegraphics[height=12.5cm,width=8cm,angle=270]{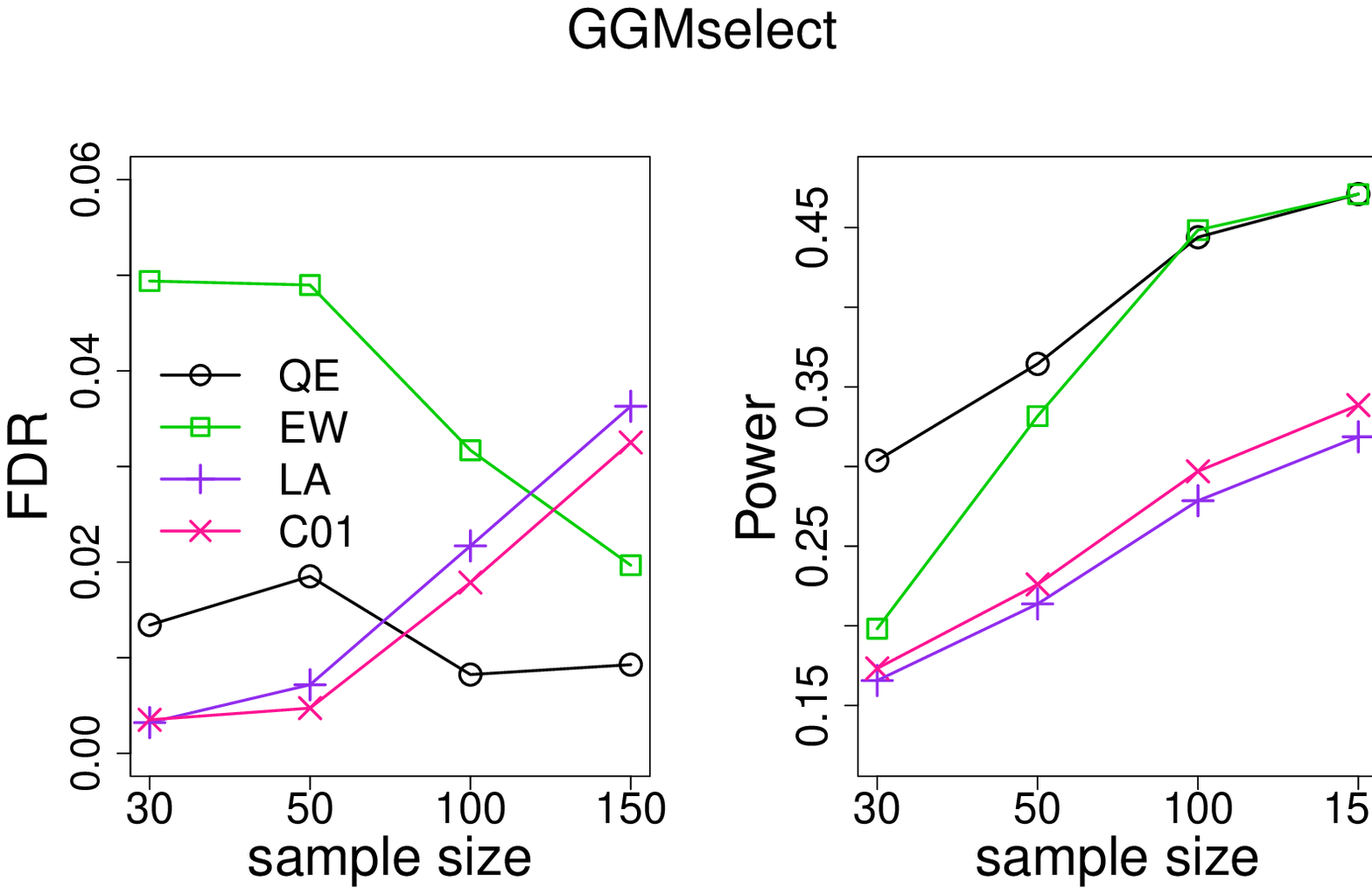}

\includegraphics[height=12.5cm,width=8cm,angle=270]{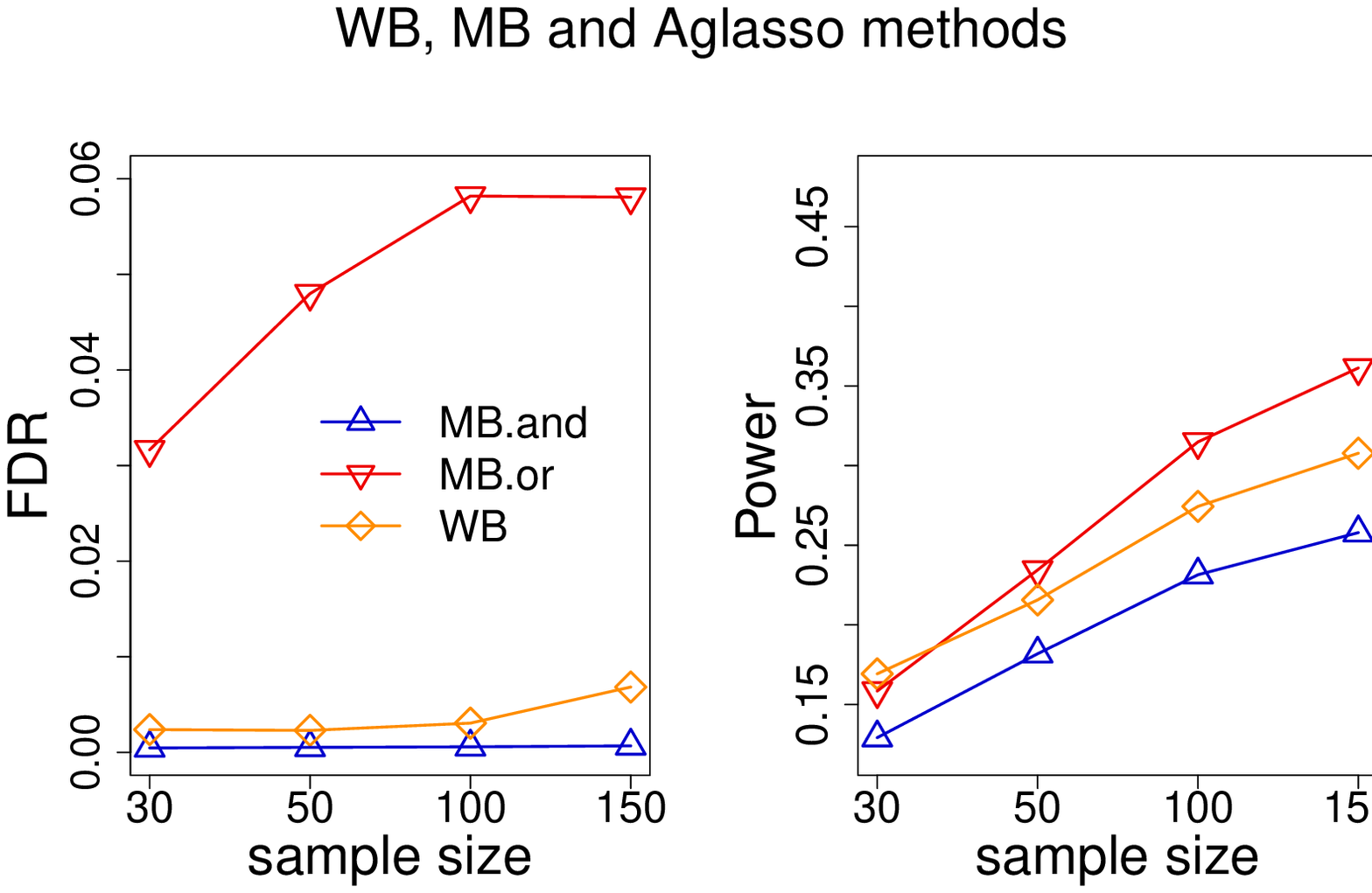}}
\end{center}
\caption{\label{dess4.fg}FDR and power
  estimated values as a function  of $n$ for $p=100$ and $I_{s}=3$. The results
  are calculated on the basis 
  of $N_{G}=20$ simulated graphs and $N_{X}=20$ runs of matrices $\X$ for each
  simulated graph. Our procedures were carried out with $K=2.5$. The
  value of $D$ was  equal to 3 for the {\tt QE} method and 5 for the
  others. For the procedures {\tt MB.or}, {\tt MB.and} and {\tt WB}
  the   tuning parameter $\alpha$ was taken equal to $5\%$.}
\end{figure}

\subsubsection{Effect of graph sparsity}

We have seen that when $p$ is large,  the GGMselect procedures using the
graphs families {\tt QE} and {\tt EW} are powerful and have a good
control of the FDR. Nevertheless, the simulated graphs were sparse,
$I_{s}=3$, and it may be worthwhile testing how the methods perform
when the graph sparsity varies. Because the performances depend
strongly on the simulated graph, the FDR and power are estimated on
the basis of a 
large number of simulations: the number of simulated graphs $N_{G}$
equals 50 and the number of simulated matrices $\X$ for each graph,
$N_{X}$ equals 50. In order to keep  reasonable computing times, we
choose $p=30$. The results are shown
in Figure~\ref{dess1.fg}. The standard errors of the
means over the $N_{G}$ graphs are smaller than 0.0055 for the FDR,
and 0.025 for the power. 


\begin{figure}[htbp]
\begin{center}
{\includegraphics[height=12.5cm,width=8cm,angle=270]{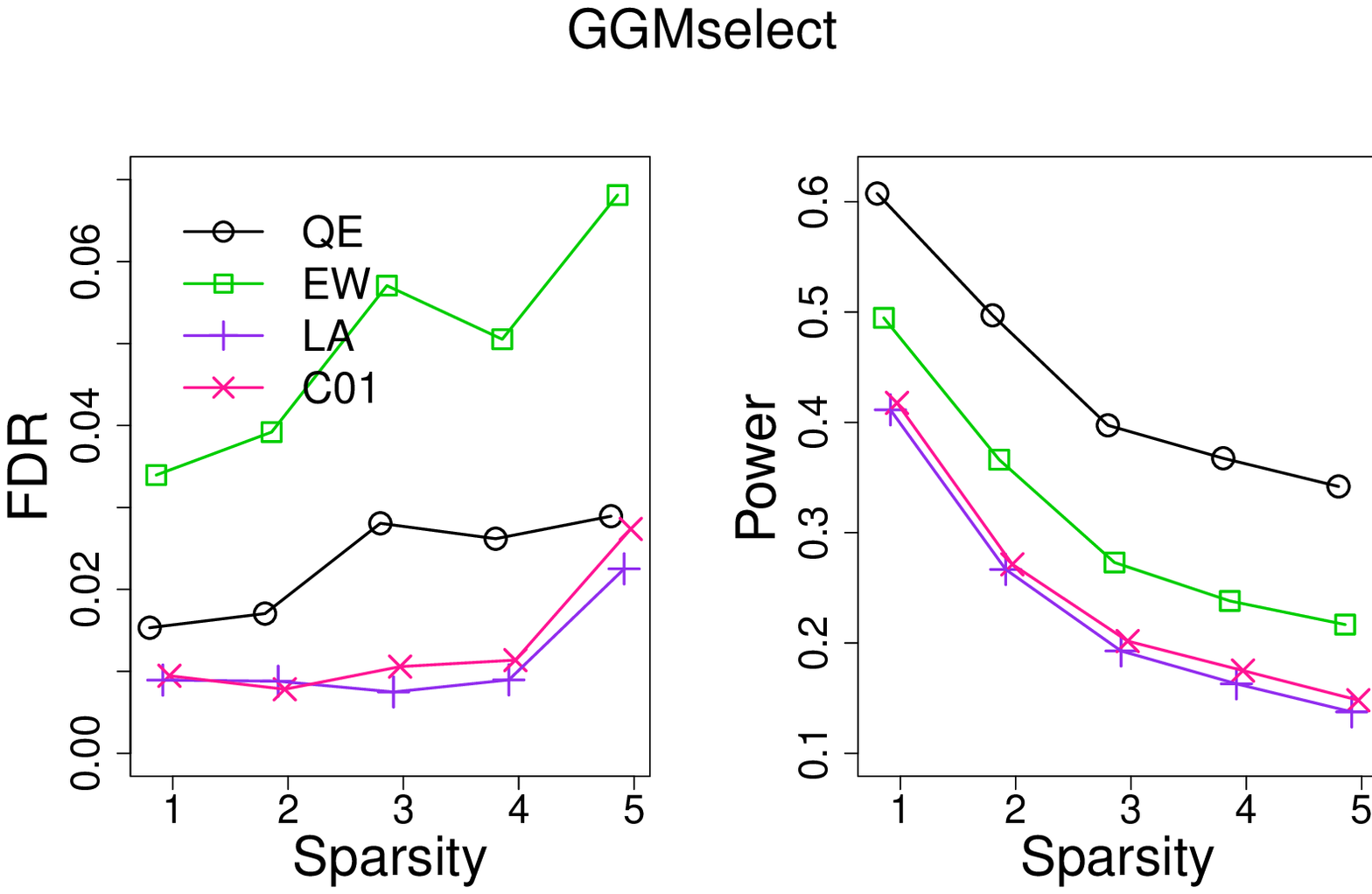}

\includegraphics[height=12.5cm,width=8cm,angle=270]{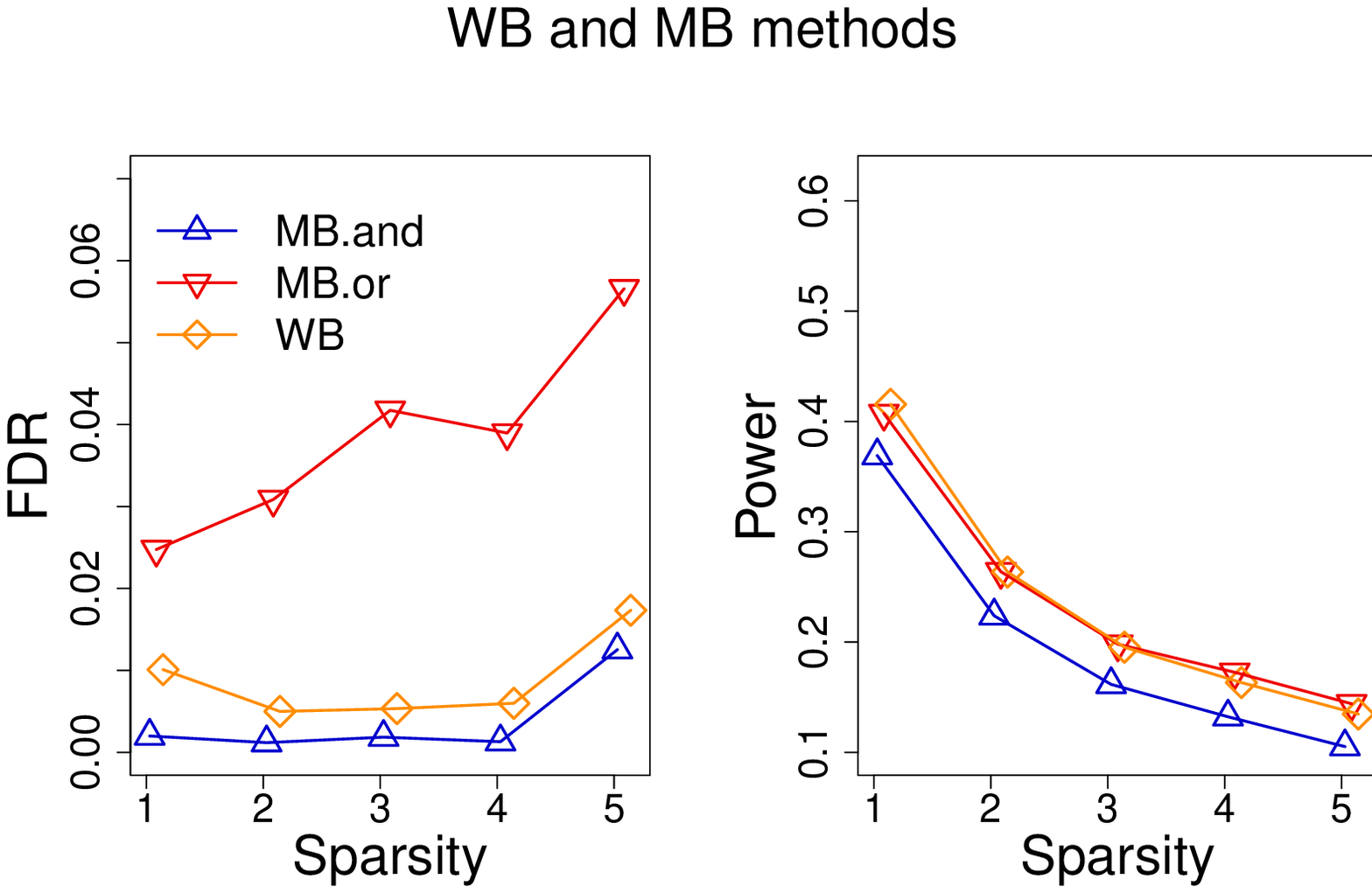}
}
\end{center}
\caption{\label{dess1.fg} Graphs of FDR and power
  estimated values versus the graph sparsity $I_{s}$, for
  $p=30$ and $n=30$. The results are
  calculated on the basis 
  of $N_{G}=50$ simulated graphs and $N_{X}=50$ runs of matrices $\X$ for each
  simulated graph. Our procedures were carried out with $K=2.5$ and
  $D=5$. For the procedures {\tt MB.or}, {\tt MB.and} and {\tt WB} the
  tuning parameter $\alpha$ was taken equal to $5\%$.}
\end{figure}

For all methods the power decreases when $I_{s}$ increases. The FDR
values are slightly increasing with $I_{s}$ for the {\tt EW} and {\tt
  MB.or} methods. The superiority of {\tt QE} over the others is
clear. {\tt EW} is more powerful then {\tt LA}, {\tt C01}, {\tt MB}
and {\tt WB} methods but its FDR is greater. 

\subsubsection{GGMselect : mixing the graphs families}

Our procedure allows to mix several graphs families. It may happen
that some graphs, or type of graphs, are known to be good candidates
for modelling the observed data set. In that case, they can be
considered in the procedure, and thus compete with
$\widehat{\mathcal{G}}_{\EW}$ or
$\widehat{\mathcal{G}}_{\QE}$. This can be done with the function {\tt
selectMyFam} of the package {\tt GGMselect}.

Considering the results of our
simulation study, we could ask if mixing $\widehat{\mathcal{G}}_{\LA}$
or  $\widehat{\mathcal{G}}_{\CO 1}$ with $\widehat{\mathcal{G}}_{\EW}$
would not give a better control of the FDR than {\tt EW} while keeping
high values of the power. To answer this question we carried out
simulation studies taking
$\widehat{\mathcal{G}}_{{\tt mix}} = \widehat{\mathcal{G}}_{\CO 1} \cup
\widehat{\mathcal{G}}_{\LA} \cup \widehat{\mathcal{G}}_{\EW}$  
as the family  of graphs. In all considered cases for $p$, $n$,
$I_{s}$, the FDR and power values  based on $\widehat{\mathcal{G}}_{{\tt
      mix}}$ are 
similar to those based on $\widehat{\mathcal{G}}_{\EW}$. This result
can be explained by studying the behavior of the MSEP  estimated  by
averaging the quantities $\|\Sigma^{1/2} 
(\widehat{\theta}_{\widehat{G}} - \theta)\|^{2}$  over
the $N_{G} \times N_{X}$ runs.  The results are given at
Figure~\ref{dess11.fg}. One can see that the smallest values of the MSEP are
obtained for {\tt QE}, then {\tt EW}. Moreover,  the MSEP decreases when the
power increases, while it does not show any particular tendency when
the FDR varies. Considering these tendencies together with the fact
that our procedure aims at
minimizing the MSEP, we can understand why we do not improve the
performances of {\tt EW} by considering $\widehat{\mathcal{G}}_{{\tt
      mix}}$. 

\begin{figure}[htbp]
\begin{center}
{\includegraphics[height=14cm,width=6cm,angle=270]{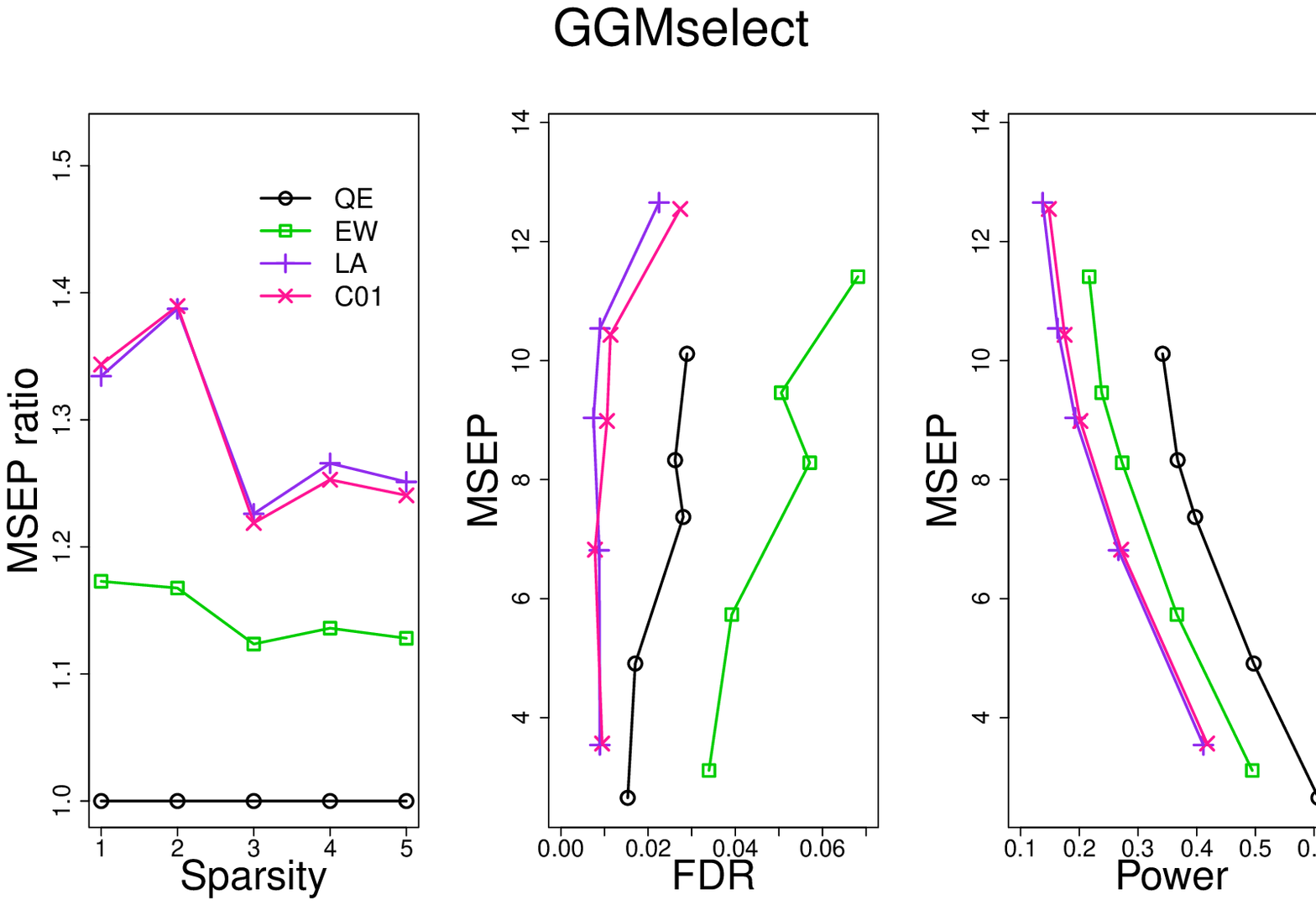}

\includegraphics[height=14cm,width=6cm,angle=270]{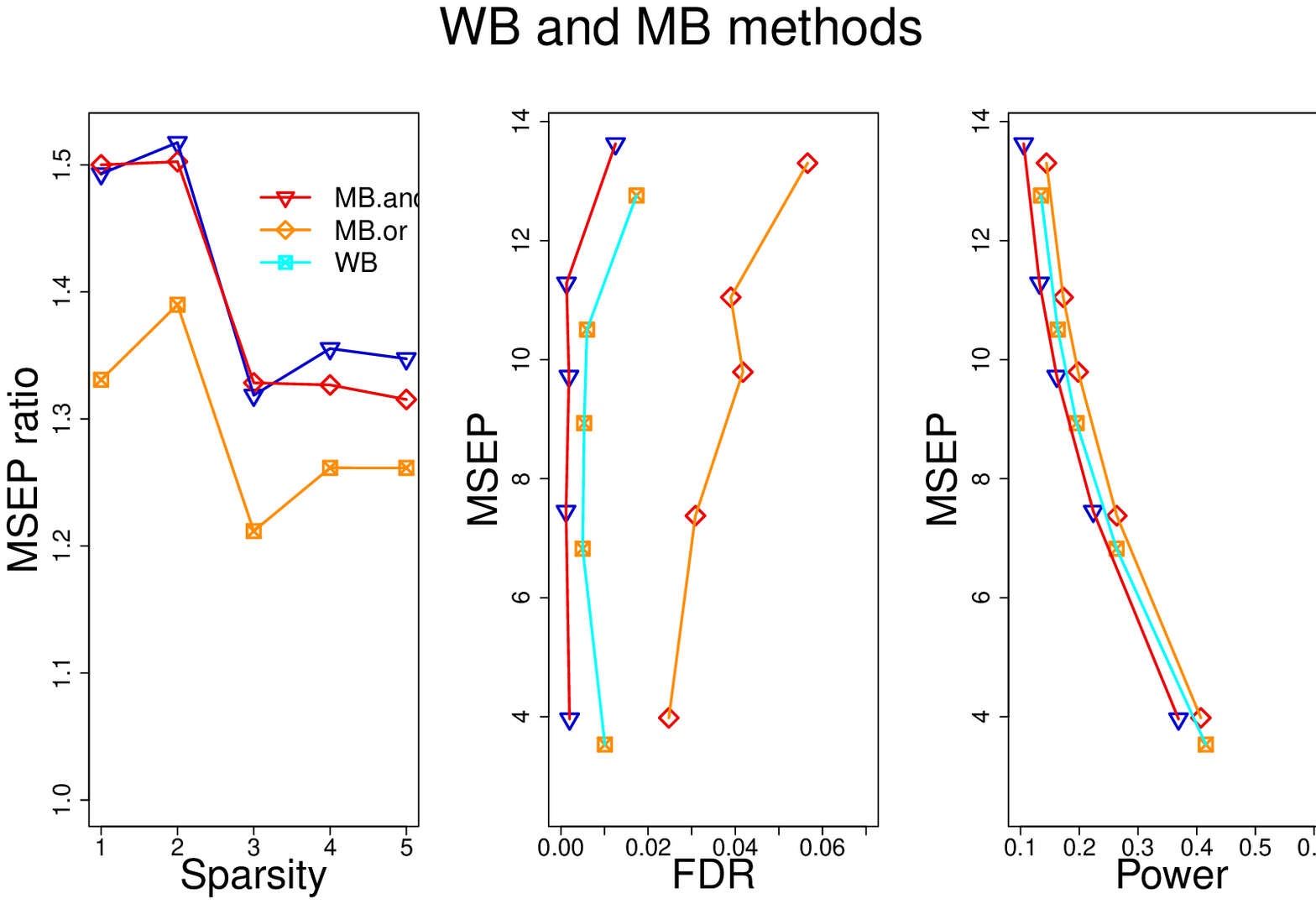}}
\end{center}
\caption{\label{dess11.fg}Values of the MSEP  for
  the simulation results given at Figure~\ref{dess1.fg}. The first
  graphic on the left presents the ratio of the MSEP over the MSEP
  of the {\tt QE} method. The two others present the MSEP
  versus the FDR and the power.}
\end{figure}

\subsubsection{Comparison with the BIC criteria}
As it was shown in the fixed design regression model \cite{BGH09}, the BIC
criterion overfits in a high-dimensional setting. To compare BIC
with our procedure for estimating an empty graph, we simulate
$N_{X}$ matrices $\X$ composed of $n=100$ i.i.d. rows distributed as
$\mathcal{N}_{p}(0, I_{p})$, with $p=1000$.  We consider the
collection of graphs given by the 
procedure LA, with $D=7$, and choose among this collection using our criterion,
and the BIC criterion. 
The mean of
the number of false positive, calculated on the basis of 100
simulations, equals 0 for our procedure, and equals 1077 when applying
the BIC procedure. This confirms that BIC should not be used for such
problems.

\subsection{Summary}

We recommend to use the {\tt QE} method if the calculation of
$\widehat{G}_{K,\text{and}}$ and
$\widehat{G}_{K,\text{or}}$ is possible. Next, working out the
family  $\widehat{\mathcal{G}}_{\QE}$ can always be done using
some suitable  algorithms if necessary (as a stepwise procedure for
example). When $p$ is large, {\tt QE}  can be used for small values of
$D$ ($D=3$ or even $D=2$). It may perform better than all the
others when $n$ is small.  The procedure 
based on $\widehat\G_{\EW}$  can be used for large $p$: the gain in
power over {\tt LA}, {\tt C01}, {\tt MB} and  {\tt WB} methods is
significant, but the FDR is slightly greater.  
  The {\tt LA} and {\tt C01} methods are running very quickly, keep
  the FDR under control and 
  are slightly more powerful than {\tt WB} and {\tt MB.and}.

\section{\label{section_ex}Breast cancer data}

We test our procedure on a gene expression data set provided in Hess et al.~\cite{Hess}. The data set
concerns 133 patients with breast cancer treated with
chemotherapy. The patient response to chemotherapy  can be classified
 into two groups according
to a pathologic complete response (PCR) or residual disease
(NotPCR). Natowicz et al.~\cite{Nato} selected 26 genes having a high
predictive value for this response. We propose to estimate possible
regulation dependencies between these 26 genes, for each group of
patients : 34 patients achieved PCR, while 99 did not. 

This data set was already considered by Ambroise et al.~\cite{Ambroise} who proposed a method to infer a Gaussian
Graphical Model taking into account some hidden structure on the
nodes. They simultaneously infer the nodes groups and the graph using
an $l^{1}$ penalized likelihood criterion. Their method is
performed in an iterative EM-like algorithm, names SIMoNe, available
in an R-package~\cite{Chiquet}. 

We apply our procedure for choosing among the graphs coming from the
families {\tt QE}, {\tt LA}, {\tt C01}, {\tt EW} and from the family
of graphs proposed by the SIMoNe algorithm. 

We only present results for the group of patients not achieving PCR. The chosen graph presents 14
edges. The minimum value of the
criteria equals 686.64 and is achieved for the {\tt QE} family. Let us
assess the stability of the results between the different methods, and
the  stability when the constant $K$ in our procedure is varying.\\

{\bf Stability  between the different methods}.
Let us  have a look at graphs that minimize the
criteria for each family considered. The results are given at
Figure~\ref{fig.cancer}. Firstly, let us note that for this data set,
the SIMoNe algorithm gives results similar to the {\tt LA} method.  Secondly we
remark that some 
characteristics of the {\sl best} graph are shared by the others, as
for example the path between KIA1467, GAMT, E2F3, MELK and RRM2. This
 allows to be  confident in that motif. All methods allocate
edges between genes ZNF552, FLJ10916, JMJD2B, BECNI, PDGFRA, but the
motifs connecting these genes differ between methods. This suggests
that these genes are probably linked, but the estimation of the motif
is not completely secure. If the motif is an hub centered in JMJD2B, as
it is shown by the {\tt QE} method, then instability in estimating
this motif is not surprising: it is more difficult to estimate the
neighbours of highly connected nodes. \\

\begin{figure}[htbp]
\begin{center}
\hspace{-1cm}\includegraphics[height=9cm,width=6cm,angle=270]{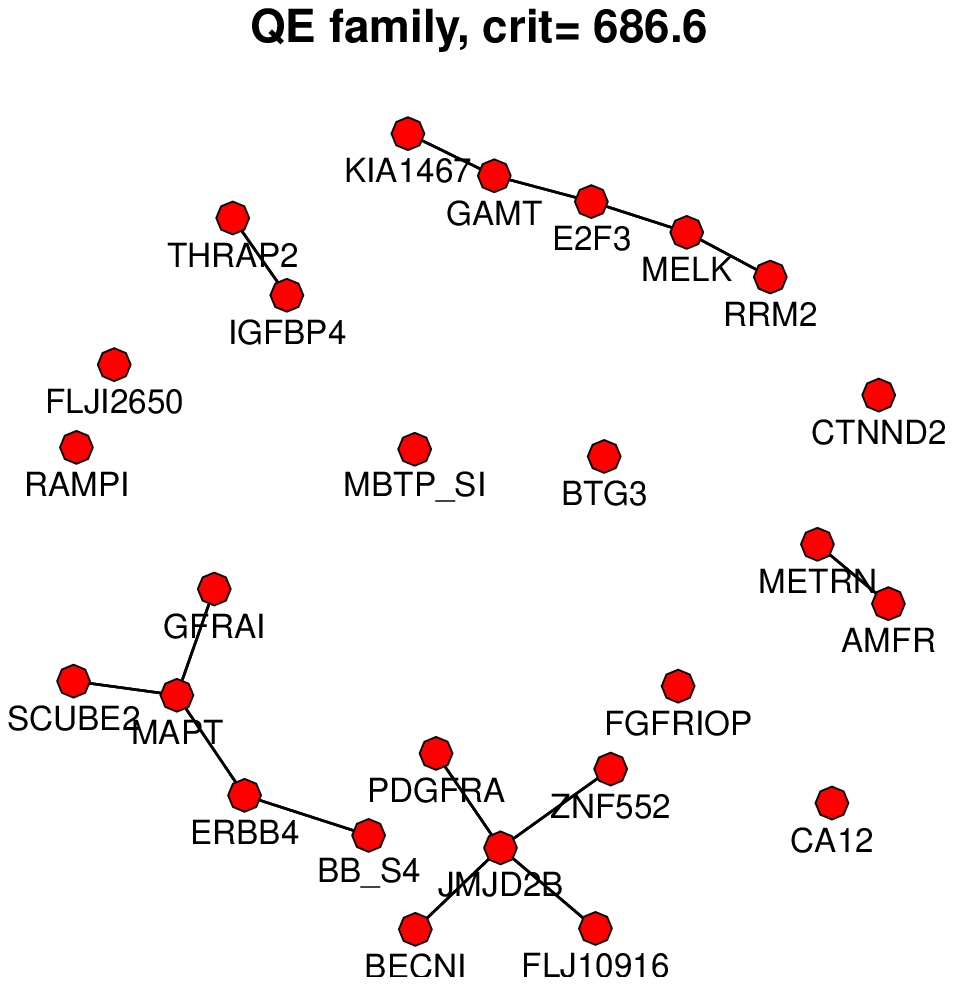}~\hspace{-3cm}\includegraphics[height=9cm,width=6cm,angle=270]{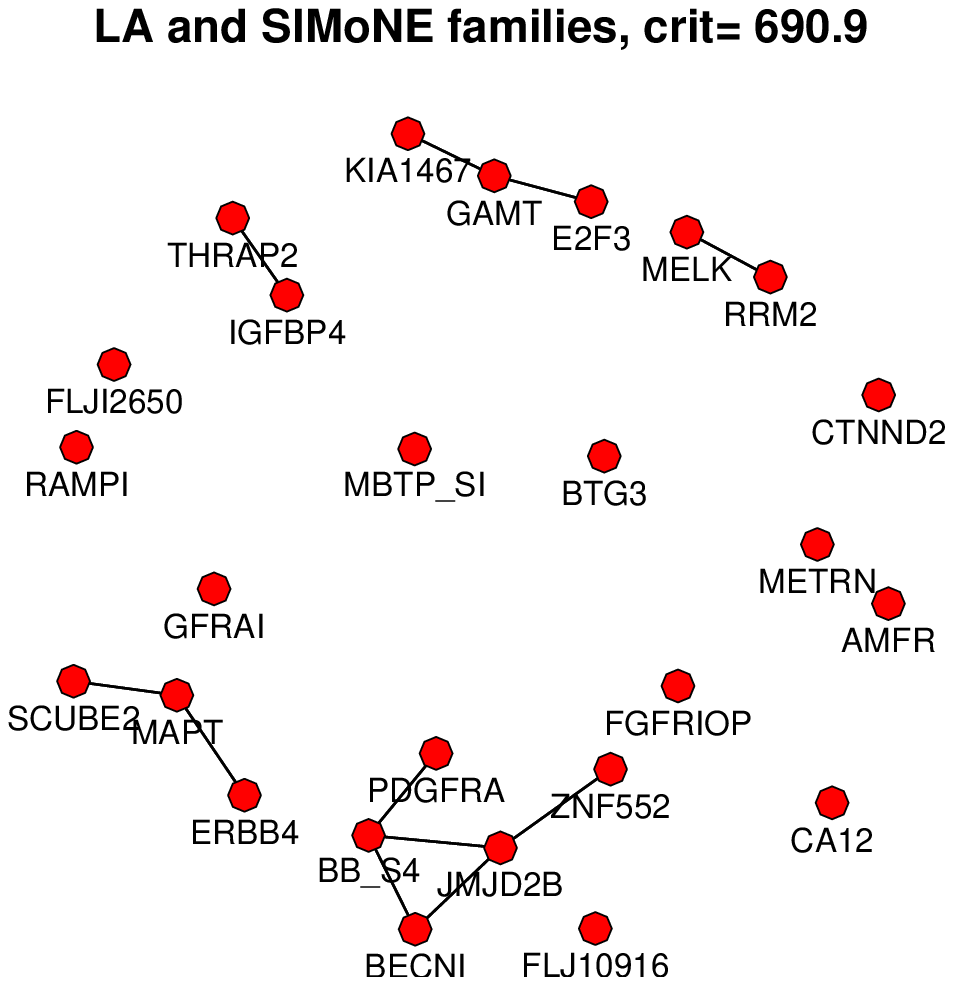}

\includegraphics[height=9cm,width=6cm,angle=270]{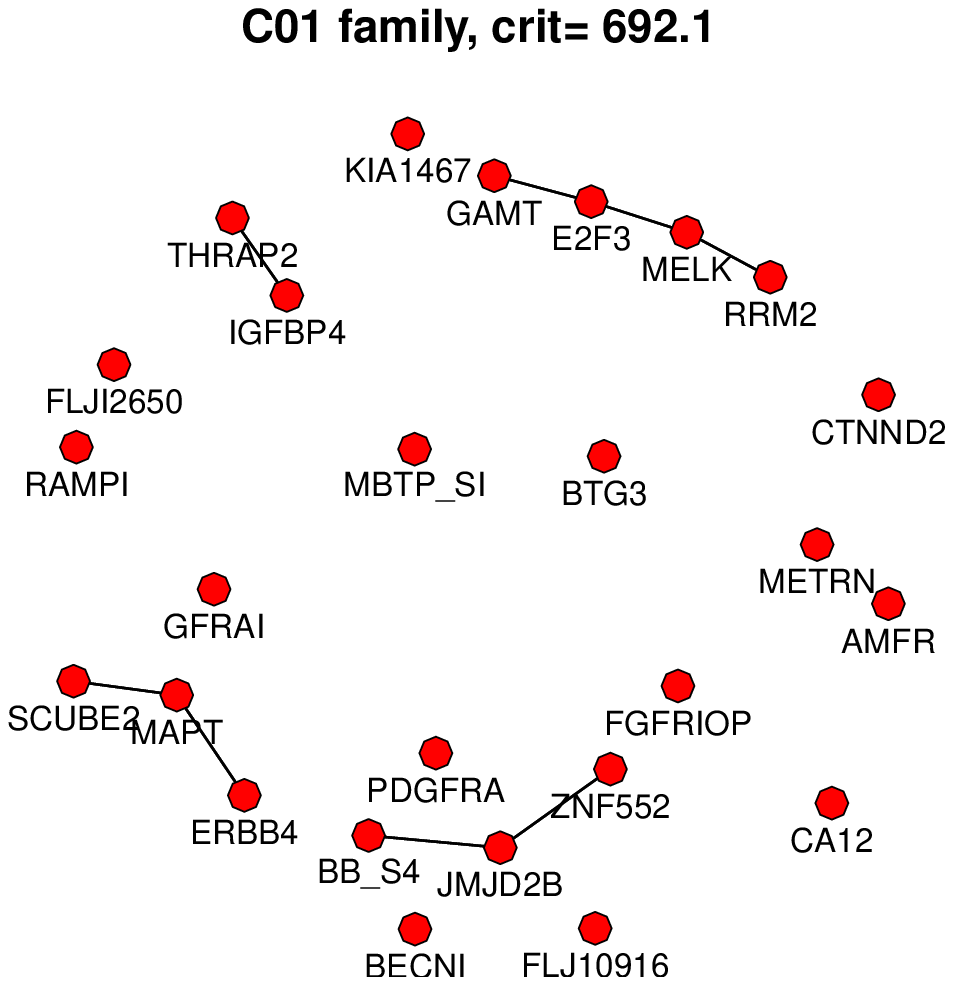}~\hspace{-3cm}\includegraphics[height=9cm,width=6cm,angle=270]{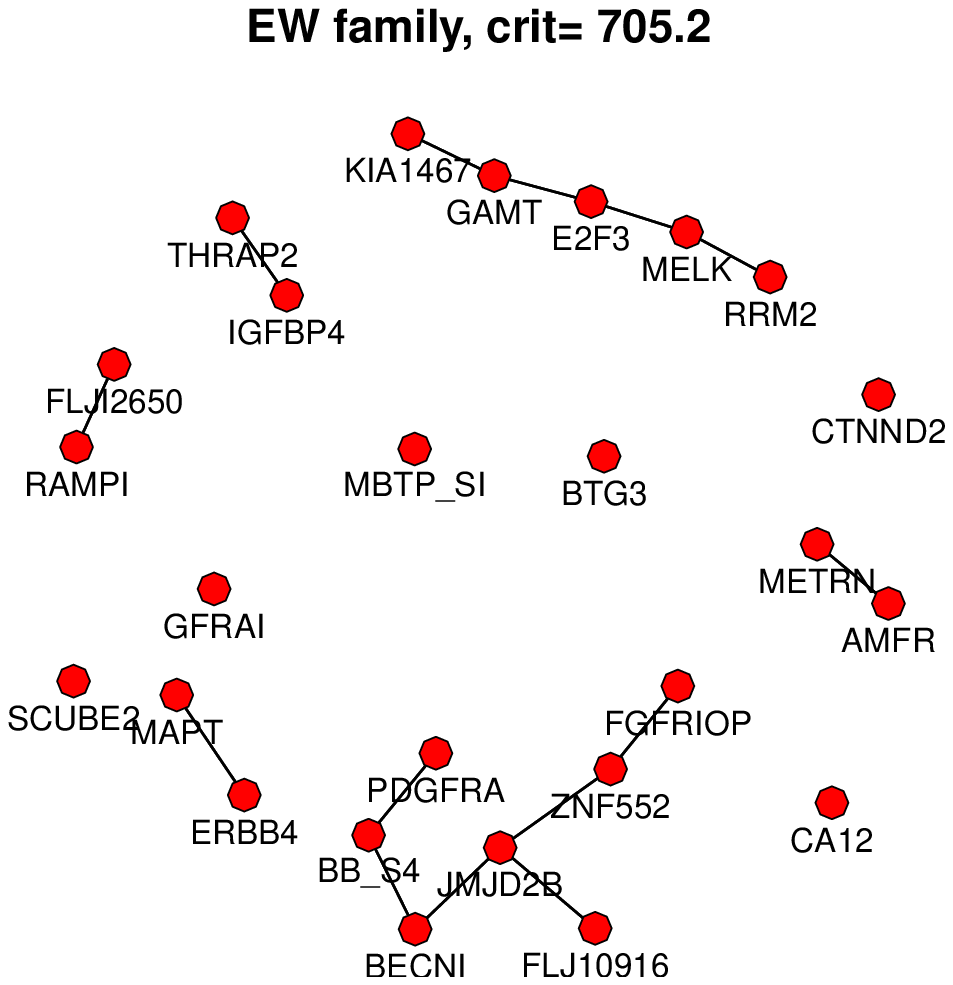}
\caption{\label{fig.cancer} For each considered family, criterion (\ref{definition_critere}) value
  and estimated graph for the group of patients
  with residual disease. The graph minimizing
  the criterion (\ref{definition_critere}) is given by the {\tt QE} method.}
\end{center}
\end{figure}

{\bf Stability  when the constant $K$  varies}.
The estimated graphs based on the {\tt LA} family when $K$ varies from
1.5 to 3 are presented at Figure~\ref{fig.cancerKvaries}. Increasing
$K$ to 3 leads to delete two edges. The difference between $K=2$ and
$K=2.5$ is more important: some of the edges that were detected by the {\tt
 QE} method with $K=2.5$ are detected by the {\tt LA} method with
$K=2$. There is no difference between the estimated graphs using 
$K=1.5$ or $K=2$. This suggests that all potentially detectable edges
with the {\tt LA} method are detected with $K=2$. 

\begin{figure}[htbp]
\begin{center}
\includegraphics[height=9cm,width=6cm,angle=270]{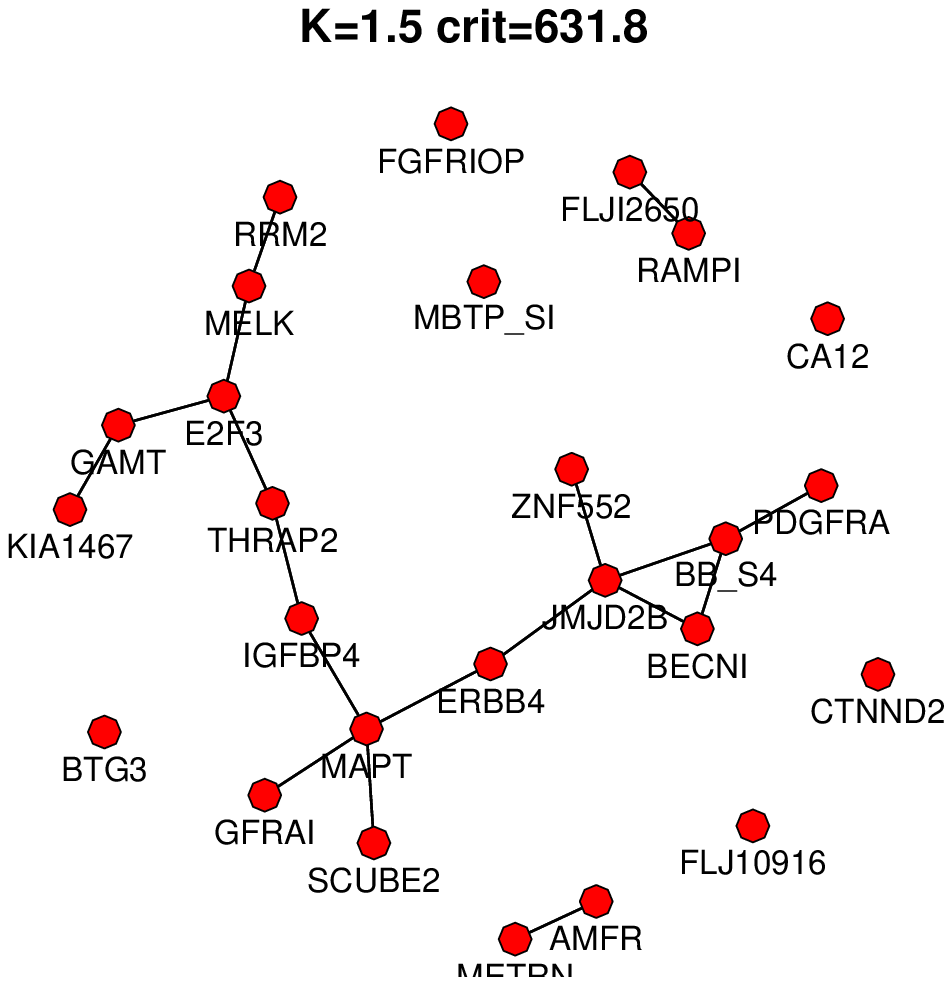}~\hspace{-3cm}\includegraphics[height=9cm,width=6cm,angle=270]{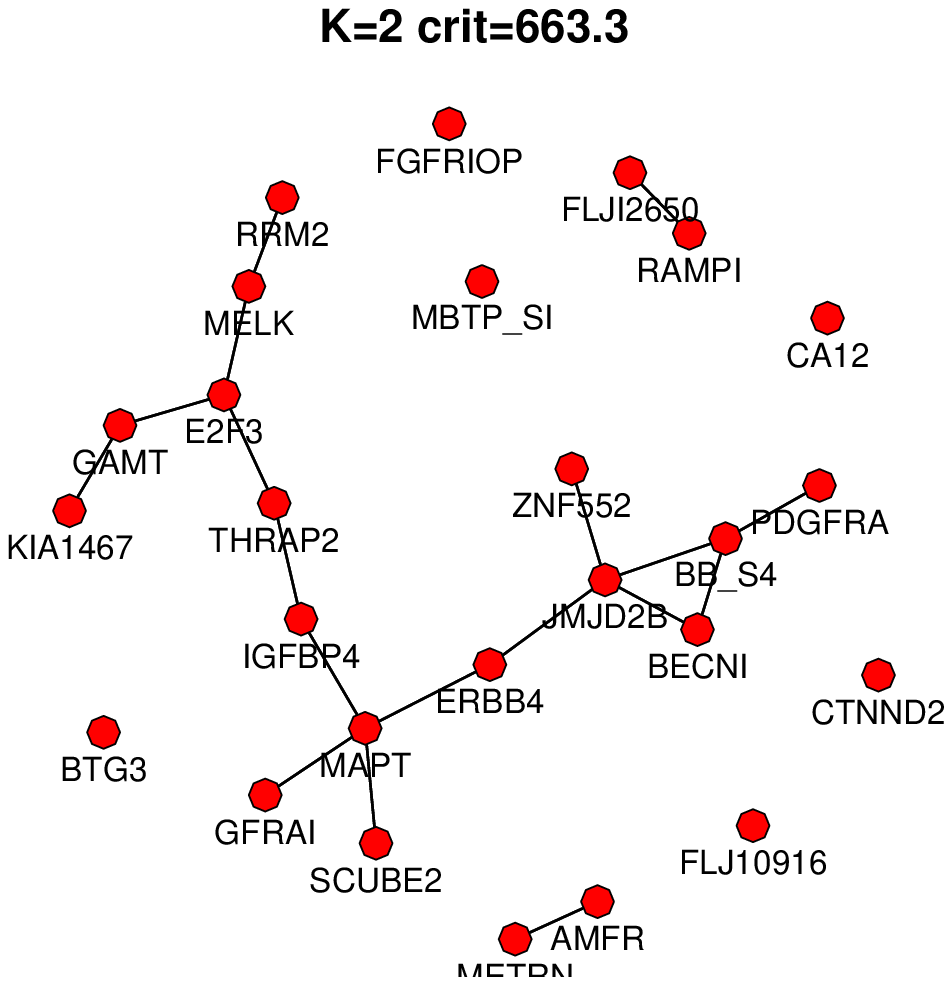}

\includegraphics[height=9cm,width=6cm,angle=270]{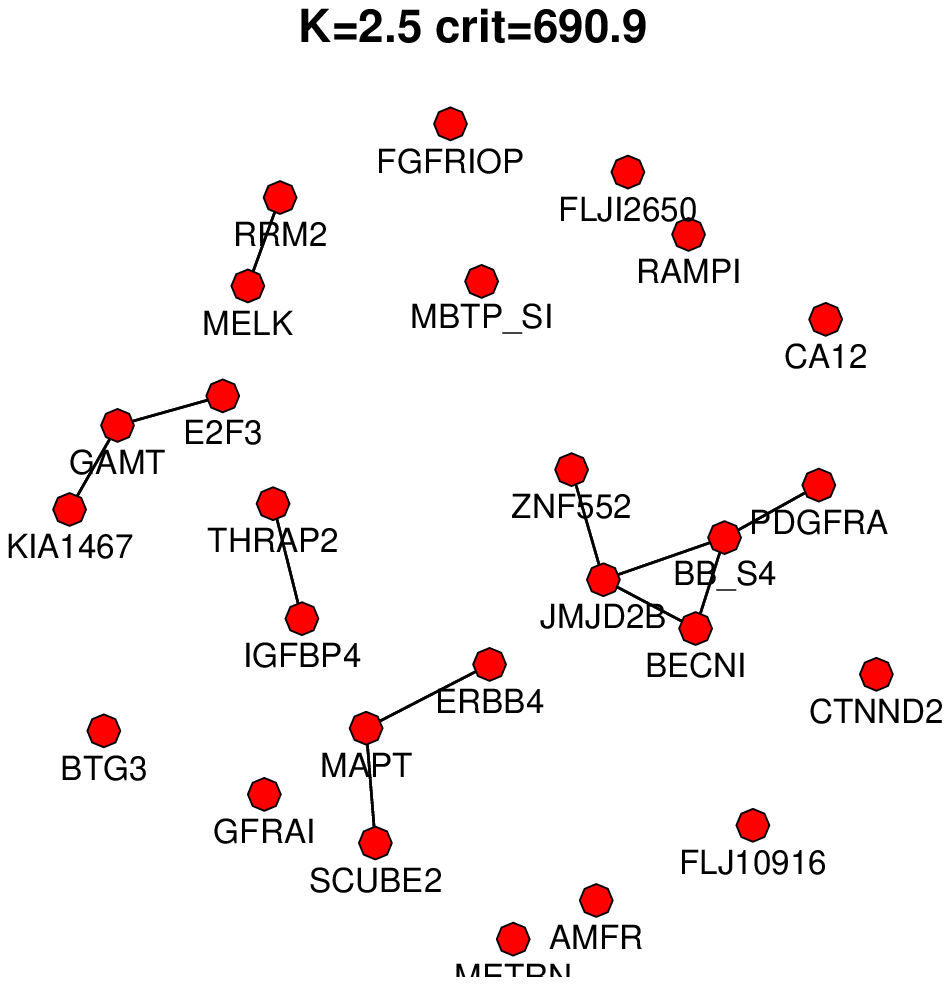}~\hspace{-3cm}\includegraphics[height=9cm,width=6cm,angle=270]{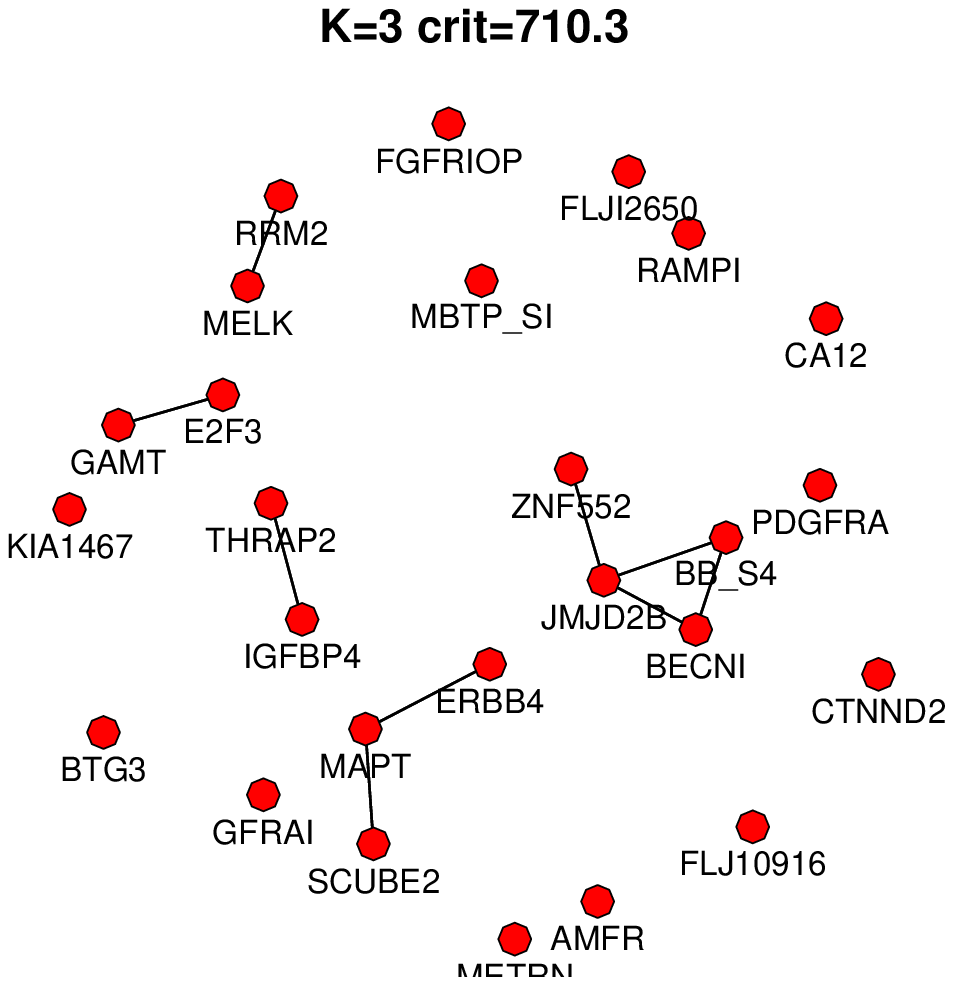}
\caption{\label{fig.cancerKvaries} For the family {\tt LA},  criteria value
  and estimated graph when the constant $K$ varies.}
\end{center}
\end{figure}

\section{Proofs}\label{section_preuves}

In the sequel, $L$, $L_1$, $L_2$,$\ldots$ denote universal constants that may vary from line to line. The notation $L(.)$ specifies the dependency on some quantities. 

\subsection{Proof of Theorem \ref{proposition_risque}}
We write $\G_{D}$ for the family of all the graph with nodes in $\Gamma$ and
degree less than $D$. We remind the reader that 
for any graph $G\in\G_{D}$ we have noted $\Theta_{G}$ the space of $p\times p$
matrices $\theta$ such that $\theta_{a,b}$ is non zero if and only if there is
en edge between $a$ and $b$ in $G$.  We also set
$\bar\Theta_{D_{\max}}=\cup_{G\in\G_{D_{\max}}}\Theta_{G}$. 
norm $\|\cdot\|_{q\times p}$ on $q\times p$ matrices.
We set $\lambda=(1-\sqrt{\gamma})^2$ and introduce the event
$$\B=\ac{\lambda\|\Sigma^{1/2}A\|_{p \times p}\leq {1\over \sqrt n}\|\X
A\|_{n\times p}\leq \lambda^{-1}\|\Sigma^{1/2}A\|_{p \times p},\ \textrm{for all
}A\in\theta+\bar\Theta_{D_{\max}}}.$$
On this event we can control the $l^2$-loss of $\thetat$ by the empirical loss
since 
\begin{equation}\label{borne1}
\|\Sigma^{1/2}(\thetat-\theta)\|^2_{p\times p}\1_{\B}\leq {\lambda^{-2}\over
n}\|\X(\thetat-\theta)\|^2_{n \times p}\1_{\B}\,.
\end{equation}
Moreover, according to Lemma 1 in \cite{giraud08},  we have $\P(\B^c)\leq
2e^{-n(\sqrt{\gamma}-\gamma)^2/2}$ when Condition~(\ref{condition_degree}) is
met. To bound the risk of the procedure, we consider apart the events $\B$ and
$\B^c$.

\subsubsection{Bound on $\E\cro{\|\Sigma^{1/2}(\thetat-\theta)\|^2_{p \times
p}\1_{\B}}$}
We have $\X=\X\theta+\eps$, where $\eps$ is a $n\times p$  matrix  distributed
as follows: for each $a\in\Gamma$, the column $\eps_{a}$ is independent of
$\X_{-a}$ and is distributed according to the Gaussian law
$\mathcal{N}(0,\sigma^2_{a}I_{n})$, with $\sigma^2_{a}=1/\Omega_{a,a}$. For any
$G\in\G_{D}$, we write henceforth $\theta^{G}$ for the orthogonal projection of
$\theta$ on $\Theta_{G}$ according to the Euclidean norm
$\|\Sigma^{1/2}\cdot\|_{p\times p}$ on $\R^{p\times p}$.
Similarly, we write $\bar \theta^{G}$ for the orthogonal projection of $\theta$
on $\Theta_{G}$ according to the (random) Euclidean norm $\|\X\cdot\|_{n\times
p}$ on $\R^{p\times p}$.
For any $G\in\G_{D}$, we write $d_{a}(G)$ for the degree of the node $a$ in $G$
and introduce the positive quantity
\begin{multline*}
R(G)=\sum_{a=1}^p\pa{1+{\pen(d_{a}(G))\over n-d_{a}(G)}}\pa{\|\X(\theta_{a}-\bar
\theta_{a}^G)\|^2+2|<\X\theta_{a}-\X\bar\theta^G_{a},\eps_{a}>|}\\
+\sum_{a=1}^p{\pen(d_{a}(G))\over n-d_{a}(G)}\|\eps_{a}\|^2,
\end{multline*}
where $\|.\|$ and $<.,.>$ denote the canonical norm and scalar product on
$\R^n$.
Following the same lines as in the beginning of the proof of Theorem~2
in \cite{BGH09}, we get for any $G^*$ in $\widehat \G$
\begin{equation}\label{borne2}
{K-1\over K}\|\X(\thetat-\theta)\|_{n\times p}^2\1_{\B}\leq
R(G^*)\1_{\B}+\Delta(\widehat G)\1_{\B}
\end{equation}
with 
$$\Delta(G)=\sum_{a=1}^p{\sigma^2_{a}}\pa{KU_{\nei_{G}(a)}-{\pen(d_{a}(G))\over
n-d_{a}(G)}V_{\nei_{G}(a)}}_{+}$$
where $U_{\nei_{G}(a)}$ and $V_{\nei_{G}(a)}$ are two independent $\chi^2$
random variables with $d_{a}(G)+1$ and $n-d_{a}(G)-1$ degrees of freedom.

We note that under Condition (\ref{condition_degree}) there exists some constant
$c(\gamma)$ depending on $\gamma$ only, such that 
$$\pen(d)\leq c(\gamma) K (d+1)\log (p),\quad\textrm{for all }
d\in\ac{0,\ldots,D_{\max}},$$
see Proposition~4 in \cite{BGH09}. In particular, 
we have  for any $G\in\G_{D}$ 
$${\pen(d_{a}(G))    \over n-d_{a}(G)}\leq{c(\gamma)K(D_{\max}+1)\log(p)\over
n/2}\leq 4K\gamma c(\gamma) = L_{\gamma,K}.$$
Using this bound together with  
$$|2<\X\theta-\X\bar\theta_{a}^G,\eps_{a}>|\leq \|\X(\theta_{a}-\bar
\theta_{a}^G)\|^2+\sigma_{a}^2\xi_{a,G}^2,$$
 where for any  $G\in\G$ and  $a\in\ac{1,\ldots,p}$, the random variable 
 $$\xi_{a,G}=<\X(\theta_{a}-\bar
\theta_{a}^G),\eps>/(\sigma_{a}\|\X(\theta_{a}-\bar \theta_{a}^G) \|)$$ is 
standard Gaussian, we obtain
\begin{eqnarray*}
R(G)&\leq&(1+L_{\gamma,K})\sum_{a=1}^p\pa{2\|\X(\theta_{a}-\bar
\theta_{a}^G)\|^2+\sigma_{a}^2\xi_{a,G}^2}+{\pen(d_{a}(G))    \over
n-d_{a}(G)}\,\|\eps_{a}\|^2\\
&\leq & 2(1+L_{\gamma,K})\|\X(\theta-\bar \theta^G)\|_{n\times
p}^2+(4+L_{\gamma,K})\sum_{a=1}^p\pen(d_{a}(G))\sigma_{a}^2+r(\G_{D})
\end{eqnarray*}
where $r(\G_{D})$ equals
$$\sum_{a=1}^p\sigma_{a}^2\pa{(1+L_{\gamma,K})\sum_{G\in\G}\cro{\xi_{a
,G}^2-\pen(d_{a}(G))}_{+}+L_{\gamma,K}\cro{\|\eps_{a}\|^2/\sigma_{a}^2-3n/2}_{+}
}.$$
Furthermore, we have $\|\X(\theta-\bar \theta^G)\|_{n\times p}\leq \|\X(\theta-
\theta^G)\|_{n \times p}$ and on the event $\B$ we also have $\|\X(\theta-
\theta^G)\|^2_{n\times p}\leq n\lambda^{-2}\|\Sigma^{1/2}(\theta-
\theta^G)\|^2_{p \times p}$ so that on $\B$
$$R(G)\leq L'_{\gamma,K}\pa{n\lambda^{-2}\|\Sigma^{1/2}(\theta-\theta^G)\|^2_{p
\times p}+\sum_{a=1}^p\pen(d_{a}(G))\sigma_{a}^2}+r(\G_{D}),$$
with $L'_{\gamma,K}=\max(2+2L_{\gamma,K},4+L_{\gamma,K})$.
Putting this bound  together with~(\ref{borne1}) and(\ref{borne2}), we obtain
\begin{eqnarray*}
\|\Sigma^{1/2}(\thetat-\theta)\|^2_{p\times p}\1_{\B} &\leq&
{K\over n\lambda^2(K-1)}\pa{\inf_{G^*\in\widehat \G}R(G^*)+\Delta(\widehat
G)}\1_{\B}\\ 
&\leq&
L''_{\gamma,K}\inf_{G^*\in\widehat\G}\pa{\|\Sigma^{1/2}(\theta-\theta^{G^*}
)\|^2_{p\times p}+\sum_{a=1}^p\pen(d_{a}(G^*)){\sigma_{a}^2\over n}}\\ 
& &+L''_{\gamma,K}n^{-1}\pa{r(\G_{D})+\Delta(\widehat G)}.
\end{eqnarray*}
We note that
$$n^{-1}\E(r(\G_{D}))\leq \sum_{a=1}^p{\sigma_{a}^2\over n}
(1+L_{\gamma,K})(3+\log(p))$$
and we get from the proof of Theorem 1 in \cite{giraud08} that 
\begin{eqnarray*}
n^{-1}\E(\Delta(\widehat G))&\leq& n^{-1}\E\pa{\sup_{G\in\G_{D}}\Delta(G)}\leq
 K\sum_{a=1}^p{\sigma_{a}^2\over n} (1+\log(p)).
\end{eqnarray*}
Since $\pen(d)\leq c(\gamma) K (d+1)\log (p)$, the latter bounds enforce the
existence of constants $L_{\gamma,K}$ and $L'_{\gamma,K}$ depending on $\gamma$
and $K$ only, such that
\begin{eqnarray*}
\lefteqn{\E\cro{\|\Sigma^{1/2}(\thetat-\theta)\|^2_{p\times p}\1_{\B}}}\\
&\leq&L_{\gamma,K}\,\E\cro{\inf_{G^*\in\widehat\G}\pa{\|\Sigma^{1/2}
(\theta-\theta^{G^*})\|^2_{p \times p}+\sum_{a=1}^p\big(\log
(p)\vee\pen[d_{a}(G^*)]\big){\sigma_{a}^2\over n}}}\\
&\leq& L'_{\gamma,K}\,\log(p)\pa{
\E\cro{\inf_{G^*\in\widehat\G}\textrm{MSEP}(\widehat\theta_{G^*})}\vee\sum_{a=1}
^p{\sigma_{a}^2\over n}}.
\end{eqnarray*}
Finally, we note that $\sum_{a=1}^p{\sigma_{a}^2/ n}=\MSEP(I)$.

\subsubsection{Bound on $\E\cro{\|\Sigma^{1/2}(\thetat-\theta)\|^2_{p\times
p}\1_{\B^c}}$\label{section_risque_petit_evenement}}

We now prove the bound~\\
$\E\cro{\|\Sigma^{1/2}(\thetat-\theta)\|^2_{p\times p}\1_{\B^c}} \leq  
Ln^3\text{tr}(\Sigma)\sqrt{\mathbb{P}(\mathbb{B}^c)}$. 
We have
$$\E\cro{\|\Sigma^{1/2}(\thetat-\theta)\|^2_{p\times p}\1_{\B^c}}=
\sum_{a=1}^p\E\cro{\|\Sigma^{1/2}(\thetat_{a}-\theta_{a})\|^2\1_{\B^c}}$$
and we will upper bound  each of  the $p$ terms in this sum. Let $a$ be any node
in $\Gamma$.
Given a graph $G$, the vector $[\widehat{\theta}_{G}]_a$ depends on $G$ only
through the neighborhood $\nei_G(a)$ of $a$ in $G$.
Henceforth, we write $\widehat{\theta}_{\nei_{\widehat{G}}(a)}$ for
$\thetat_{a}$ in order to emphasize this dependency. By definition
$\widehat\theta_{\nei_{\widehat{G}}(a)}$ is the least-squares estimator of
$\theta_a$ with support  included in $\nei_{\hat{G}}(a)$. Let us apply the same
arguments as in the proof of Lemma 7.12 in \cite{Verzelen08}. By Cauchy-Schwarz
inequality, we have
\begin{equation}\label{majoration1}
 \E\cro{\|\Sigma^{1/2}(\thetat_{a}-\theta_a)\|^2\mathbf{1}_{\mathbb{B}^c}} \leq
\sqrt{\mathbb{P}(\mathbb{B}^c)}\
\sqrt{\mathbb{E}\cro{\|\Sigma^{1/2}(\widehat{\theta}_{\nei_{\widehat{G}}(a)}
-\theta_a)\|^4}}.
\end{equation}
Let $\mathcal{N}_D(a)$  be the set made of all the subsets of
$\Gamma\setminus\{a\}$ whose size are smaller than $\gamma
n/[2(1.1+\sqrt{\log(p)})^2]$. By Condition (\ref{condition_degree}), it holds
that 
the estimated neighborhood $\nei_{\widehat{G}}(a)$ belongs to
$\mathcal{N}_D(a)$, so H\"older inequality gives
\begin{eqnarray*}
\lefteqn{\mathbb{E}\cro{\|\Sigma^{1/2}(\widehat{\theta}_{\nei_{\widehat{G}}(a)}
-\theta_a)\|^4}=\sum_{\nei(a) \in
\mathcal{N}_D(a)}\mathbb{E}\cro{\mathbf{1}_{\nei_{\widehat{G}}(a)=\nei(a)}
\|\Sigma^{1/2}(\widehat{\theta}_{\nei(a)}-\theta_a)\|^4}}\\
&\leq& \sum_{\nei(a) \in
\mathcal{N}_D(a)}\mathbb{P}\left[\nei_{\widehat{G}}(a)=\nei(a)\right]^{1/u}
\mathbb{E}\cro{\|\Sigma^{1/2}(\widehat{\theta}_{\nei(a)}-\theta_a)\|^{4v}}^{1/v}
\\
&\leq& \sum_{\nei(a) \in
\mathcal{N}_D(a)}\mathbb{P}\left[\nei_{\widehat{G}}(a)=\nei(a)\right]^{1/u}
\sup_{\nei(a) \in
\mathcal{N}_D(a)}\mathbb{E}\cro{\|\Sigma^{1/2}(\widehat{\theta}_{\nei(a)}
-\theta_a)\|^{4v}}^{1/v},
\end{eqnarray*}
where  $v=\left\lfloor\frac{n}{8}\right\rfloor$, and $u=\frac{v}{v-1}$
(we remind the reader that $n$ is larger than $8$). In particular, we have the
crude bound
\begin{multline*}
\sqrt{\mathbb{E}\cro{\|\Sigma^{1/2}(\widehat{\theta}_{\nei_{\widehat{G}}(a)}
-\theta_a)\|^4}}\\
\leq  \left[\text{Card}(\mathcal{N}_D(a))\right]^{1/2v}\sup_{\nei(a) \in
\mathcal{N}_D(a)}\mathbb{E}\cro{\|\Sigma^{1/2}(\widehat{\theta}_{\nei(a)}
-\theta_a)\|^{4v}}^{1/2v},
\end{multline*}
since the sum is maximum when every $\mathbb{P}[\nei(a)=\nei_{\widehat{G}}(a)]$
equals $[\text{Card}(\mathcal{N}_D(a))]^{-1}$. 
We first bound the term $\left[\text{Card}(\mathcal{N}_D(a))\right]^{1/2v}$. The
size of the largest subset in $\mathcal{N}_D(a)$ is smaller than $n/(2\log
(p))$, so the cardinality of $\mathcal{N}_D(a)$ is  smaller than $p^{D_{\widehat
\G}}$. Since $n$ is larger than 8, we get
\begin{eqnarray*}
 \left[\text{Card}(\mathcal{N}_D(a))\right]^{1/2v}\leq
\exp\left[\frac{n}{4\lfloor n/8\rfloor}\right]\leq L\ ,
\end{eqnarray*}
which ensures the bound
\begin{equation}
 \sqrt{\mathbb{E}\cro{\|\Sigma^{1/2}(\widehat{\theta}_{\nei_{\widehat{G}}(a)}
-\theta_a)\|^4}}\leq L\sup_{\nei(a) \in
\mathcal{N}_D(a)}\mathbb{E}\cro{\|\Sigma^{1/2}(\widehat{\theta}_{\nei(a)}
-\theta_a)\|^{4v}}^{1/2v}. \label{majoration_holder}
\end{equation}

To conclude, we need to upper bound this supremum.  Given a subset $\nei(a)$ in 
$\mathcal{N}_D(a)$, we define $\theta_{\nei(a)}$ 
as the vector in $\R^p$ 
such that $\Sigma^{1/2}\theta_{\nei(a)}$ is the orthogonal projection  of
$\Sigma^{1/2}\theta_a$ onto the linear span 
$\ac{\Sigma^{1/2}\beta:\textrm{supp}(\beta)\subset\nei(a)}$. 
Pythagorean inequality gives
\begin{eqnarray*}
\|\Sigma^{1/2}(\widehat{\theta}_{\nei(a)}-\theta_a)\|^{2}=\|\Sigma^{1/2}(\theta_
{\nei(a)}-\theta_a)\|^{2}+\|\Sigma^{1/2}(\widehat{\theta}_{\nei(a)}-\theta_{
\nei(a)})\|^{2}
\end{eqnarray*}
and  we obtain from  Minkowski's inequality that
\begin{multline*}
\mathbb{E}\cro{\|\Sigma^{1/2}(\widehat{\theta}_{\nei(a)}-\theta_a)\|^{4v}}^{
1/(2v)} \\
 \leq
\|\Sigma^{1/2}(\theta_{\nei(a)}-\theta_a)\|^{2}+\mathbb{E}\cro{\|\Sigma^{1/2}
(\widehat{\theta}_{\nei(a)}-\theta_{\nei(a)})\|^{4v}}^{1/(2v)}.\\
\end{multline*}
The first term is smaller than $\var(X_a)$. In order to bound the second term,
we use the following lemma which rephrases Proposition~7.8
in \cite{Verzelen08}.
\begin{lemma}\label{lemme_majoration_risque_lp}
 For any neighborhood $\nei(a)$ and any $r>2$ such that $n-|\nei(a)|-2r+1>0$, 
$$\mathbb{E}\cro{\|\Sigma^{1/2}(\widehat{\theta}_{\nei(a)}-\theta_{\nei(a)})\|^{
2r}}^{1/r} \leq Lr|\nei(a)|n\var(X_a)\,.$$
\end{lemma}
Since $v$ is smaller than $n/8$ and  since $|\nei(a)|$ is smaller than $n/2$, it
follows that for any model $\nei(a)\in\mathcal{N}_D(a)$, $n-|\nei(a)|-4v+1$ is
positive and
\begin{eqnarray*}
\mathbb{E}\cro{\|\Sigma^{1/2}(\widehat{\theta}_{\nei(a)}-\theta_a)\|^{4v}}^{
1/(2v)}     \leq \var(X_a)\left[1+ Ln^2v\right]
\leq  Ln^3\Sigma_{a,a}\,.
\end{eqnarray*} 
 Gathering this last upper bound with (\ref{majoration1}) and 
(\ref{majoration_holder}), we get that
\begin{eqnarray*}
 \mathbb{E}\cro{\|\Sigma^{1/2}(\thetat-\theta)\|^2_{p \times
p}\mathbf{1}_{\mathbb{B}^c}}&\leq&   Ln^3\text{tr}(\Sigma)
\sqrt{\mathbb{P}(\mathbb{B}^c)}.
\end{eqnarray*}

\subsubsection{Conclusion}
Finally, putting  together the bound on $
\mathbb{E}[\|\Sigma^{1/2}(\thetat-\theta)\|^2\mathbf{1}_{\mathbb{B}}]$, the
bound on  $
\mathbb{E}[\|\Sigma^{1/2}(\thetat-\theta)\|^2\mathbf{1}_{\mathbb{B}^c}]$, and
the bound $\P(\B^c)\leq 2pe^{-n(\sqrt{\gamma}-\gamma)^2/2}$, we obtain
$$\MSEP(\widetilde{\theta}) \leq  L_{K,\gamma}\log(p)\pa{
\E\cro{\inf_{G\in\widehat \G}\pa{\MSEP(\widehat{\theta}_G)}}\vee {\MSEP(I)\over
n}}
+R_n\,,$$
with $R_{n}\leq Ln^3\text{tr}(\Sigma)
e^{-n(\sqrt{\gamma}-\gamma)^2/4}$.

\subsection{Proof of Proposition \ref{corollaire_risque}}

The result is proved analogously except that we replace the event $\mathbb{B}$
by 
$$\mathbb{B}'= \mathbb{B}\cup \left\{
G_{\Sigma}\in\widehat{\mathcal{G}}\right\}\ .$$
Hence, the residual term now satisfies 
\begin{eqnarray*}
R_n &\leq &Ln^3\text{tr}(\Sigma) \sqrt{\mathbb{P}(\mathbb{B}^c)}\\ & \leq
&Ln^3\text{tr}(\Sigma) \left[e^{-n(\sqrt{\gamma}-\gamma)^2/4}+
\sqrt{\alpha}e^{-\frac{\beta}{2}n^\delta}\right]\ .
\end{eqnarray*}

\subsection{Proof of Theorem \ref{proposition_consistance}}
In this proof, the notations $o(1)$, $O(1)$ respectively refer to sequences that
converge to $0$ or stay bounded when $n$ goes to infinity. These sequences may
depend on $K$,  $s$, $s'$ but \emph{do not}  depend on $G_n$, on the covariance
$\Sigma$, or a particular subset $S\subset\Gamma$. The technical lemmas are
postponed to Section \ref{section_lemmas}. In the sequel, we omit the dependency
of $p$ and $\Sigma$ on $n$ for the sake of clarity.
First, observe that the result is trivial if $n/\log(p)^2<1$, because the
assumptions imply that $G_{\Sigma}$ is the empty graph whereas the family
$\widehat{\mathcal{G}}$ contains at most the empty graph. In the sequel, we
assume that  $n/\log(p)^2\geq 1$.\\

\noindent 
Let us set $D_\textrm{max}=n/\log(p)^2$. We shall prove that for some $L>0$, 
\begin{eqnarray}\label{equation_consistance_preuve1}
 \mathbb{P}\left(\crit(G_{\Sigma}) = \inf_{G',\ \deg(G')\leq D_\textrm{max}}
\crit(G')\right) \geq  1- Lp^{-1/2}\ ,
\end{eqnarray}
for $n$ larger than $n_0(K,s,s')$. Since $\widehat{G}$ minimizes the criterion
$\crit(.)$ on the family $\widehat{\mathcal{G}}$, this will imply the result of
the theorem.

In fact, we shall prove a slightly stronger result than
(\ref{equation_consistance_preuve1}). Let $a$ be a node in $\Gamma$ and let
$\nei(a)$ be a subset of $\Gamma\setminus\{a\}$. As defined in Section
\ref{section_risque_petit_evenement},  $\widehat{\theta}_{\nei(a)}$ is the
least-squares estimator of $\theta_a$ whose support is included in $\nei(a)$.
\begin{eqnarray*}
 \widehat{\theta}_{\nei(a)} = \arg\inf_{\theta'_a,\ \supp(\theta'_a)\subset
\nei(a)} \|{\bf X}_a-{\bf X}\theta'_a\|_n^2\ .
\end{eqnarray*}
If $G$ is a graph such that the neighborhood $\nei_G(a)$ equals $\nei(a)$, then
$\widehat{\theta}_{\nei(a)}=[\widehat{\theta}_G]_a$. We then define the partial
criterion $\crit(a,\nei(a))$ by 
\begin{eqnarray*}
\crit(a,\nei(a))=\|{\bf X}_a- {\bf
X}\widehat{\theta}_{\nei(a)}\|_n^2\left(1+\frac{\pen(|\nei(a)|)}{n-|\nei(a)|}
\right) \ .
\end{eqnarray*}

Observe that for any  graph $G$, $\crit(G)=\sum_{a=1}^p\crit(a,\nei_{G}(a))$. We
note $\widehat{\nei}(a)$ the set that minimizes the criterion $\crit(a,.)$ among
all subsets of size smaller than $D_\textrm{max}$.
\begin{eqnarray*}
 \widehat{\nei}(a) = \arg \inf_{\nei(a)\in \mathcal{N}_{D_\textrm{max}}(a)}
\crit(a,\nei(a))\ .
\end{eqnarray*}
If for all nodes $a\in\Gamma$, the selected set $\widehat{\nei}(a)$ equals
$\nei_{G_{\Sigma}}(a)$, then $G_{\Sigma}$ minimizes the criterion $\crit(.)$
over all graphs of degree smaller than $D_\textrm{max}$. Consequently, the
property (\ref{equation_consistance_preuve1}) is satisfied if for any node
$a\in\Gamma$, it holds that
\begin{eqnarray}\label{equation_consistance_preuve2}
 \mathbb{P}\left[\widehat{\nei}(a) = \nei_{G_{\Sigma}}(a)\right]\geq 1-
7p_{n}^{-3/2}\ ,
\end{eqnarray}
for $n$ larger than some $n_0[K,s,s']$.\\

\noindent
Let us fix some node $a\in\Gamma$. We prove the lower bound
(\ref{equation_consistance_preuve2}) in two steps:

\begin{enumerate}
 \item With high probability, the estimated neighborhood $\widehat{\nei}(a)$
does not strictly contain the true one $\nei_{G_{\Sigma}}(a)$.
\begin{eqnarray}\label{equation_consistance_preuve3}
 \mathbb{P}\left[\widehat{\nei}(a) \varsupsetneq \nei_{G_{\Sigma}}(a)
\right]\leq p_{n}^{-3/2}\ ,
\end{eqnarray}
for $n$ larger than some $n_0[K,s,s']$.

\item With high probability, the estimated neighborhood $\widehat{\nei}(a)$
contains the true one $\nei_{G_{\Sigma}}(a)$.
\begin{eqnarray}\label{equation_consistance_preuve4}
 \mathbb{P}\left[\widehat{\nei}(a) \nsupseteq  \nei_{G_{\Sigma}}(a) \right]\leq
6p_{n}^{-3/2}\ ,
\end{eqnarray}
for $n$ larger than some $n_0[K,s,s']$.
\end{enumerate}

\noindent
The remaining part of the proof is deserved to
(\ref{equation_consistance_preuve3}) and (\ref{equation_consistance_preuve4}).\\

\noindent
Let us recall some notations and let us introduce some other ones. The component
$X_a$ decomposes as 
\begin{eqnarray*}
 X_a= X\theta_a+\epsilon_a\ ,
\end{eqnarray*}
where $\epsilon_a$ follows a centered normal distribution with variance
$\Omega_{a,a}^{-1}=\var(X_a|X_{-a})$. The variables $\epsilon_a$ are independent
of $X_{-a}$.
Given a set $S\subset \Gamma$, $\Pi_S$ stands for the projection of
$\mathbb{R}^n$ into the space generated by $({\bf X}_a)_{a\in S}$, whereas
$\Pi_S^{\perp}$ denotes the projection along the space generated by $({\bf
X}_a)_{a\in S}$. The notation $\langle.,.\rangle_n$  refers to the empirical
inner product associated with the norm $\|.\|_n$.
For any neighborhood $\nei(a)\subset\Gamma\setminus\{a\}$ such that
$|\nei(a)|\leq D_\textrm{max}$, let us define
$\Delta(\nei(a),\nei_{G_{\Sigma}}(a))$ by
\begin{eqnarray*}
 \Delta(\nei(a),\nei_{G_{\Sigma}}(a))= \crit(a,\nei(a)) -
\crit(a,\nei_{G_{\Sigma}}(a))\ .
\end{eqnarray*}

\subsubsection{Bound on  $\mathbb{P}\left(\widehat{\nei}(a) \varsupsetneq
\nei_{G_{\Sigma}}(a)  \right)$}

We shall upper bound the probability that $\Delta(\nei(a),\nei_{G_{\Sigma}}(a))$
is negative for at least one of the neighborhoods
$\nei(a)\in\mathcal{N}_{D_{\textrm{max}}}(a)$ such that $\nei(a)$ strictly
contains $\nei_{G_{\Sigma}}(a)$. For such a set $\nei(a)$, 
$\Delta(\nei(a),\nei_{G_{\Sigma}}(a))$ decomposes as (see e.g. Lemma 7.1 in \cite{Verzelen08}).
\begin{eqnarray*}
\lefteqn{ \Delta(\nei(a),\nei_{G_{\Sigma}}(a))  }\\ & = &\|\Pi_{\nei(a)}^{\perp}
\boldsymbol{\epsilon}_a\|_n^2\left[1+\frac{\pen(|\nei(a)|)}{n-|\nei(a)|}\right]-
\|\Pi_{\nei_{G_{\Sigma}}(a)}^{\perp}
\boldsymbol{\epsilon}_a\|_n^2\left[1+\frac{\pen(|\nei_{G_{\Sigma}}(a)|)}{
n-|\nei_{G_{\Sigma}}(a)|}\right]\\
&= & -\|\Pi_{\nei_{G_{\Sigma}}(a)^{\perp}\cap \nei(a)
}\boldsymbol{\epsilon}_a\|_n^2\left[1+\frac{\pen(|\nei_{G_{\Sigma}}(a)|)}{
n-|\nei_{G_{\Sigma}}(a)|}\right]\\ & & \mbox{}+\|\Pi_{\nei(a)}^{\perp}
\boldsymbol{\epsilon}_a\|_n^2\left[\frac{\pen(|\nei(a)|)}{n-|\nei(a)|}-\frac{
\pen(|\nei_{G_{\Sigma}}(a)|)}{n-|\nei_{G_{\Sigma}}(a)|}\right]\ .
\end{eqnarray*}
Hence, $\Delta(m,\nei_{G_{\Sigma}}(a))>0$  if 
\begin{eqnarray}\label{minoration_importante_consistance}
\lefteqn{\frac{\|\Pi_{\nei_{G_{\Sigma}}(a)^{\perp}\cap \nei(a)
}\boldsymbol{\epsilon}_a\|_n^2/(|\nei(a)\setminus\nei_{G_{\Sigma}}(a)|)}{\|\Pi_{
\nei(a)}^{\perp} \boldsymbol{\epsilon}_a\|_n^2/(n-|\nei(a)|)}} & &\nonumber \\ &
& <\,
\frac{\pen(|\nei(a)|)-\pen(|\nei_{G_{\Sigma}}(a)|)}{|\nei(a)\setminus\nei_{G_{
\Sigma}}(a)|}\left[1+\frac{\pen(|\nei_{G_{\Sigma}}(a)|)}{n-|\nei_{G_{\Sigma}}
(a)|}\right]^{-1}\ .
\end{eqnarray}

To conclude, it remains to prove that the bound
(\ref{minoration_importante_consistance}) holds with high probability. Let us
call $A_1$ the right expression of (\ref{minoration_importante_consistance}) and
let us derive a lower bound of $A_1$. Afterwards, we shall upper bound with high
probability the left expression of
(\ref{minoration_importante_consistance}).\\

{\bf Upper bound of $A_1$.} We first upper bound the penalty function.

\begin{lemma}\label{lemme_penalite_general}
Let $d_1\geq d_2$ be two positive integers such that $d_1\leq e^{-2}(p-1)$. We
have
\begin{eqnarray}\label{minoration_difference_penalite}
 \pen(d_1)-\pen(d_2)\geq 2K(d_1-d_2) \log\left(\frac{p-d_1}{d_1}\right)\ .
\end{eqnarray}
\end{lemma}
 
\noindent
A proof of this lemma is provided in Section \ref{section_lemmas}.
By Proposition 4 in \cite{BGH09}, the penalty $\pen(|\nei_{G_{\Sigma}}(a)|)$
satisfies
$$\pen\left(|\nei_{G_{\Sigma}}(a)|\right)\leq
LK\frac{|\nei_{G_{\Sigma}}(a)|}{n}\log\left(\frac{p-1}{|\nei_{G_{\Sigma}}(a)|}
\right)\ ,$$    
where $L$ is some numerical constant. This last term converges towards $0$ as
$n$ goes to infinity since $|\nei_{G_{\Sigma}}(a)|\leq
(n^s/\log(p))\wedge(n/\log(p)^2)$ (Assumption 2). Gathering this upper bound
with Lemma \ref{lemme_penalite_general}, we get
\begin{eqnarray}
A_1\geq 
2K\frac{\log\left(\frac{p-|\nei(a)|}{|\nei(a)|}\right)}{1+\frac{\pen(|\nei_{G_{
\Sigma}}(a)|)}{n-|\nei_{G_{\Sigma}}(a)|}}\geq
2K\log\left(\frac{p}{|\nei(a)|}\right)\left(1-o(1)\right)\ .
\label{mino1_consistance}
\end{eqnarray}
\medskip

\noindent
{\bf Lower bound of the left part of (\ref{minoration_importante_consistance})}.
The random variables involved in this expression  follow a 
 Fisher distribution with $|\nei(a)\setminus\nei_{G_{\Sigma}}(a)|$ and
$n-|\nei(a)|$ degrees of freedom. To conclude, we only need to compare the
quantile of such a variable with the bound (\ref{mino1_consistance}).
Let $u\in (0,1)$ and let $F^{-1}_{D,N}(u)$ denote
the $1-u$ quantile of a Fisher random variable with $D$ and $N$ degrees of
freedom. By Lemma 1 in \cite{Baraud03}, it holds that 
\begin{eqnarray*}
DF^{-1}_{D,N}(u)&\leq &D+2
\sqrt{D\left(1+2\frac{D}{N}\right)\log\left(\frac{1}{u}\right)}\\&
+&\left(1+2\frac{D}{N}\right)\frac{N}{2}\left[\exp\left(\frac{4}{N}
\log\left(\frac{1}{u}\right)\right)-1\right]\ .
\end{eqnarray*}
Let us set $u$ to 
\begin{eqnarray*}
 u= \left\{p^{3/2}e^{|\nei(a)\setminus
\nei_{G_{\Sigma}}(a)|}\binom{p-|\nei_{G_{\Sigma}}(a)|-1}{|\nei(a)\setminus
\nei_{G_{\Sigma}}(a)|}\right\}^{-1}\ .
\end{eqnarray*}

\noindent 
Since we consider the case $n/\log(p)^2 \geq 1$ and $p\geq n$, the term
$4/(n-|\nei(a)|)\log(1/u)$ goes to $0$ with $n$ (uniformly w.r.t. $\nei(a)$).
\begin{eqnarray*}
A_2= F^{-1}_{|\nei(a)\setminus\nei_{G_{\Sigma}}(a)|,n-|\nei(a)|}(u) &\leq& 1+ 2 
\sqrt{\frac{1}{|\nei(a)\setminus\nei_{G_{\Sigma}}(a)|}
\left(1+o(1)\right)\log\left(\frac{1}{u}\right)}\\
&
&+\frac{2}{|\nei(a)\setminus\nei_{G_{\Sigma}}(a)|}
\left(1+o(1)\right)\log\left(\frac{1}{u}\right)\ .
\end{eqnarray*} 
The term $\log(1/u)/|\nei(a)\setminus\nei_{G_{\Sigma}}(a)|$ goes to infinity
with $n$ (uniformly w.r.t. $\nei(a)$). Hence, we get
\begin{eqnarray}
A_2 &\leq&
1+\frac{2}{|\nei(a)\setminus\nei_{G_{\Sigma}}(a)|}\log\left(\frac{1}{u}
\right)\left(1+ o(1)\right)\nonumber\ .
\end{eqnarray}
Applying the classical inequality $\log\binom{l}{k}\leq k\log(el/k)$, we obtain
\begin{eqnarray}
A_2 &\leq & \left[3\frac{\log(p)}{|\nei(a)\setminus \nei_{G_{\Sigma}}(a)|}+
2\log\left(\frac{p}{|\nei(a) 
\setminus\nei_{G_{\Sigma}}(a)|}\right)\right]\left(1+ o(1)\right)\nonumber\\
& \leq &
5\log\left(\frac{p}{|\nei(a)\setminus\nei_{G_{\Sigma}}(a)|}
\right)\left(1+o(1)\right)\label{majo1consistance}\ .
\end{eqnarray}
\medskip

\noindent 
{\bf Conclusion}.
Let us  compare the lower bound (\ref{mino1_consistance}) of $A_1$ with the
upper bound (\ref{majo1consistance}) of $A_2$.
\begin{itemize}
 \item Let us first assume that $|\nei(a)|\leq 2|\nei_{G_{\Sigma}}(a)|$. Then,
we have
$$A_1\geq 
2K\log\left(\frac{p}{|\nei_{G_{\Sigma}}(a)|}\right)\left(1-o(1)\right)\geq
2K(1-s)\log(p)\left(1-o(1)\right)\ ,$$
since $|\nei_{G_{\Sigma}}(a)|\leq n^s/\log(p) \leq p^s$. In particular, 
$$ A_2\leq
5\log\left(\frac{p}{|\nei(a)\setminus\nei_{G_{\Sigma}}(a)|}
\right)\left(1+o(1)\right)< A_1 ,$$ for $n$ large enough since we assume that
$2K(1-s)>5$.
\item If $|\nei(a)|> 2|\nei_{G_{\Sigma}}(a)|$, we also have 
$$ A_2\leq 5\log\left(\frac{p}{|\nei(a)|}\right)\left(1+o(1)\right)< A_1\ ,$$
for $n$ large enough since we assume that $2K>5$.
\end{itemize}

\noindent 
It follows from Ineq. (\ref{minoration_importante_consistance}) and the
definition of $A_1$ and $A_2$ that
\begin{eqnarray*}
 \mathbb{P} \left[\Delta(\nei(a),\nei_{G_{\Sigma}}(a))<0 \right]\leq
\left\{p^{3/2}e^{|\nei(a)\setminus
\nei_{G_{\Sigma}}(a)|}\binom{p-|\nei_{G_{\Sigma}}(a)|}{|\nei(a)\setminus
\nei_{G_{\Sigma}}(a)|}\right\}^{-1}\ ,
\end{eqnarray*}
for $n$ larger than some positive constant that may depend on $K$, $s$, but does
\emph{not} depend on $\nei(a)$.
Applying this bound to any neighborhood $\nei(a)$ that strictly contains
$\nei_{G_{\Sigma}}(a)$ yields Statement (\ref{equation_consistance_preuve3}):
\begin{eqnarray*}
\mathbb{P}\left[\widehat{\nei}(a) \varsupsetneq \nei_{G_{\Sigma}}(a) 
\right]\leq p^{-3/2}\ ,
\end{eqnarray*}
for $n$ large enough.

\subsubsection{Bound on  $\mathbb{P}\left(\widehat{\nei}(a) \nsupseteq
\nei_{G_{\Sigma}}(a)  \right)$}

Again, we shall prove that $\Delta[\nei(a),\nei_{G_{\Sigma}}(a)]$ is positive
for $\nei(a)\nsupseteq \nei_{G_\Sigma}(a)$ with overwhelming probability. We
recall that $\theta_{\nei(a)}$ is the vector in $\mathbb{R}^{p}$ such that
$\Sigma^{1/2}\theta_{\nei(a)}$ is the orthogonal projection of
$\Sigma^{1/2}\theta_{a}$ onto the linear span $\left\{\Sigma^{1/2}\beta\ :
\supp(\beta)\subset\nei(a)\right\}$. Moreover,
$\|\Sigma^{1/2}(\theta_a-\theta_{\nei(a)})\|^2=\var(X_a|X_{\nei(a)})-\var(X_a|X_
{-a})$ (see e.g. Lemma 7.1 in \cite{Verzelen08}). 

Then,  $\Delta(\nei(a),\nei_{G_{\Sigma}}(a))$ decomposes as 
\begin{eqnarray*}
\lefteqn{\Delta(\nei(a),\nei_{G_{\Sigma}}(a))   =  \left\|\Pi_{\nei(a)}^{\perp}
\left[\boldsymbol{\epsilon}_a+{\bf  X
}(\theta_a-\theta_{\nei(a)})\right]\right\|_n^2\left[1+\frac{\pen(|\nei(a)|)}{
n-|\nei(a)|}\right]} \\& & \mbox{}- \left\|\Pi_{\nei_{G_{\Sigma}}(a)}^{\perp} 
\boldsymbol{\epsilon}_a\right\|_n^2\left[1+\frac{\pen(|\nei_{G_{\Sigma}}(a)|)}{
n-|\nei_{G_{\Sigma}}(a)|}\right]\ .\hspace{3cm}
\end{eqnarray*}
Let $\kappa=6/7$ and let us define $$E_{\nei(a)}= \kappa^{-1}\left\langle
\frac{\Pi_{\nei(a)}^{\perp}{\bf  X
}(\theta-\theta_{\nei(a)})}{\|\Pi_{\nei(a)}^{\perp}{\bf  X
}(\theta-\theta_{\nei(a)})\|_n},\Pi_{\nei(a)}^{\perp} \boldsymbol{\epsilon}_a
\right\rangle_n^2+ \|\Pi_{\nei(a)} \boldsymbol{\epsilon}_a\|_n^2\ .$$ 
We recall that $\langle.,.\rangle_n$ is the inner product associated to the norm
$\|.\|_n$.
The quantity $\Delta(\nei(a),\nei_{G_{\Sigma}}(a))$ is positive  if  
\begin{eqnarray}
\lefteqn{(1-\kappa)\|\Pi_{\nei(a)}^{\perp} {\bf  X
}(\theta-\theta_{\nei(a)})\|_n^2  >
E_{\nei(a)}\left[1+\frac{\pen(|\nei(a)|)}{n-|\nei(a)|}\right]}\nonumber\\
& + &
\|\boldsymbol{\epsilon}_a\|_n^2\left[\frac{\pen(|\nei_{G_{\Sigma}}(a)|)}{
n-|\nei_{G_{\Sigma}}(a)|}- \frac{\pen(|\nei(a)|)}{n-|\nei(a)|} \right]\
.\label{minoration_importante_deuxieme_cas}
\end{eqnarray}
We respectively call $A_3$ and $A_4$ the right and the left terms of the
inequality. We shall control their deviations  in
order to prove that (\ref{minoration_importante_deuxieme_cas}) holds with high
probability.\\

\noindent
{\bf Upper Bound of $A_3$}.
On an event $\mathbb{A}$ of probability larger than $1- 2p^{-3/2}$, the random
variable $\|\boldsymbol{\epsilon}_a\|_n^2$ satisfies (see Lemma 1 in
\cite{Laurent00}).
$$1- 2\sqrt{\frac{3\log(p)}{2n}}\leq
\frac{\|\boldsymbol{\epsilon}_a\|_n^2}{\var(X_a|X_{-a})}\leq 1+
2\sqrt{\frac{3\log (p)}{2n}}+3\frac{\log (p)}{n}\ .$$

Let us bound the other random variables involved in
(\ref{minoration_importante_deuxieme_cas}). As explained in 
the proof of Th.3.1  in \cite{Verzelen08}, the random variables
$\|\Pi_{\nei(a)}^{\perp} {\bf  X }(\theta-\theta_{\nei(a)})\|_n^2$ and
$E_{\nei(a)}$ follow distributions of linear combinations of $\chi^2$ random
variables. 
We apply again Lemma 1 in \cite{Laurent00} . On a event $\mathbb{A}_{\nei(a)}$
of  probability larger than $1-
2p^{-3/2}e^{-|\nei(a)|}\binom{p-1}{|\nei(a)|}^{-1}$, it holds that 
\begin{eqnarray*}
\lefteqn{\frac{\|\Pi_{\nei(a)}^{\perp} {\bf  X
}(\theta-\theta_{\nei(a)})\|_n^2}{\var(X_a|X_{\nei(a)})-\var(X_a|X_{-a})} \geq 1
- \frac{|\nei(a)|}{n}}\hspace{4cm}\\  & & \mbox{}- 2\sqrt{\frac{\frac{3}{2}\log
(p) + |\nei(a)|\left[2+\log\left(p-1\right)\right]}{n} }
\end{eqnarray*}
and 
\begin{eqnarray*}
\lefteqn{\frac{E_{\nei(a)}}{\var(X_a|X_{-a})} \leq
\frac{|\nei(a)|+\kappa^{-1}}{n}}\\  & & \mbox{}+       
\frac{2}{n}\sqrt{(|\nei(a)|+\kappa^{-2})\left[|\nei(a)|\left(2+\log\left(\frac{
p-1}{|\nei(a)|}\right)\right)+\frac{3}{2}\log(p)\right]}\\ & &\mbox{}
+\frac{2\kappa^{-1}}{n}\left[|\nei(a)|\left(2+\log\left(\frac{p-1}{|\nei(a)|}
\right)\right)+\frac{3}{2}\log(p)\right]\ .
\end{eqnarray*} 
We derive that
\begin{eqnarray*}
\frac{E_{\nei(a)}}{\var(X_a|X_{-a})} &\leq &   
\frac{2\kappa^{-1}}{n}\left[|\nei(a)|\log\left(\frac{p-1}{|\nei(a)|}
\right)+\frac{3}{2}\log(p)\right]\left(1+o(1)\right)\\
&+&\frac{\sqrt{6|\nei(a)|\log(p)}}{n}+\frac{\kappa^{-1}}{n}\ .
\end{eqnarray*}

\begin{itemize}
 \item {\bf CASE 1}: {\bf $\nei(a)$ is non empty}.
\begin{eqnarray*}
 \frac{E_{\nei(a)}}{\var(X_a|X_{-a})} \leq
\kappa^{-1}\frac{2|\nei(a)|\log\left(\frac{p-1}{|\nei(a)|}\right)+3\log(p)}{n}
(1+o(1))\ . 
\end{eqnarray*}
Let us upper bound the terms involving $\pen(|\nei(a)|)$ in
(\ref{minoration_importante_deuxieme_cas}) on the event $\mathbb{A}\cap
\mathbb{A}_{\nei(a)}$.
\begin{eqnarray*}
\lefteqn{\left\{E_{\nei(a)}\left[1+\frac{\pen(|\nei(a)|)}{n-|\nei(a)|}\right]-
\|\boldsymbol{\epsilon}_a\|_n^2\frac{\pen(|\nei(a)|)}{n-|\nei(a)|}\right\}
/\var(X_a|X_{-a})}\hspace{2cm}\\ & \leq & 
\frac{\kappa^{-1}}{n}\left(2|\nei(a)|\log\left(\frac{p-1}{|\nei(a)|}
\right)+3\log(p)\right)(1+o(1))\\ & & \mbox{}-
\frac{2K}{n}|\nei(a)|\log\left(\frac{p-1}{|\nei(a)|}\right)(1+o(1))\ .
\end{eqnarray*} 
This last quantity is negative for $n$ large enough since $K\geq 3$. 

\item   {\bf CASE 2}: {\bf $\nei(a)$ is empty}. We get the upper bound
\begin{eqnarray*}
\lefteqn{E_{\nei(a)}\left[1+\frac{\pen(|\nei(a)|)}{n-|\nei(a)|}\right]-
\|\boldsymbol{\epsilon}_a\|_n^2\frac{\pen(|\nei(a)|)}{n-|\nei(a)|}}\hspace{3cm}
\\ & \leq & \frac{\kappa^{-1}+ 3\log (p)}{n}\var(X_a|X_{-a})\\ &\leq& 
(3+\kappa^{-1})n^{s-1}\var(X_a|X_{-a})\ .
\end{eqnarray*}
 
\noindent
Indeed, $\log(p)$ has to be  smaller than $n^s$. If this is not the case, then
$\nei_{G_{\Sigma}}(a)$ should be empty and $\nei(a)$ cannot satisfy
$\nei_{G_{\Sigma}}(a)\nsubseteq \nei(a)$.
\end{itemize}

\noindent
We conclude that  on the event $\mathbb{A}\cap\mathbb{A}_{\nei(a)}$,
\begin{eqnarray*}
 A_3\leq  (3+\kappa^{-1})n^{s-1}\var(X_a|X_{-a})+
\|\boldsymbol{\epsilon}_a\|_n^2\frac{\pen(|\nei_{G_{\Sigma}}(a)|)}{n-|\nei_{G_{
\Sigma}}(a)|}\ ,
\end{eqnarray*}
for $n$ large enough.
 Let us upper bound the penalty term as done in the upper bound of $A_1$.
$$\pen(\nei_{G_{\Sigma}}(a))\leq
LK\frac{|\nei_{G_{\Sigma}}(a)|}{n}\log\left(\frac{p-1}{|\nei_{G_{\Sigma}}(a)|}
\right)\ .$$
 Since  $|\nei_{G_{\Sigma}}(a)|$ is assumed to be smaller than
$\frac{n^s}{\log(p)}$, the term $A_3$ is  upper bounded as follows
\begin{eqnarray}\label{borne_c}
A_3\leq  (K+1)n^{s-1}\var(X_a|X_{-a})O(1)\ .
\end{eqnarray}
for $n$ large enough.\\

\noindent 
{\bf Lower Bound of $A_4$}.
 Let us lower bound the left term  $A_4$ in
(\ref{minoration_importante_deuxieme_cas}) on the event
$\mathbb{A}\cap\mathbb{A}_{\nei(a)}$.
\begin{eqnarray*}
A_4 &\geq &
(1-o(1))(1-\kappa)\left[\var(X_a|X_{\nei(A)})-\var(X_a|X_{-a})\right]\\ & \geq &
(1-o(1))(1-\kappa)\min_{b\in\Gamma\setminus\{a\}}\left(\theta_{a,b}
\right)^2\min_{b,c\in\Gamma\setminus\{a\}}\frac{\var(X_b|X_{-b})}{\var(X_c|X_{-c
})}\var(X_a|X_{-a})\\
&\geq & (1-\kappa)(1-o(1))n^{s'-1}\var(X_a|X_{-a})\ .
\end{eqnarray*}

Thanks to the last bound and  (\ref{borne_c}) and since $s'$ is larger than $s$,
$A_3< A_4$ on the event $\mathbb{A}\cap\mathbb{A}_{\nei(a)}$ and for $n$ large
enough (not depending on $\nei(a)$). Hence, for $n$ large enough the inequality
(\ref{minoration_importante_deuxieme_cas}) holds  simultaneously for all
neighborhoods $\nei(a)$ such that $\nei_{G_{\Sigma}}(a)\nsubseteq \nei(a)$ with
probability larger than $1-2 p^{-3/2}- 2(e/(e-1))p^{-3/2}$. We conclude that
\begin{eqnarray*}
\mathbb{P}\left(\widehat{\nei}(a) \nsupseteq \nei_{G_{\Sigma}}(a)  \right)\leq
6p^{-3/2}\ ,
\end{eqnarray*}
for $n$ large enough.

\subsection{Lemmas}\label{section_lemmas}

Let us prove the following lemmas.

\begin{lemma}\label{lemme_minoration_penalite_seul}
For any positive integer $d\leq e^{-2}(p-1)$, 
$$\EDKhi\left[d+1,n-d-1,\left[\binom{p-1}{d}(d+1)^2\right]^{-1}\right] \geq d+1\
.$$
\end{lemma}

\begin{lemma}\label{lemme_penalite1} 
For any positive number $x$ and any positive integers $d$ and $N$,
$\EDKhi(d,N,x)$ is an increasing function with respect to $d$ and a decreasing
function with respect to $N$.
\end{lemma}

\begin{lemma}\label{lemme_penalite2} 
For any integer $d\geq 2$, the function
$$\frac{\mathbb{E}\left[(X_d-x\frac{X_N}{N})_+\right]}{\mathbb{E}\left[
(X_2-x)_+\right]}$$ is increasing with respect to $x$ as soon as $x\geq d$.
\end{lemma}

\subsubsection{Proof of Lemma \ref{lemme_penalite_general}}
Let us write $L_1=\log\left(\binom{p-1}{d_1}\right)$ and
$L_2=\log\left(\binom{p-1}{d_2}\right)$. Lemma \ref{lemme_penalite1} ensures
that
\begin{eqnarray}\label{lemme_pena_minoration_2}
\EDKhi\left(d_1+1,n-d_1-1,e^{-L_1}\right)\geq
\EDKhi\left(d_2+1,n-d_2-1,e^{-L_1}\right)\ .
\end{eqnarray}
Let $x_1\geq x_2$ be two positive numbers larger than some integer $d_2+1$. By
Lemma \ref{lemme_penalite2}, it holds that 
\begin{eqnarray*}
\frac{\DKhi(d_2+1,n-d_2-1,x_1)}{\DKhi(d_2+1,n-d_2-1,x_2)}\geq
\frac{\mathbb{E}\left[(X_2-x_1)_+\right]}{\mathbb{E}\left[(X_2-x_2)_+\right]} =
e^{-(x_1-x_2)/2}\ .
\end{eqnarray*}
By Lemma \ref{lemme_minoration_penalite_seul}, $\EDKhi(d_2+1,n-d_2-1,e^{-L_2})$
is larger than $d_2+1$. Setting $x_1=\EDKhi(d_2+1,n-d_2-1,e^{-L_1})$  and
$x_2=\EDKhi(d_2+1,n-d_2-1,e^{-L_2})$, we obtain
\begin{equation}\label{lemme_pena_minoration_1}
\EDKhi(d_2+1,n-d_2-1,e^{-L_1})- \EDKhi(d_2+1,n-d_2-1,e^{-L_2})\geq 2(L_1-L_2)\ ,
\end{equation}
for $d_2\geq 1$. Gathering  the bounds (\ref{lemme_pena_minoration_2}),
(\ref{lemme_pena_minoration_1}) with the definition (\ref{definition_penalite})
of the penalty enables to conclude
\begin{equation*}
\pen(d_1)-\pen(d_2)\geq 2K(d_1-d_2) \log\left(\frac{p-d_1}{d_1}\right)\ .
\end{equation*}

\subsubsection{Proof of Lemma \ref{lemme_minoration_penalite_seul}}
We write henceforth $X_{d}$ and $X'_{N}$ for two independent $\chi^2$ 
variables with $d$ and $N$ degrees of freedom. By Jensen inequality, we get
\begin{eqnarray*}
d\times \DKhi(d,N,x)&=&\E\cro{\pa{X_{d}-x\, {X'_{N}\over N}}_{+}}\\
&\geq& \E\cro{\pa{X_{d}-x}_{+}}\geq \E\cro{\pa{X_{2}-x}_{+}}=2e^{-x/2}.
\end{eqnarray*}
for any $x>0$ and any $d\geq 2$.
Setting $x=\EDKhi(d,N,e^{-L})$ with $L\geq0$, we obtain
$$\EDKhi(d,N,e^{-L})\geq 2L-2\log(d),\quad \textrm{for } d\geq 2.$$
\begin{eqnarray*}
\EDKhi\left[d+1,n-d-1,\left[\binom{p-1}{d}(d+1)^2\right]^{-1}\right] & \geq 
&2\log\binom{p-1}{d}\ , 
\end{eqnarray*}
which is larger than $2d\log[(p-1)/(ed)]$. This allows to conclude.

\subsubsection{ Proof of Lemma \ref{lemme_penalite1}}
By definition (\ref{definition_penalite}) of the function $\EDKhi$, we only have
to prove that $\DKhi(d,N,x)$ is increasing with respect to $d$ and decreasing
with respect to $n$.\\

\noindent
Conditioning on $X_N$ (resp. $X_d$) it suffices to prove the two following
facts:

\noindent
FACT 1: Let $d$ be a positive integer. For any positive number $x$, 
\begin{eqnarray*}
d\mathbb{E}\left[(X_{d+1}-x)_+\right]\geq
(d+1)\mathbb{E}\left[(X_{d}-x)_+\right]\ .
\end{eqnarray*}
        
\noindent
FACT 2: Let $N$ be a positive integer. For any positive numbers $x$ and $x'$, 
\begin{eqnarray*}
\mathbb{E}\left[\left(x'-x\frac{X_N}{N}\right)_+\right]\geq
\mathbb{E}\left[\left(x'-x\frac{X_{N+1}}{N+1}\right)_+\right]\ .
\end{eqnarray*}

\noindent
{\bf Proof of FACT 1}. Let $(Z_1,\ldots ,Z_{d+1})$ be $d+1$ independent $\chi^2$
random variables with $1$ degree of freedom. Let $Y=\sum_{i=1}^{d+1}Z_i$ and for
any $i\in \{1,\ldots d+1\}$, let $Y^{(i)}$ be the sum $Y^{(i)}= \sum_{j\neq
i}Z_j$. The variable $Y$ follows a $\chi^2$ distribution with $d+1$ degrees of
freedom, while the variables $Y^{(i)}$ follow $\chi^2$ distribution with $d$
degrees of freedom. It holds that
\begin{eqnarray}\label{ineq_fact1}
 d\left(Y-x\right)_+\geq \sum_{i=1}^{d+1}\left(Y^{(i)}-x\right)_{+}\ .
\end{eqnarray}
Indeed, if all the variables $Y^{(i)}$ are larger than $x$, one observes that
$d\left(Y-x\right)_+= d(\sum_{i=1}^{d+1}Z_i-dx)$ while the second term equals
$d\sum_{i=1}^{d+1}Z_i-d(d+1)x$. If some of the variables $Y^{(i)}$ are smaller
than $x$, it is sufficient to note that the variables $Y^{(i)}$ are smaller than
$Y$. We prove FACT 1 by integrating the inequality (\ref{ineq_fact1}).\\

 \noindent
{\bf Proof of FACT 2}.
It is sufficient to prove that for any positive number $x$,
\begin{eqnarray*}
\mathbb{E}\left[\left(x-\frac{X_N}{N}\right)_+\right]\geq
\mathbb{E}\left[\left(x-\frac{X_{N+1}}{N+1}\right)_+\right]\ .
\end{eqnarray*}
Observe that $\mathbb{E}\left[\left(x-\frac{X_N}{N}\right)_+\right]= (x-1)+
\mathbb{E}\left[\left(\frac{X_N}{N}-x\right)_+\right]$. Hence, it remains to
prove that
\begin{eqnarray}\label{ineq_fact2}
 (N+1)\mathbb{E}\left[\left(X_N-Nx\right)_+\right]\geq
N\mathbb{E}\left[\left(X_{N+1}-(N+1)x\right)_+\right]\ .
\end{eqnarray}
As in the proof of FACT 1, let $(Z_1,\ldots ,Z_{d+1})$ be $d+1$ independent
$\chi^2$ random variables with $1$ degree of freedom. Let
$Y=\sum_{i=1}^{d+1}Z_i$ and for any $i\in \{1,\ldots d+1\}$, let $Y^{(i)}$ be
the sum $Y^{(i)}= \sum_{j\neq i}Z_i$. It holds that
\begin{eqnarray}\label{ineq2_fact2}
  \sum_{i=1}^{N+1}\left(Y^{(i)}-Nx\right)_{+}\geq N \left(Y-(N+1)x\right)_{+} \
.
\end{eqnarray}
This bound is trivial if $Y\leq (N+1)x$. If $Y$ is larger than $(N+1)x$, then
the second term equals $(N+1)\sum_{i=1}^{N+1}(Y^{(i)}- Nx)$, which is clearly
smaller than the first term. Integrating the bound (\ref{ineq2_fact2}) enables
to prove (\ref{ineq_fact2}) and then FACT 2.

\subsubsection{Proof of Lemma \ref{lemme_penalite2}}
We show that the derivate of the function ~\\
$\mathbb{E}\left[(X_d-x\frac{X_N}{N})_+\right]/\mathbb{E}\left[(X_2-x)_+\right]$
in non-negative for any $x\geq d$. Thus, we have to prove the following
inequality:
\begin{eqnarray*}
\frac{\mathbb{E}\left[\left(X_d- x
\frac{X_{N}}{N}\right)_+\right]}{\mathbb{E}\left[\frac{X_N}{N}\mathbf{1}_{
X_d\geq x\frac{X_{N}}{N}}\right]} \geq \frac{\mathbb{E}\left[\left(X_2- x
\right)_+\right]}{\mathbb{P}(X_2\geq x)}=2\ .
\end{eqnarray*}
Hence, we aim at proving that the function
\begin{eqnarray*}
\Psi(x)= \mathbb{E}\left[\left(X_d- x \frac{X_{N}}{N}\right)_+\right]- 2
\mathbb{E}\left[\frac{X_N}{N}\mathbf{1}_{X_d\geq x\frac{X_{N}}{N}}\right]
\end{eqnarray*}
is positive. Observe that $\Psi(x)$ converges to $0$ when $x$ goes to infinity.
Let us respectively note $f_{X_d}(t)$ and $f_{\frac{X_N}{N}}(t)$ the densities
of $X_d$ and $X_N/N$.
\begin{eqnarray*}
\Psi'(x)=\int_{t=0}^{\infty}t\left[2tf_{X_d}(xt)-\int_{u=xt}^{\infty}f_{X_d}
(u)du\right]f_{\frac{X_N}{N}}(t)dt\ .
\end{eqnarray*}
Integrating by part the density of a $\chi^2$ distribution, we get the lower
bound
\begin{eqnarray*}
 \int_{u=xt}^{\infty}f_{X_d}(u)du\geq
\frac{(1/2)^{d/2}}{\Gamma(d/2)}2(xt)^{d/2-1}e^{-xt/2}\ .
\end{eqnarray*}
Finally, we upper bound $\Psi'(x)$.
\begin{eqnarray*}
 \Psi'(x) &\leq
&\frac{(1/2)^{d/2-1}}{\Gamma(d/2)}\int_{t=0}^{\infty}t(xt)^{d/2-1}e^{-xt/2}
(t-1)f_{\frac{X_N}{N}}(t)dt\\
& \leq&
\frac{(1/2)^{(N+d)/2-1}}{\Gamma(d/2)\Gamma(N/2)}N^{N/2}x^{d/2-1}\int_{t=0}^{
\infty}t^{d/2}(t-1)t^{N/2-1}e^{-(x+N)t/2}dt\\
&\leq&
\frac{2N^{N/2}x^{d/2-1}}{\Gamma(d/2)\Gamma(N/2)(x+N)^{(d+N)/2}}\int_{t=0}^{
\infty}t^{(d+N)/2-1}\left(\frac{2t}{x+N}-1\right)e^{-t}dt\\
&\leq &\frac{2N^{N/2}x^{d/2-1}}{\Gamma(d/2)\Gamma(N/2)(x+N)^{(d+N)/2}}
\left[\frac{2\Gamma\left(\frac{d+N}{2}+1\right)}{x+N}-
\Gamma\left(\frac{d+N}{2}\right)\right]\\
&\leq &
\frac{2N^{N/2}x^{d/2-1}\Gamma\left(\frac{d+N}{2}\right)}{
\Gamma(d/2)\Gamma(N/2)(x+N)^{(d+N)/2}}\left[\frac{d+N}{x+N}-1\right]\leq 0\ ,
\end{eqnarray*}
since $x\geq d$. Hence, $\Psi$ is decreasing to $0$ for $x$ larger than $d$ and
it is therefore non-negative.

\section{Details for the family $\widehat{\G}$ of candidate graphs}\label{section_family_details}

\subsection{CO1 family $\widehat{\G}_{\CO 1}$}

The following construction of the family $\widehat\G_{01}$ derives
from the estimation procedure of Wille and B\"uhlmann~\cite{WB06}. We write $P(a,b|c)$ for the $p$-value of the
likelihood ratio test of the hypothesis "$R_{a,b|c}=0$" and set 
$$P_{\max}(a,b)=\max\left\{P(a,b|c),\ 
c\in\{\emptyset\}\cup\Gamma\setminus\{a,b\}\right\}\,.$$
For any $\alpha>0$, the graph $\widehat{G}_{01,\alpha}$ is defined by
$$a\stackrel{\widehat{G}_{01,\alpha}}{\sim}b\ \ \Longleftrightarrow \ \
P_{\max}(a,b)\leq \alpha$$
and the family $\widehat{\G}_{\CO1}$ is  the  family of nested graphs
$$ \widehat{\mathcal{G}}_{\CO 1}= \left\{\widehat{G}_{01,\alpha},\ \alpha>0
\text{ and } \degr(\widehat{G}_{01,\alpha})\leq D \right\}.$$

\begin{center}
\fbox{
\begin{minipage}{0.9\textwidth}
{\bf C01 Algorithm}
\begin{enumerate}
 \item Compute the $p(p-1)/2$ values $P_{\text{max}}(a,b)$.
\item Order them.
\item Extract from these values the  nested graphs $\ac{\widehat
G_{01,\alpha}:\alpha>0}$. 
\item Stop when the degree becomes larger than $D$.
\end{enumerate}
\end{minipage}
}
\end{center}

\subsection{Lasso-And family $\widehat{\G}_{\LA}$}

From a computational point of view, the  family $\widehat{\mathcal{G}}_{\LA}$
can be efficiently computed with the LARS-lasso algorithm. 
The optimization problem~(\ref{lasso_general}) is broken into the $p$
independent  minimization problems
\begin{equation}\label{lasso_particulier}
 \widehat{\theta}_a^{\lambda}=\arg\!\min\left\{\|{\X}_a-{\X}
v\|^2+\lambda\|v\|_1: \ v\in\R^p\textrm{ and } v_{a}=0\right\}, \ \text{for any
$a\in\Gamma$,}
\end{equation}
with $\|v\|_1=\sum_{b=1}^p|v_{b}|$. When $\lambda$ decreases, the support of
$\widehat \theta_{a}^{\lambda}$ is piecewise constant and the LARS-lasso
algorithm provides the sequences $(\lambda_a^l)_{l\geq 1}$ of the values of
$\lambda$ where the support of $\widehat\theta^\lambda_{a}$ changes, as well as
the sequence of the supports $\pa{\textrm{supp}(\theta^{\lambda_a^l})}_{l\geq
1}$.  Then, we gather these $p$ sequences as described
in the algorithm below.

 Given $\lambda>0$, we define the graph
$\widehat{G}^{\lambda}_{\text{and}}$ by
$$a\stackrel{\widehat{G}^{\lambda}_{\text{and}}}{\sim}b\ \ \Longleftrightarrow \
\ \widehat{\theta}_{a,b}^{\lambda}\neq 0 \textrm{ \underline{and} }
\widehat{\theta}_{b,a}^{\lambda}\neq0\,.$$ 

 Finally, we define the family
$\widehat{\mathcal{G}}_{\LA}$ as the set of graphs
$\widehat{G}^{\lambda}_{\text{and}}$ with $\lambda$ large enough to ensure that
$\degr(\widehat{G}^{\lambda}_{\text{and}})\leq D$, viz
\begin{eqnarray*}
\widehat{\mathcal{G}}_{\LA} = \left\{\widehat{G}^{\lambda}_{\text{and}}\ ,
\lambda > \widehat{\lambda}_{\text{and},D}\right\},
 &\text{ where}& \widehat{\lambda}_{\text{and},D}= \sup\left\{\lambda,\
\degr(\widehat{G}^{\lambda}_{\text{and}})>D\right\}.
\end{eqnarray*}

\begin{center}

\fbox{
\begin{minipage}{0.9\textwidth}
{\bf LA Algorithm}
\begin{enumerate}
\item Compute   with LARS-lasso the
$\pa{\lambda_a^l,\textrm{supp}(\widehat\theta^{\lambda_a^l})}_{l\geq 1}$ for all
$a\in\Gamma$.
\item Order the sequence $\ac{\lambda_a^l: a\in\Gamma,\ l\geq 1}$.
\item Compute $\widehat{G}^{\lambda^l_{a}}_{\text{and}}$ for all 
$\lambda^l_{a}>\widehat{\lambda}_{\text{and},D}$.
\end{enumerate}
\end{minipage}
}

\end{center}

\subsection{Adaptive lasso family $\widehat{\G}_{\EW}$} 

 To build the family $\widehat\G_{\EW}$ we start by computing the
 Exponential Weight estimator $\widehat\theta^{EW}$. For each $a\in\Gamma$, we
set $H_{a}=\ac{v\in\R^p: v_{a}=0}$ and 
 \begin{equation}\label{EW}
\widehat\theta^{EW}_{a}=\int_{H_{a}}v\, e^{-\beta\|\X_{a}-\X v\|^2_{n}}\,
\prod_{j}\pa{1+(v_{j}/\tau)^2}^{-\alpha}\, {dv \over \mathcal{Z}_{a}}\,,
\end{equation}
 with
$\mathcal{Z}_{a}=\int_{H_{a}}e^{-\beta\|\X_{a}-\X v\|^2_{n}}\,
\prod_{j}\pa{1+(v_{j}/\tau)^2}^{-\alpha}\, dv$  and  $\alpha,\beta,\tau>0$.
We note that $\widehat\theta^{EW}_{a}$ with $\beta=n/(2\sigma_{a}^2)$ and
$\sigma_{a}^2=\mathrm{var}(X_{a}\, |\,X_{-a})$ is simply  the Bayesian estimator
of $\theta_{a}$ with prior distribution $d\pi(v)\propto
\prod_{j}\pa{1+(v_{j}/\tau)^2}^{-\alpha}\, dv$ on $H_{a}$. 
In the Gaussian setting,  
Dalalyan and Tsybakov~\cite{DT08} give a sharp and assumption-free sparse
inequality for $\widehat\theta^{EW}_{a}$ with $\beta\leq n/(4\sigma_{a}^2)$, see
Corollary~4 in Dalalyan and Tsybakov. 

The construction of $\widehat\G_{\EW}$ is now similar to the construction of
$\widehat\G_{\LA}$. 
For any $\lambda>0$ we set
\begin{eqnarray}\label{adaptive_lasso_general}
 \widehat{\theta}^{EW,\lambda}=\arg\!\min\left\{\|{\X}-{\X}\theta'\|_{n\times
p}^2+\lambda
\|\theta'/ \widehat{\theta}^{EW}\|_1:\theta'\in\Theta\right\}\ ,
\end{eqnarray}
and we define the graph $\widehat{G}^{\EW,\lambda}_{\text{or}}$
by setting an edge between $a$ and $b$ if either
 $\widehat{\theta}_{b,a}^{EW,\lambda}$ \underline{or}
$\widehat{\theta}_{a,b}^{EW,\lambda}$ is non-zero:
 $$a\stackrel{\widehat{G}^{\EW,\lambda}_{\text{or}}}{\sim}b\ \
\Longleftrightarrow \ \ \widehat{\theta}_{a,b}^{\EW,\lambda}\neq 0 \ \textrm{
\underline{or} } \ \widehat{\theta}_{b,a}^{\EW,\lambda}\neq0\,.$$

 Finally, the family  $\widehat{\mathcal{G}}_{\EW}$ is given by
\begin{eqnarray*}
\widehat{\mathcal{G}}_{\EW} = \left\{\widehat{G}^{\EW,\lambda}_{\text{or}},\
\lambda > \widehat{\lambda}^{EW}_{\text{or},d}\right\}, 
&\text{where}& \widehat{\lambda}^{\EW}_{\text{or},D}= \sup\left\{\lambda,\
\degr(\widehat{G}^{\EW,\lambda}_{\text{or}})>D\right\}.
\end{eqnarray*}

The Exponential Weight estimator $\widehat\theta^{EW}$ can be computed with a
Langevin Monte-Carlo algorithm. We refer to \cite{DT09}
for the details. Once $\widehat\theta^{EW}$ is computed, the family
$\widehat{\mathcal{G}}_{\EW}$ is obtained as before with the help of the
LARS-lasso algorithm.

As for the family $\widehat{\mathcal{G}}_{\LA}$, the collection $\widehat{\mathcal{G}}_{\EW}$ is computed efficiently by breaking down the criterion (\ref{adaptive_lasso_general}) into $p$ independent minimization problems.
When $\lambda$ decreases, the support of
$\widehat \theta_{a}^{\EW,\lambda}$ is piecewise constant and the LARS-lasso
algorithm provides the sequences $(\lambda_a^{\EW,l})_{l\geq 1}$ of the values of
$\lambda$ where the support of $\widehat\theta^{\EW,\lambda}_{a}$ changes. Then, we gather these $p$ sequences as described
in the algorithm below.

\begin{center}

\fbox{
\begin{minipage}{0.94\textwidth}
{\bf EW Algorithm}
\begin{enumerate}
\item Compute $\widehat\theta^{EW}$ with a Langevin Monte-Carlo algorithm.
\item Compute   with LARS-lasso the
$\pa{\lambda_a^{\EW,l},\textrm{supp}(\widehat\theta^{\lambda_a^{\EW,l}})}_{l\geq 1}$ for all
$a\in\Gamma$.
\item Order the sequence $\ac{\lambda_a^{\EW,l}: a\in\Gamma,\ l\geq 1}$.
\item Compute $\widehat{G}^{\EW,\lambda^l_{a}}_{\text{or}}$ for all 
$\lambda^{\EW,l}_{a}>\widehat{\lambda}^{\EW}_{\text{or},D}$.
\end{enumerate}
\end{minipage}
}
\end{center}

\subsection{Quasi-exhaustive family $\widehat{\G}_{\QE}$}

\begin{center}
\fbox{
\begin{minipage}{0.9\textwidth}
{\bf QE Algorithm}
\begin{enumerate}
 \item  Compute $\widehat{\nei}(a)$  for all $a\in\Gamma$.
\item Compute the graphs $\widehat{G}_{K,\text{and}}$ and
$\widehat{G}_{K,\text{or}}$.
\item Work out the family $\widehat{\mathcal{G}}_{\QE}$.
\end{enumerate}
\end{minipage}
}
\end{center}


\bibliographystyle{plain}

\bibliography{estimation}

\end{document}